\documentclass[11pt, reqno]{amsart}
\usepackage[margin=2.75cm]{geometry}

\usepackage{amsthm,
amsmath,mathrsfs,enumitem,
amssymb,mathtools,mathrsfs,
dsfont,centernot,graphicx,
xcolor,hyperref,tikz,subcaption}
\usetikzlibrary{decorations.pathreplacing,calligraphy}

\usepackage[english]{babel}
\usepackage[hypcap]{caption}

\theoremstyle{plain}
\newtheorem{theorem}{Theorem}[section]
\newtheorem{lemma}[theorem]{Lemma}
\newtheorem{proposition}[theorem]{Proposition}
\newtheorem{corollary}[theorem]{Corollary}

\theoremstyle{definition}
\newtheorem{definition}[theorem]{Definition}
\newtheorem{question}[theorem]{Question}
\newtheorem{example}[theorem]{Example}
\newtheorem{assumption}[theorem]{Assumption}
\newtheorem{remark}[theorem]{Remark}
\numberwithin{equation}{section}

\DeclareFontFamily{U}{mathx}{\hyphenchar\font45}
\DeclareFontShape{U}{mathx}{m}{n}{
      <5> <6> <7> <8> <9> <10>
      <10.95> <12> <14.4> <17.28> <20.74> <24.88>
      mathx10
      }{}
\DeclareSymbolFont{mathx}{U}{mathx}{m}{n}
\DeclareFontSubstitution{U}{mathx}{m}{n}
\DeclareMathAccent{\widecheck}{0}{mathx}{"71}
\DeclareMathAccent{\wideparen}{0}{mathx}{"75}

\setlist{nosep}
\newcommand{\Law}{\mathcal{L}}
\newcommand{\Prob}{\mathbb{P}}
\newcommand{\Exp}{\mathbb{E}}
\newcommand{\Space}{\mathcal{S}}
\newcommand{\indf}{\mathds{1}}

\newcommand{\F}{\mathcal{F}}
\newcommand{\calG}{\mathcal{G}}

\renewcommand{\k}{\kappa}
\renewcommand{\b}{\beta}

\newcommand{\calT}{\mathcal{T}}

\newcommand{\X}{\mathcal{X}}
\newcommand{\Y}{\mathcal{Y}}

\newcommand{\bbN}{\mathbb{N}}

\newcommand{\bbV}{\mathbb{V}}

\newcommand{\B}{\mathcal{B}}
\newcommand{\N}{\mathbb{N}}
\newcommand{\R}{\mathbb{R}}
\newcommand{\calI}{\pmb{\mathscr{I}}}
\newcommand{\calboldI}{\pmb{\mathcal{I}}}
\newcommand{\frakI}{\pmb{\mathfrak{I}}}
\newcommand{\calItilde}{\widetilde{\pmb{\mathscr{I}}}}

\renewcommand{\P}{\mathcal{P}}
\newcommand{\ER}{\mathsf{ER}}
\newcommand{\FE}{\mathsf{FE}}
\newcommand{\CM}{\mathsf{CM}}

\newcommand{\NZ}{\bbN_0}

\newcommand{\GNM}{\calG_{n, m_n}}
\newcommand{\GER}{\calG(n, \k/n)}
\newcommand{\GCM}{\calG_n(\alpha_n)}

\newcommand{\dego}{\text{deg}_{\TB}(o)}

\newcommand{\Conn}{\mathsf{C}}
\newcommand{\cl}{\mathsf{cl}}
\newcommand{\Gstar}{\mathcal{G}_{*}}
\newcommand{\Gdstar}{\mathcal{G}_{**}}
\newcommand{\Tdstar}{\mathcal{T}_{**}}

\newcommand{\EDRinf}{\Exp_{\rho}[\dego]}
\renewcommand{\r}[1]{\rho_{#1}}

\newcommand{\SBR}[1]{\bar{\rho}_{#1}}
\newcommand{\SBeta}[1]{\bar{\eta}_{#1}}
\newcommand{\SBRstar}[1]{\bar{\rho}^*_{#1}}
\newcommand{\PIR}[1]{\pi_{\rho_{#1}}}
\newcommand{\PIRstar}[1]{\pi_{\rho^*_{#1}}}
\newcommand{\PIeta}{\pi_{\eta_1}}
\newcommand{\PIetaalp}{\pi_{\etaa_1}}
\newcommand{\PIetab}[1]{\pi_{\eta_1^\text{Poi}({#1})}}
\newcommand{\etaa}{\eta^\alpha}
\newcommand{\etab}[1]{\eta^{\text{Poi}({#1})}}

\newcommand{\SBRM}[1]{\SBR{#1}^{1,o}}
\newcommand{\SBRMZero}[1]{\SBR{#1}^{o}}
\newcommand{\SBRMOne}[1]{\SBR{#1}^{1}}
\newcommand{\SBRMOneGivenZero}[1]{\SBR{#1}^{1|o}}
\newcommand{\SBRstarM}[1]{(\SBRstar{#1})^{1,o}}
\newcommand{\SBRstarMOne}[1]{(\SBRstar{#1})^{1}}
\newcommand{\SBRstarRM}[1]{(\SBRstar{#1})^{1}}

\newcommand{\SBetaM}{\SBetabarM}

\newcommand{\SBetaMI}{\SBetabarM}

\newcommand{\sqpr}[1]{\left[{#1}\right]}
\newcommand{\cpr}[1]{\left({#1}\right)}
\newcommand{\pr}[1]{\left\{{#1}\right\}}

\newcommand{\T}[1]{\calT_{*,{#1}}[\X ; \Y]}
\newcommand{\G}[1]{\calG_{*,{#1}}[\X ; \Y]}

\newcommand{\NO}{\dego}

\newcommand{\PIMUT}{\tilde{\pi}_{\mu}}

\newcommand{\SBRRMh}[1]{\bar{\rho}_{#1}^1}
\newcommand{\SBRCRMh}[1]{\bar{\rho}_{#1}^{1|o}}
\newcommand{\SBMuMZero}{\bar{\mu}^o}
\newcommand{\SBMuM}{\bar{\mu}^{1,o}}
\newcommand{\SBMuRM}{\bar{\mu}^{1}}

\newcommand{\SBMuCRM}{\bar{\mu}^{1|o}}
\newcommand{\SBetabarCRM}{\bar{\eta}_1^{1|o}}
\newcommand{\SBetabarRM}{\bar{\eta}_1^{1}}

\newcommand{\SBetabarM}{\bar{\eta}_1^{1,o}}
\newcommand{\SBetabarMOne}{\bar{\eta}_1^{1}}
\newcommand{\SBetabarMZero}{\bar{\eta}_1^{o}}
\newcommand{\SBetabarMOneGivenZero}{\bar{\eta}_1^{1|o}}

\newcommand{\Tinf}{\calT_{*}[\X; \Y]}
\newcommand{\Ginf}{\calG_{*}[\X; \Y]}

\newcommand{\SBRinf}{\bar{\rho}}

\newcommand{\RP}{\widecheck{\rho}}
\newcommand{\RCP}{\widehat{\rho}}

\newcommand{\MP}{\widecheck{\mu}}
\newcommand{\MCP}{\widehat{\mu}}
\newcommand{\MPdef}{\widecheck{\mu}[\eta_1]}
\newcommand{\MCPdef}{\widehat{\mu}[\eta_1]}

\newcommand{\PUK}{\P_u^\k(\Tinf)}
\newcommand{\PK}{\P^\k(\Tinf)}
\newcommand{\PB}{\P^\beta(\Tinf)}
\newcommand{\PUB}{\P_u^\beta(\Tinf)}

\newcommand{\PU}[1]{\widehat{P}_{#1, \r{#1}}}
\newcommand{\UGWR}[1]{\mathsf{UGWT}(\rho_{#1})}
\newcommand{\Rstar}[1]{\rho^*_{#1}}

\newcommand{\EDLDRinf}{\Exp_{\rho} \sqpr{\dego \log (\dego)}}
\newcommand{\EDLDR}[1]{\Exp_{\rho_{#1}}\sqpr{\dego \log (\dego)}}
\newcommand{\PIRM}[1]{\pi_{\rho_{#1}}^1}
\newcommand{\PIRstarM}[1]{\pi_{\rho^*_{#1}}^1}
\newcommand{\Geom}{\mathsf{Geom}_{1/2}}

\newcommand{\GNA}{\calG_n(\alpha_n)}
\newcommand{\kvr}{\vec{\k}_{\rho}}
\newcommand{\bvr}{\vec{\b}_{\rho}}

\newcommand{\Gnm}{\calG_{n, m_n}}

\newcommand{\GB}{\mathbf{G}}

\newcommand{\TB}{{ \boldsymbol{\scalebox{1.2}[1.2]{$\tau$}}}}
\newcommand{\TXY}{\TB}
\newcommand{\TBdef}{\tilde{\tau}}
\newcommand{\TX}{\calT_{*,1}[\X]}
\newcommand{\SB}{\boldsymbol{\sigma}}

\newcommand{\TOV}{\tau(v \setminus o)}
\newcommand{\TVO}{\tau(o\setminus v)}

\newcommand{\TBOV}{\TB(v \setminus o)}
\newcommand{\TBVO}{\TB(o \setminus v)}

\newcommand{\TBdefOV}{\TBdef(v \setminus o)}
\newcommand{\TBdefVO}{\TBdef(o \setminus v)}

\newcommand{\TBOOne}{\TB(1 \setminus o)}
\newcommand{\TBOneO}{\TB(o \setminus 1)}

\newcommand{\TOU}{\tau(u \setminus o)}
\newcommand{\TUO}{\tau(o \setminus u)}

\newcommand{\TBUO}{\TB(o \setminus u)}

\newcommand{\bfX}{{\bf X}}
\newcommand{\bfY}{{\bf Y}}
\newcommand{\TXM}{\calT_{*,1}^M[\X]}

\newcommand{\lambdavec}{\vec{\lambda}}
\newcommand{\Erdos}{Erd\H{o}s-R\'enyi }

\newcommand\numberthis{\addtocounter{equation}{1}\tag{\theequation}}

\title[Large deviations of marked random graphs]{On the large deviation rate function for marked sparse random graphs}\thanks{The authors were supported by the Office of Naval Research under the Vannevar Bush Faculty Fellowship
N0014-21-1-2887. The first author was also supported by the National Science Foundation under grant
DMS-2246838.}
\thanks{Date: 23 December 2023}
\author{Kavita Ramanan}
\author{Sarath Yasodharan}
\address{Division of Applied Mathematics, Brown University, Providence, RI 02912, USA}
\email{kavita\_ramanan@brown.edu, sarath\_yasodharan@brown.edu}
\begin{document}
\maketitle
\begin{abstract}
We consider (annealed) large deviation principles for component empirical measures of several families of marked sparse random graphs, including (i) uniform graphs on $n$ vertices with a fixed degree distribution; (ii) uniform graphs on $n$ vertices with a fixed number of edges; (iii) Erd\H{o}s-R\'enyi $G(n, c/n)$ random graphs. Assuming that edge and vertex marks are independent, identically distributed, and take values in a finite state space, we show that the large deviation rate function admits a concise representation as a sum of relative entropies that quantify the cost of deviation of a probability measure on marked rooted graphs from certain auxiliary independent and conditionally independent versions. The proof exploits unimodularity, the consequent mass transport principle, and random tree labelings to express certain combinatorial quantities as expectations with respect to size-biased distributions, and to identify unimodular extensions with suitable conditional laws. We also illustrate how this representation can be used to establish Gibbs conditioning principles that provide insight into the structure of marked random graphs conditioned on a rare event. Additional motivation for this work arises from the fact that such a representation is also useful for characterizing the annealed pressure of statistical physics models with general spins, and large deviations of evolving interacting particle systems on sparse random graphs.

\vspace{10pt}

\end{abstract}

\section{Introduction}
\subsection{Context and description of results} 
Large deviation principles (LDPs) characterize the exponential decay rate of a sequence of probabilities of rare events in a Polish space ${\mathcal S}$ in terms of an optimization problem involving an associated function from ${\mathcal S}$ to  $[0,\infty]$ called the rate function. An LDP with a tractable rate function facilitates the computation of rare event probabilities and the study of limit laws conditioned on a rare event. Such conditional limit laws, often referred to as  Gibbs conditioning principles, provide insight into the most likely way in which a rare event happens. The focus of this work is on  LDPs for the neighborhood (and more general component) empirical measures of sequences of sparse random graphs with independent and identically distributed (i.i.d.) vertex and edge marks taking values in a finite state space. We consider three classes of random graph sequences: the  $\CM$ sequence comprised of graphs generated by the configuration model (i.e., uniform random graphs with a given empirical degree distribution), which includes random regular graphs as a special case, the $\FE$ sequence comprised of uniform random graphs with a fixed number of edges, and the $\ER$ sequence consisting of sparse \Erdos $G(n,c/n)$ graphs (see Section \ref{sec:model-iid} for definitions). Our main result shows that in each case, the rate function of the LDP for the neighborhood and component empirical measures takes a similar succinct form. Such a representation has many applications, as elaborated in Section \ref{subs-ram}.

 \begin{figure}
    \centering
    \begin{center} 
\begin{tikzpicture}[scale = 0.8]
\filldraw[black] (0,0) circle(3pt);
\filldraw[black] (-1,-1) circle(3pt);
\filldraw[black] (0,-1) circle(3pt);
\filldraw[black] (1,-1) circle(3pt);
\node at (0,0.3){$o$};
\node at (0.4,0){$x_o$};
\node at (-1,-1.4){$x_1$};
\node at (0,-1.4){$x_2$};
\node at (1,-1.4){$x_3$};
\draw[black, thick] (0,0) -- (-1,-1);
\draw[black, thick] (0,0) -- (0,-1);
\draw[black, thick] (0,0) -- (1,-1);
\end{tikzpicture}
\end{center}
    \caption{An illustration of a marked $3$-star rooted at $o$. The root mark is denoted by $x_o$ and the marks on the neighbors of the root are denoted by $x_1, x_2$, and $x_3$.}
    \label{fig:marked-star}
\end{figure}
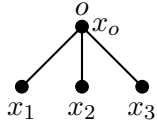

To describe our results, fix $\kappa \geq 2$ and let $G_n$ be the random $\kappa$-regular graph on $n$ vertices equipped with i.i.d.\,vertex marks $\{X_v\}_{v \in G_n}$  taking values in a finite set ${\mathcal S}$ with common marginal $\nu$. Let $L_n$ represent the (random) marked neighborhood empirical measure, defined  by
\[  L_n := \frac{1}{n} \sum_{v \in G_n} \delta_{(X_v, X_{\partial v})}, \]
where $X_{\partial v} := \{X_u, u \sim v \},$ and  $u \sim v$ indicates that  $u$ and $v$ are neighbors in $G_n$. As is well known by Sanov's theorem \cite{San57}, the sequence of empirical measures of just the vertex marks $\frac{1}{n} \sum_{v \in G_n} \delta_{X_v}$, $n \in \N$, satisfies an LDP with rate function equal to $H(\cdot\|\nu)$, the relative entropy functional with respect to $\nu$. The challenge in proving an LDP for the sequence $\{L_n\}$ arises from the dependencies between $X_{\partial v}, v \in V,$ caused by overlaps in vertex neighborhoods, which depend on the random graph topology. In a seminal work \cite{BorCap15}, inspired by the well-known configuration model used to generate random graphs from a given degree sequence, Bordenave and Caputo introduced a novel colored configuration model and used combinatorial arguments to establish LDPs for the component empirical measure of {\em unmarked} $\CM$, $\FE$ and $\ER$ random graph sequences. Their approach was adapted to prove LDPs in the more complicated setting of {\em marked} graphs (with discrete i.i.d.\,marks) for the $\FE$ and $\ER$ random graph sequences in \cite{DelAna19,BalOlietal22}, and subsequently for a broad class of uniform marked random graphs, including the  $\CM$, bipartite and stochastic block model, in \cite{CheRamYas23}. For random $\kappa$-regular graphs, it is shown in \cite{CheRamYas23} that the sequence  $\{L_n\}$ satisfies an LDP on ${\mathcal S} \times {\mathcal S}^{\kappa}$, representing the space of marks on $\kappa$-stars (see Figure \ref{fig:marked-star}), with a rate function ${\mathcal J}$ that takes the following form for probability measures $\mu$ on ${\mathcal S} \times {\mathcal S}^{\kappa}$ satisfying certain symmetry conditions:
\begin{equation}
{\mathcal J}(\mu) =  \bar{s}(\kappa)  +  H(\mu^o \| \nu) +  H(\mu^o) - H(\mu)  + \frac{\kappa}{2} H(\mu^{1,o}) + 
\sum_{x, x' \in {\mathcal S}} \Exp_{\mu} [ \log \cpr{\sqpr{ E_1(x, x')(\TB)}!}].  \label{eqn:rate-fn-illusrate}
\end{equation}
Here $\bar{s}(\kappa) := \kappa (\log \kappa - 1)/2 - \log ( \kappa!)$, $\mu^o$ represents the root mark marginal, $\mu^{1,o}$ represents the marginal on a (uniformly chosen) neighbor of the root and the root, $H$  represents the entropy functional, $H(\cdot\|\nu)$ is the relative entropy functional mentioned above, and $E_1(x,x')(\TB)$ is a certain combinatorial quantity that is a function of the marked $\kappa$-star $\TB$. It is worth emphasizing that the form of the rate function is more complicated in the marked setting. Indeed, the rate function for the neighborhood empirical measure on {\em unmarked} graphs is rather simple:  for $\kappa$-regular graphs, it is degenerate (taking the value zero at the point mass on $\kappa$ and infinity otherwise) and for other graph ensembles, it is equal to relative entropy with respect to the true root degree distribution, with appropriate constraints. The expressions for the rate functions for marked $\CM$, $\FE$ and $\ER$ models obtained in \cite{CheRamYas23} and \cite[Theorems 4.1 and 4.2]{BalOlietal22} have a similar complexity as that in \eqref{eqn:rate-fn-illusrate} for $\kappa$-regular graphs, though with an extra term that accounts for randomness in the root degree distribution, and the precise form depending on the graph sequence. In the presence of edge marks and when considering the full component empirical measure, further additional combinatorial terms arise in the rate functions of \cite{BalOlietal22,CheRamYas23}. Such descriptions make it challenging to analyze optimization problems involving the rate function to compute or gain insight into probabilities of rare events. 

In this article, we show that in fact the rate function admits a more tractable expression, involving only relative entropies. As a special case of our main result (Theorem \ref{thm:iid-nbd}) we show  that the rate function ${\mathcal J}$ in  \eqref{eqn:rate-fn-illusrate} associated with the $\kappa$-regular random graph with i.i.d.\,vertex marks takes the following simple form: for  $\mu \in {\mathcal P}({\mathcal S} \times {\mathcal S}^{\kappa})$ such that both the second marginal of $\mu$ and the marginal $\mu^{1,o}$ are symmetric, we show 
\begin{equation}
\label{eqn:rate-fn-simple2}
{\mathcal J} (\mu) =   \frac{1}{2} [ H (\mu \| \MP)  +  H (\mu \| \MCP) ],
\end{equation}
where  $\MP (x_o, (x_u, u \sim o))  = \nu(x_o) \prod_{u \sim o} \mu^{1} (x_u)$ and $\MCP (x_o, (x_u, u \sim o))  = \nu(x_o) \prod_{u \sim o} \mu^{1|o} (x_u|x_o)$. The form in \eqref{eqn:rate-fn-simple2} provides an intuitive interpretation of the rate function as the cost of $\mu$ deviating from its product and conditionally independent versions. Since the vertex mark empirical measure $\frac{1}{n} \sum_{v \in G_n} \delta_{X_v}$ is a projection of the neighborhood empirical measure $L_n$ onto the root mark, one should be able to apply the contraction principle to immediately recover Sanov's theorem from the LDP for the neighborhood empirical measure.  While this is indeed the case when the rate function is of the form \eqref{eqn:rate-fn-simple2}, it is not as apparent given the form in \eqref{eqn:rate-fn-illusrate}. Further, the expression for the rate function remains the same as in \eqref{eqn:rate-fn-simple2} in the presence of edge marks, and also in the more general setting of $\CM$, $\FE$ or $\ER$ graph sequences, though with $\MP$ and $\MCP$ suitably redefined, and an additional relative entropy term present in the $\ER$ setting to account for deviations of the mean root degree from that under the true law. In the absence of edge marks, the form \eqref{eqn:rate-fn-simple2} can be shown to take an even simpler form for all graph sequences (see Corollary \ref{cor-kapparf}).

Our results also extend to the component empirical measure. In this case, the corresponding rate function is represented as a countable sum, with each summand capturing deviations of the depth $h$-neighborhood of the root from the true law via a form analogous to that in \eqref{eqn:rate-fn-simple2}, although the definitions of the analogs of $\MP$ and $\MCP$ are now more involved, entailing unimodular extensions of marked graphs (see definitions \eqref{eqn:RPh-def}-\eqref{eqn:RCPh} and Theorem \ref{thm:iid}). Nevertheless, the rate function has a common form for all three graph ensembles, providing a Sanov-like theorem for unimodular sparse graphs, with the quantity in \eqref{eqn:rate-fn-simple2} representing a fundamental {\em divergence functional} on the space of probability measures on rooted graphs with respect to a ``true law'' of the graph with i.i.d.\,vertex and edge marks.

\subsection{Comments on the proof}

The proof of our main result is carried out in two steps. In the first step, we establish an intermediate form of the rate function in terms of {\em differences of relative entropies}; see \eqref{eqn:calItilde-generic} and Theorem \ref{thm:iid-old}. An analogous representation holds for unmarked graphs \cite[Remark 5.13]{BorCap15}, but the proof in the marked case is considerably more difficult because, as mentioned above, the corresponding form \eqref{eqn:rate-fn-illusrate} of the rate function has many additional terms that are not present in the unmarked case. This entails establishing representations of combinatorial quantities, similar to those appearing in \eqref{eqn:rate-fn-illusrate},  
in terms of relative entropies with respect to the true law. While this intermediate form has a useful counting-based interpretation, and is also much more tractable than the combinatorial form in \eqref{eqn:rate-fn-illusrate}, 
it involves relative entropies of different signs, and even its non-negativity is not immediate.

The next step, which involves showing that the intermediate form of the rate function coincides with the expression in \eqref{eqn:rate-fn-simple2}, is considerably more involved. First, we show that the auxiliary quantities $\MCP$ and $\MP$ for the component empirical measure are indeed probability measures. 
In the case of vertex-marks they are specified explicitly via the definition below \eqref{eqn:rate-fn-simple2}, and so this is immediate. However, in the presence of edge-marks and for component measures, they are only defined by specifying their densities with respect to another measure, and hence it requires work to show that they are indeed probability measures. Next, we identify probability measures on equivalence classes of rooted trees with corresponding probability measures on randomly {\em labeled} rooted trees. We then exploit resultant exchangeability properties to show that certain covariance measures coincide with suitable sized-biased marginals (see \cite[Section 4.2]{BacBorSze22}) and to identify unimodular extensions of certain probability measures arising in the definition of the rate function with associated conditional laws (see Remark \ref{rmk:unimod-full-extension}(1)). Throughout, we crucially exploit the {\em unimodularity} property of marked random graphs and the associated mass transport principle (see Section \ref{sec:alternate-form}). In the course of the proof, we also arrive at an alternative characterization of a combinatorial quantity called the microstate entropy that arises in \cite{BorCap15,BacBorSze22}, which may be of independent interest (see Corollary \ref{cor:sigma-alternative}).
\subsection{Motivation and applications} 
\label{subs-ram}
There are multiple motivations for obtaining a simple, easily interpretable form of the rate function. We discuss three of these below:

\noindent {\em A. Gibbs conditioning principles:} Consider the following question: what is the distribution of the random graph sequence and marks, when conditioned on the event that the fraction of vertices for which the sum of marks on neighboring vertices is much larger than its mean? In Proposition \ref{prop-gcp}, we illustrate how our new form of the rate function can be exploited to provide a rather explicit answer to this question for the $\CM$ graph sequence. In many discrete-time locally interacting particle systems on a sparse graph, transition probabilities of the process are governed by the sum (or some other functional) of neighboring marks. Thus, conditional limit laws of this type would provide insight into how rare events occur in such interacting particle systems. More examples of such  Gibbs conditioning principles will be elucidated in forthcoming work.  \\

\noindent {\em B. Variational formulas for characterizations of the annealed pressure:} In a companion work \cite{annealed-future-paper}, the neighborhood empirical measure LDP for graphs with more general (possibly continuous) i.i.d.\,marks is shown to hold in a stronger topology. The latter form is then exploited to obtain variational formulas and characterize the annealed pressure of a  class of statistical physics models on sparse random graphs with general (not necessarily discrete) spins.  While such variational formulas are easily derived in the discrete case using elementary combinatorics, consideration of continuous spin models naturally entails a large deviations analysis of marked empirical measures.  \\

\noindent {\em C. Extension to Polish mark spaces:}  It is natural to ask if one can also extend the LDP for the component empirical measure to the case of i.i.d.\,marks on general Polish spaces. The old form of the rate function in \eqref{eqn:rate-fn-illusrate} admits an interpretation only for discrete marks. An approximation argument was used in \cite[Theorems 4.3, 4.4]{BalOlietal22} to establish the LDP for continuous marks, yielding a non-explicit rate function, expressed in terms of limits. In contrast, the definition of the rate function in \eqref{eqn:rate-fn-simple2} automatically extends to general Polish space marks. It is natural to conjecture that this extension in fact coincides with the rate function of the LDP for the component measure with general marks. While (as mentioned above) this conjecture can be verified for the neighborhood empirical measure (see \cite{annealed-future-paper}), the case of the component empirical measure remains open, but the form of the rate function in \eqref{eqn:rate-fn-simple2} potentially allows for a different proof approach. Such a result would have important implications for understanding large deviations and conditional limit laws for a broad class of evolving interacting particle systems for which hydrodynamic limits have been established in \cite{OliReiSto20,LacRamWu20,LacRamWu23,GanRam22-Hydro,CocRam23,Ram22-ICM} (e.g., see the application to interacting gradient diffusions in \cite[Theorem 4.9]{BalOlietal22}).            \\

\subsection{Related work}
Finally, we mention some related literature on LDPs for sparse random graphs. In \cite{DokMor10,IbrDok21}, combinatorial arguments were used to establish an  LDP for the joint vertex-edge empirical measure in marked \Erdos and random regular graphs with i.i.d.\,marks in a discrete space. This LDP was extended in \cite{LiaRam24} to general marks by approximation. The corresponding rate functions are in terms of relative entropies, but they involve certain ``sub-consistency'' conditions. None of these works consider the component empirical measure. In \cite{DemMorShe05}, an   LDP for the (vertex, children) empirical measure of a “tree-indexed Markov chain'' was established. Several other LDPs have been established for {\em unmarked} sparse random graphs that have completely different motivations. For example, the works of \cite{Oco98,Puh05,ChaHof21,GanHieNam22,AndKonLanPat23} are motivated by combinatorial questions concerning the prevalence of various kinds of subgraphs and component sizes in sparse random graphs. These complement earlier work addressing analogous questions for dense graphs such as \cite{ChaVar11}, which uses the theory of graphons, and large deviations for sparse graphs (but not uniformly sparse, as considered here) using the theory of nonlinear large deviations \cite{ChaDem16}. The speed of the LDP in \cite{Oco98,Puh05,ChaHof21,GanHieNam22,AndKonLanPat23} depends on the structure and size of the subgraph of interest. In contrast, our LDPs are always at speed $n$ and the focus in this article is on {\em marked} random graphs, which (as mentioned above) is required for understanding large deviations for interacting particle systems on sparse random graphs. The latter theory is inchoate in comparison to the well developed theory in the ``mean-field'' setting (of interacting particle systems on the complete graph).  

\subsection{Organization}

The rest of this paper is organized as follows. In Section \ref{sec:notation}, we recall some preliminaries on rooted marked graphs and probability measures on them, and introduce some notation. We define the model and quantities of interest and state our main results on LDPs for neighborhood and component empirical measures in Sections \ref{sec:results-nbd} and \ref{sec:results-com}, respectively. Although the neighborhood empirical measure LDPs follow from the more general component measure LDPs, we state them separately for ease of readability and to defer the substantial additional notation required to state the full component measure result. Section \ref{sec:gibbs} contains the statement of a conditional limit theorem for the neighborhood empirical measure for the $\CM$ random graph sequence; its proof is relegated to Section \ref{sec:gibbs-proof}. The proof of the intermediate representation is given in Section \ref{sec:iid-proof-complete}. The proof of the neighborhood and component empirical measure LDPs are given in Sections \ref{sec:alternate-form-1} and \ref{sec:alternate-form-h}, with several technical results relegated to Appendices \ref{appendix:lemma-finiteness}-\ref{appendix:lemma-Rstarabscont}.

\section{Preliminaries and notation}
\label{sec:notation}
\subsection{Generic notation}
Let $\bbN$ denote the set of natural numbers, let $\R$ denote the set of real numbers, and define $\NZ := \bbN \cup \{0\}$.  For $x \in \R$, we write $x^+ = \max\{x, 0\}$ and $x^- = -\min\{x, 0\}$. For a finite set $A$, let $|A|$ denote its cardinality. Let $\indf\{\cdot\}$ denote the indicator function. Given a Polish space  $(\Space, d)$, let $\B(\Space)$ denote its Borel $\sigma$-field. Let $\P(\Space)$ denote the space of probability measures on $(\Space, \B(\Space))$ equipped with the  L\'evy-Prohorov metric (which we denote by $d_{\P(\Space)}$), which generates the weak topology on $\P(\Space)$. If $\nu \in \P(\Space)$ and $f$ is a $\nu$-integrable function on $\Space$, let $\langle \nu ,  f \rangle $ denote $\int_{\Space} f d\nu$.  If $\Space$ is either a  finite or countable set and  $\nu \in \P(\Space)$, we denote $H(\nu)$ to denote the Shannon entropy of $\nu$:
\begin{align}
H(\nu) = -\sum_{x \in \Space} \nu(x) \log \cpr{\nu(x)},
\label{eqn:def-entropy}
\end{align}
with the convention $0 \log 0 = 0$. If $\Space$ is  a Polish space and $\mu, \nu \in \P(\Space)$, the relative entropy of $\mu$ with respect to $\nu$ is given by
\begin{align}
H(\mu \| \nu) = \begin{dcases}
\int_\Space \log \cpr{\frac{d\mu}{d\nu}} \, d\mu & \text{if } \mu \ll \nu,\\
\infty & \text{otherwise}.
\end{dcases}
\label{eqn:def-renelt}
\end{align}

Let $(\Omega, \F, \Prob)$ be a complete probability space. The expectation with respect to $\Prob$ is denoted by $\Exp$. This probability space is assumed to be rich enough to support all random elements defined in this paper. For a random variable $\bfX$, let $\Law(\bfX)$ denote its law. Also, we adopt the convention that the sum over the empty set is $0$ and the product over the empty set is $1$.

\subsection{The space of rooted marked graphs}
An (undirected) graph is denoted by $G = (V, E)$, where $V$ is the vertex set and $E$ is the edge set. For a vertex $v \in V$, let $N_v(G) := \{u \in V: \{u, v\} \in E\}$ denote the set of all neighbors of $v$ in $G$. The degree of vertex $v$ is the cardinality of $N_v(G)$.   The vertex set $V$ is either a finite set or a countable set, and all graphs considered in this paper are locally finite; that is, each vertex has a finite degree. For $u, v \in V$, the distance between $u$ and $v$ is the length of the shortest path connecting $u$ and $v$. Two graphs $G_1 = (V_1, E_1)$ and $G_2 = (V_2, E_2)$ are said to be isomorphic (denoted by $G_1 \simeq G_2$) if there exists a one-to-one map $\varphi: V_1 \to V_2$ such that $\{u, v\} \in E_1$ if and only if $\{\varphi(u), \varphi(v)\} \in E_2$.

A connected graph $G$ with a distinguished vertex $o$ is denoted by $(G, o)$. Such a graph is called a {\em rooted graph}, and the vertex $o$ is called its root. The rooted graph  $(G_1, o_1)$ is said to be isomorphic to the rooted graph $(G_2, o_2)$ if there exists a one-to-one map $\varphi : V_1 \to V_2$ such that $\varphi(o_1) = o_2$ and $\{u, v\} \in E_1$ if and only if $\{\varphi(u), \varphi(v)\} \in E_2$. Here, $\varphi$ is called an isomorphism from $(G_1, o_1)$ to $(G_2, o_2)$, and this equivalence relation is denoted by $(G_1, o_1) \simeq (G_2, o_2)$.   Let $\Gstar$ denote the set of equivalence classes of rooted graphs under the above equivalence relation. Alternatively, we view $\Gstar$ as the space of all {\em unlabeled} rooted graphs. For $(G, o) \in \Gstar$, let $(G, o)_h$,  $h \in \NZ$, denote the subgraph of $G$ rooted at $o$ containing all vertices at a distance of at most $h$ from the root. 
We equip $\Gstar$ with the {\em local topology}, which makes $\Gstar$ a Polish space (see \cite[Section 2]{AldLyo07} or \cite[Section 3.2]{Bor16} for precise definitions).

We now define corresponding marked graphs. Let $\X$ and $\Y$ be nonempty finite sets. Let $\Ginf$ denote the set of unlabeled rooted {\em marked} graphs,  where each vertex is assigned an $\X$-valued mark and each directed edge is assigned a  $\Y$-valued mark. An element of $\Ginf$ is denoted by $(G, o, (x, y))$, where $o$ is the root,  $\{x_v\}_{v \in G}$ are the vertex marks, and $\{(y_{(u,v)}, y_{(v,u)})\}_{\{u,v\} \in E}$ are the  edge marks. Here, $y_{(u,v)}$ (resp. $y_{(v,u)}$) is the mark on the edge $\{u,v\}$ associated with the vertex $u$ (resp. $v$); hence, each undirected edge gets a $\Y \times \Y$-valued mark. We shall also denote a typical element of $\Ginf$ by $(G, (x, y))$, or simply $G$. For $h \in \NZ$, let $(G, o, (x, y))_h$ denote the marked rooted subgraph of $(G, o, (x, y))$ containing all vertices at a distance of at most $h$ from the root $o$.  As in the unmarked case, $\Ginf$ can be equipped with the local topology, and made a Polish space. Also, for $h \in \NZ$, let $\G{h} \subset \Ginf$ denote the space of rooted marked graphs with depth at most $h$ (i.e., every vertex is at most at a distance $h$ from the root), and we equip it with the subspace topology.

Let $\Tinf \subset \Ginf$ denote the set of rooted marked trees. Also, for $h \in \NZ$, let $\T{h} \subset \Tinf$ denote the set of rooted marked trees of depth at most $h$; for $h = 0$, $\T{0}$ is the set $\X$. A typical element of $\Tinf$ will generally be denoted by $(\tau, o, (x, y))$,  $(\tau, (x, y))$, or simply $\tau$, depending on the context. For $\tau \in \Tinf$, let both $\text{deg}_{\tau}(o)$ and $|N_o(\tau)|$ denote the degree of the root of $\tau$.

We now define certain subgraphs associated with a (marked) graph. Let  $(G, (x, y))$ be an (unrooted) graph with $\X$-valued vertex marks  $\{x_v\}_{v \in G}$ and  $\Y \times \Y$-valued  edge marks $\{(y_{(u,v)}, y_{(v,u)})\}_{\{u,v\} \in E}$. For a vertex $v \in G$, let $\Conn_v(G, (x, y))$ denote the marked connected component containing the vertex $v$ rooted at $v$,  viewed as an element of $\Ginf$. Also, let $\cl_v(G, (x, y))$ denote the marked depth $1$ tree rooted at $v$ viewed an element of $\T{1}$. 

Throughout this paper, specific (deterministic) marked trees will be denoted by $\tau$, $\tau^\prime$, etc., and $\Tinf$-valued random elements will be denoted by $\TB$, $\TB^\prime$, etc. We will consider both marked graphs on a finite vertex set (whose vertices are labeled) and elements of $\Ginf$ (i.e., unlabeled rooted marked graphs) -- whether a marked graph belongs to the former or latter set should be clear from the context.

\subsubsection{Ulam-Harris-Neveu labeling for rooted trees}
\label{sec:labeling}
We now describe the Ulam-Harris-Neveu random labeling scheme (see \cite[Section VI.2]{Har63}) which produces a random rooted marked labeled tree, with vertices assigned labels from the set
$\bbV = \{o\}\cup \left(\bigcup_{k=1}^\infty \N^k\right)$ in the manner described below.  (These vertex labels should not be confused with vertex marks.) Given $(\tau, o, (x,y)) \in \Tinf$,  the root is assigned the label $o$. Labels are then assigned to the other vertices based on a breadth-first search starting from the root with ties broken uniformly at random. In other words, the vertices in $N_o(\tau)$ are assigned a label from $\{1, 2, \ldots, |N_o(\tau)|\}$ uniformly at random. Now, suppose that for some $h \geq 1$, all vertices at a distance less than or equal to $h$ from the root have been assigned labels. Also, for $u = (u_1, \ldots, u_i) \in \N^i$ and $v  = (v_1, \ldots, v_j) \in \N^j$, let $uv \in \N^{i+j}$ denote the concatenation $(u_1, \ldots, u_i, v_1, \ldots, v_j)$.
Then, fix any vertex $v$ that is at a distance $h+1$ from the root and let $p(v)$ denote its parent (i.e., the unique neighbor of $v$ whose distance from the root is $h$). 
On the event that $p(v)$ has label $i = (o, i_1, i_2, \ldots, i_h)$ and $|N_{p(v)}(\tau)| = \ell +1$, assign $v$ the label $(i,m)$ where $m$ is chosen uniformly at random from $\{1,2, \ldots, \ell\}$, independently of all other random assignments made thus far.
Carry out the same procedure (independently) for each vertex at a distance $h+1$ from the root. Proceed by induction to label the whole tree.
Note that, under this labeling, an element $(\tau, (x, y)) \in \Tinf$ is viewed as a random rooted marked tree whose vertex set is a subset of $\bbV$.

\subsection{Probability measures on rooted marked  graphs}
\label{sec:prelim-probability-graphs}
Let $\P(\Ginf)$ denote the space of probability measures on $\Ginf$ equipped with the weak topology. Also, let $\P(\Tinf ) \subset \P(\Ginf)$ denote the set of probability measures on  $\Tinf$ equipped with the subspace topology. A typical element of $\P(\Ginf)$ will generally be denoted by $\rho$. Expectation with respect to $\rho$ is denoted by $\Exp_\rho$. For $h \in \NZ$, let $\rho_h$ denote the marginal of $\rho$ on the depth $h$ truncation of the (random) rooted marked graph. Note that $\rho_h \in \P(\G{h})$. In particular, $\r{o} \in \P(\X)$ denotes the law of the mark on the root vertex. If $\rho \in \P(\Tinf)$, then $\rho_h \in \P(\T{h})$. For $h = 1$, a  $\P(\T{1})$ element will generally be denoted by $\mu$. Let $\rho_{o, \text{deg}} \in \P(\NZ)$ denote the law (under $\rho$) of the degree of the root, that is, the law of $\dego$ when $\Law(\TB) = \rho$. Finally, given $\b > 0$,  let $\PB \subset \P(\Tinf)$ denote the set of probability measures on $\Tinf$ such that $\EDRinf = \b$.

\begin{remark}
\label{remark:label-unlabel}
Although $\P(\Tinf)$ is the space of probability measures on {\em unlabeled} rooted marked trees, it is sometimes convenient to view it as a collection of probability measures on {\em labeled} rooted marked trees under the Ulam-Harris-Neveu labeling scheme. This viewpoint is especially useful to describe certain functions defined on $\P(\Tinf)$. Throughout this paper, for a $\Tinf$ random element $\TB$, wherever vertex labels are used for $\TB$, it is understood that $\TB$ is viewed as a rooted labeled tree labeled according to the Ulam-Harris-Neveu scheme.
\end{remark}

\section{The LDP for the neighborhood empirical measure}
\label{sec:results-nbd}
In this section, we introduce the model and state our main results on the LDP for the neighborhood empirical measure.

\subsection{Marked random graph model}
\label{sec:model-iid}
We consider random graphs whose vertices and edges are marked with independent random variables. We consider three families of sparse graphs, each parametrized by a fixed constant $\k > 0$:
\begin{enumerate}
\item[1. $(\CM)$] {\em Graphs with a fixed degree distribution:} Let $M \in \N$, and let $\alpha \in \P(\NZ)$ have mean $\k$ and support contained in $\{0,1,\ldots,M\}$. Let $\{\alpha_n\}$  be such that
\begin{enumerate}
\item For each $n \in \N$, the support of $\alpha_n$ lies in $\{0,1,\ldots, M\}$;
\item For each $0 \leq i \leq M$, $\alpha_n(\{i\})$ is of the form $\frac{k}{n}$ for some $k \in \{0,1, \ldots, n\}$, and $n\sum_{i =0}^M i \alpha_{n}(\{i\})$ is even;
\item $\alpha_n \to  \alpha$ in $\P(\NZ)$ as  $n \to \infty$.
\end{enumerate}
Let $\GCM$ denote the set of graphs on $n$ vertices\footnote{The restriction that $n\sum_{i =0}^M i \alpha_{n}(\{i\})$ is even is required since the sum of vertex degrees must be even. Under that assumption,  by the Erd\H{o}s-Gallai theorem, $\GNA$ is nonempty for large $n$.} whose degree distribution equals $\alpha_n$.
\item[2. ($\FE$)] {\em Graphs with a fixed number of edges:} Let $\{m_n\}$ be a sequence such that $m_n/n \to \k/2$ as $n \to \infty$.  Let $\GNM$ denote the set of graphs on $n$ vertices with $m_n$ edges.
\item[3. $(\ER)$] {\em Sparse \Erdos random graphs:} Let $\GER$ denote the \Erdos random graph on $n$ vertices with connection probability $\k/n$.
\end{enumerate}
We now introduce three corresponding sequences $\{G_n\}$ of independent random graphs:
\begin{enumerate}
    \item $\CM(\alpha_n, \k)$ graph sequence: where $G_n$ is uniformly sampled from $\GNA$ for each $n \in \N$; 
    \item $\FE(\k)$ graph sequence: where $G_n$ is uniformly sampled from $\Gnm$ for each $n \in \N$;
    \item $\ER(\k)$ graph sequence: where $G_n$ is the \Erdos random graph with connection probability $\k/n$ for each $n \in \N$. 
\end{enumerate}
Fixing one of the three graph sequences above, we define a corresponding marked graph sequence by assigning i.i.d.\,marks to the vertices and edges of each graph $G_n$ in the sequence. More precisely, let $\X$ and $\Y$ be nonempty finite sets, and let $\nu$ and $\xi$ be probability measures on $\X$ and $\Y \times \Y$,  respectively. We mark each vertex of $G_n$ with i.i.d.\,random variables $\{X^n_v\}_{v \in G_n}$ with law $\nu$, independent of $G_n$. We also mark  each edge of $G_n$ with i.i.d.\,random variables $\{(Y_{(u,v)}, Y_{(v,u)})\}_{u, v \in G_n, u \in N_v(G_n)}$ with law $\bar{\xi} \in \P(\Y \times \Y)$, independent of $G_n$ and $\{X^n_v\}_{v \in G_n}$,  defined by 
\begin{align}
\bar{\xi}(y, y^\prime) := \frac{1}{2}[\xi(y, y^\prime) + \xi(y^\prime, y)], \quad y, y^\prime \in \Y,
\label{eqn:def-xibar}
\end{align}
and associate the mark $Y_{(v,u)}$ (resp. $Y_{(u,v)}$) with the vertex $v$ (resp. $u$). In  other words, for every (undirected) edge $\{u,v\}$ of $G_n$, we sample an element $(Y, Y')$ from the law $\xi$ and assign $Y$ to $u$ and $Y'$ to $v$ with probability $1/2$ and assign $Y'$ to $u$ and $Y$ to $v$ with probability $1/2$,  independent of everything else. This gives us the corresponding random marked graph sequence $\{(G_n, (X^n, Y^n))\}$. 
\subsection{The neighborhood empirical measure}
Given any of the marked random graph sequences $\{(G_n, (X^n, Y^n))\}$ described in Section \ref{sec:model-iid}, we consider the
corresponding neighborhood empirical measure, defined as:
\begin{align}
L_n := \frac{1}{n} \sum_{v \in G_n} \delta_{(\cl_v(G_n, (X^n, Y^n)))}, \quad  n \in \N,
\label{eqn:neighborhood-measure}
\end{align}
where $\cl_v(G_n, (X^n, Y^n))$ is the first neighborhood of $(G_n, (X^n, Y^n))$ rooted at $v$, viewed as an element of $\T{1}$.
Then $L_n$ is a $\P(\T{1})$ random element. 
We denote the the neighborhood empirical measure of the $\CM(\alpha_n, \k)$, $\FE(\k)$ and $\ER(\k)$ marked random graph sequences by $L^{\CM}_n$, $L^{\FE}_n$ and $L^{\ER}_n$, respectively.

\begin{remark}
    \label{remark:local-convergence-nbd} 
    Let $\eta_1$ be the law of a $\T{1}$ random element having i.i.d.\,vertex and edge marks with law $\nu$ and  $\xi$, respectively, both independent of each other and independent of the root degree. When the root degree distribution is   $\text{Poi}(\k)$ (resp. $\alpha$), we denote it by $\etab{\k}_1$ (resp. $\etaa_1$). It is well known that as $n \to \infty$, both $L^\FE_n$ and $L^\ER_n$ (resp. $L^\CM_n$) converge weakly to $\etab{\k}_1$ (resp. $\etaa_1$); see, e.g.,   \cite[Theorem 5.8]{DemMon10}, \cite[Theorem 5.8]{OliReiSto20}, \cite[Theorem 3.2]{LacRamWu23}.
\end{remark}

\subsection{Statement of the neighborhood empirical measure LDP}
To state the LDP for the sequences $\{L^\CM_n\}$, $\{L^\FE_n\}$ and $\{L^\ER_n\}$, we need to introduce some notation in order to define the rate functions. Recall $\eta_1$ from Remark \ref{remark:local-convergence-nbd}, and let $\mu \in \P(\T{1})$ be such that $\mu \ll \eta_1$ and $0 < \Exp_{\mu}\sqpr{\dego} < \infty$. Next, we define the size-biased distribution $\bar{\mu}$ of $\mu$ by
\begin{align}
\frac{d\bar{\mu}}{d\mu}(\tau) := \frac{\text{deg}_\tau(o)}{\Exp_{\mu}\sqpr{\dego}}, \quad \tau \in \T{1}.
\label{eqn:SBRinf-mu}
\end{align}
Note that $\bar{\mu}$ is indeed a probability measure since $\Exp_{\mu} \sqpr{\frac{d\bar{\mu}}{d\mu}(\TB) } = 1$. Further, note that there is at least one child under $\bar{\mu}$, that is, $\bar{\mu}(\dego \geq 1) = 1$.

We now define several important marginal laws of $\mu$. Let $\TB = (\TB, (\bfX, \bfY))$ be a random $\T{1}$ element distributed according to $\bar{\mu}$. Viewing $\TB$ as a random labeled tree using the Ulam-Harris-Neveu scheme described in Section \ref{sec:labeling} and noting that $\bar{\mu}(\dego \geq 1)$, let $\bfX_v$ (resp. $\bfX_o$) denote the vertex mark of the vertex with label $v$ (resp. the root) and, similarly, let $\bfY_{(v,o)}$ (resp. $\bfY_{(o,v)}$) denote the edge mark on the edge $\{v, o\}$ associated with the vertex with label $v$ (resp. the root $o$). Let $\SBMuRM$ denote the law of $(\bfX_1,\bfY_{(1,o)})$ and let  $\SBMuCRM$ denote the conditional law\footnote{We fix a version of this conditional law for each $\mu \in \P(\T{1})$, the choice of which is inconsequential in the definition of the rate function.} of $(\bfX_1,\bfY_{(1,o)})$ given $(\bfX_o, \bfY_{(o,1)})$.
Additionally, let $\SBMuM \in \P((\X \times \Y)^2)$ denote the joint law of $((\bfX_1, \bfY_{(1,o)}), (\bfX_o, \bfY_{(o,1)}))$, that is,
\begin{align}
    \SBMuM = \Law((\bfX_1, \bfY_{(1,o)}), (\bfX_o, \bfY_{(o,1)})). \label{eqn:sbmubar}
\end{align} We define the size-biased quantities $\SBetabarM$, $\SBetabarMOne$ and $\SBetabarMOneGivenZero$ in an exactly analogous fashion. Using $\mu \ll \eta_1$, it can be shown that $\SBMuM \ll \SBetabarM$ (see Lemma \ref{lemma:pirllpieta}).

Next, for $\mu \in \P(\T{1})$ with $0 < \Exp_{\mu}\sqpr{\dego} < \infty$, define the probability measures $\MP 
= \MPdef$ and  $\MCP = \MCPdef$ by 
\begin{align}
\frac{d\MP}{d\eta_1}(\tau) & : = \prod_{v \in N_o(\tau)} \frac{d\SBMuRM}{d\SBetabarRM}(X_v, Y_{(v,o)}), \quad \tau = (\tau, (X, Y)) \in \T{1},
\label{eqn:RP}
\end{align}
and
\begin{align}
\label{eqn:RCP}
\frac{d\MCP}{d\eta_1}(\tau) & : = \prod_{v \in N_o(\tau)} \frac{d\SBMuCRM}{d\SBetabarCRM}(X_v, Y_{(v,o)} \mid X_o, Y_{(o,v)}), \quad \tau = (\tau, (X, Y)) \in \T{1}.
\end{align}
Note that both $\MP$ and $\MCP$ depend on $\eta_1$, but we suppress this dependence in the above definition for the sake of readability. In the particular case when both $\mu$ and $\eta_1$ are supported on $\k$-regular trees without edge marks, we have $\SBMuM = \mu^{1,o}$ and $\SBetabarM = \eta^{1,o} = \nu \otimes \nu$; in this case,  $\MP$ and $\MCP$ take the explicit form
\begin{align*}
    \MP(x_o, x_1, \ldots, x_\k) & = \nu(x_o) \bigotimes_{v=1}^\k \mu^1(x_v); \\
    \MCP(x_o, x_1, \ldots, x_\k) & = \nu(x_o) \bigotimes_{v=1}^\k \mu^{1|o}(x_v \mid x_o).
\end{align*}
Although not immediately apparent, in the general case as well, it can be shown that both $\MP$ and $\MCP$ are indeed probability measures (see Lemma \ref{lemma:prob-measures-h=1}).

The rate functions will be expressed in terms of the family of functions  $\frakI_{\b, \eta_1}  : \P(\T{1}) \to [0, \infty]$, $\b \geq 0$, defined as follows:
\begin{align}
\frakI_{\b, \eta_1}(\mu) : = \begin{dcases}
H(\mu_{o} \| \nu) & \text{if } \beta = 0 \text{ and } \Exp_{\mu}[\dego] = 0,\\
\frac{1}{2}  \sqpr{ H(\mu \| \MP) + H(\mu \| \MCP)} & \text{if } \beta > 0,  \,  \Exp_{\mu}[\dego] = \beta, \, \SBMuM \text{ is symmetric} \,  \text{ and } \mu \ll \eta_1,\\
\infty & \text{ otherwise}.
\end{dcases}
\label{eqn:frakI2-generic}
\end{align}
\noindent Also, define $\ell_\k : \R_+ \to \R_+$ by
\begin{align}
    \ell_\k(\beta) := \frac{\k}{2}\left(\frac{\beta}{\k} \log\left(\frac{\beta}{\k}\right)  - \frac{\beta}{\k} + 1 \right), \quad \b \geq  0,
    \label{eqn:def-ell}
\end{align}
with the convention $0 \log 0 = 0$. We define the rate functions $\frakI^\CM, \frakI^\FE, \frakI^\ER : \P(\T{1}) \to  [0,\infty]$ as follows:
\begin{align}
\frakI^\CM(\mu) & :=
\begin{dcases}
\frakI_{\k, \etaa_1}(\mu) & \text{ if } \mu_{o, \text{deg}} = \alpha,\\
\infty & \text{ otherwise};
\end{dcases}
\label{eqn:J-CM-iid}\\
\frakI^\FE (\mu) & := \frakI_{\k, \etab{\k}_1}(\mu);
\label{eqn:J-FE-iid} \\
\frakI^\ER(\mu) & := 
\begin{dcases}
\ell_\k(\b') +  \frakI_{\b', \etab{\b'}_1}(\mu) & \text{ if } \b' := \Exp_{\mu}[\dego] < \infty,\\
\infty & \text{ otherwise}.
\end{dcases}
\label{eqn:J-ER-iid}
\end{align}

We can now state our result for the neighborhood empirical measure.
\begin{theorem}[LDP for the neighborhood empirical measure]
\label{thm:iid-nbd}
The sequence $\{L_n^{\CM}\}$ (respectively, $\{L_n^{\FE}\}$ and $\{L_n^{\ER}\}$) satisfies the LDP on $\P(\T{1})$  with rate function $\frakI^{\CM}$ (respectively, $\frakI^{\FE}$ and $\frakI^{\ER}$).
\end{theorem}

Observe that the rate functions in \eqref{eqn:J-FE-iid}-\eqref{eqn:J-ER-iid} have a tractable form, expressed as the sum of two relative entropy terms, with the first term measuring the entropic cost of deviation of the distribution of the leaf marks from independence and the second entropic cost measuring conditional independence (given the root) of the leaf marks in $\mu$. It is interesting to note that although the $\CM$ and $\FE$ random graph sequences themselves are very different, we have shown that their rate functions take a similar form. The same is true for the $\ER$ sequence, except that there is an additional cost $\ell_{\kappa}(\cdot)$ to capture possible deviations of the mean root degree from that of the true law (which are forbidden at the large deviation scale for the $\CM$ and $\FE$ random graph sequences). In the absence of edge marks, this common form can in fact be simplified further.

\begin{corollary}
    \label{cor-kapparf}
    In the particular case when there are no edge marks, we have the following alternative representation for $\frakI_{\b, \eta_1}$:
\begin{align}
\frakI_{\b, \eta_1}(\mu) = \begin{dcases}
H(\mu_{o} \| \nu) & \text{if } \beta = 0 \text{ and } \Exp_{\mu}[\dego] = 0,\\
  H(\mu \| \MCP) + \frac{\b}{2}  H(\SBMuM \| \SBMuRM \otimes \SBMuMZero) & \text{if } \beta > 0,  \,  \Exp_{\mu}[\dego] = \beta, \, \SBMuM \text{ is symmetric} \,  \text{ and } \mu \ll \eta_1,\\
\infty & \text{ otherwise}.
\end{dcases}
\end{align}
\end{corollary}
The proof of Corollary \ref{cor-kapparf} is deferred to Section \ref{sec:proof-vertex-only}. In the next section, we provide one example to illustrate how the results of this section facilitate the analysis of conditional limit laws. 

\subsection{A Gibbs conditioning principle}
\label{sec:gibbs}
For simplicity, we shall only consider vertex marks. 
Let $\TX$ denote the space of rooted marked trees of depth $1$ with vertex marks from $\X$. 
Given a function $h : \X \to \R$, consider the local $h$-sum functional $f : \TX \to \R$ defined by
\begin{align}
f(\tau, X) = \sum_{v \in N_o(\tau)} h(X_v), \quad (\tau, X) \in \TX.
\label{eqn:gibbs-sum-fn}
\end{align}
Consider the neighborhood empirical measure $L_n=L_n^\CM$ of the  $\CM(\alpha_n, \k)$ random graph sequence with limiting degree distribution $\alpha \in \P(\{0,1,\ldots, M\})$ and limiting true law $\etaa_1 \in \P(\TX)$.
Since $\alpha$ has finite support and $\X$ is a finite set, $f$ is uniformly bounded. Therefore, by the weak convergence result in Remark \ref{remark:local-convergence-nbd}, it follows that $\langle L_n, f \rangle : = \Exp_{L_n}[f]$ converges weakly to $\Exp_{\etaa_1}[f]$ as $n \to \infty$. Quantities of the form $\Exp_{L_n}[f]$  are relevant for many interacting particle system models. For example, in the SIR model for epidemics (see \cite{CocRam23}), each vertex takes the state S (susceptible), I (infected) or R (recovered), and at any given time, the probability that a susceptible vertex becomes infected depends on the
number of neighboring nodes that are infected.  In this case, if $f$ is defined as in \eqref{eqn:gibbs-sum-fn} with $h(X_v) = \indf{\{X_v = I\}}$, then the quantity  $\langle L_n, f \rangle = \Exp_{L_n} [f]$  captures the growth rate of infected individuals at any time.  The most likely way in which there is an unusually large infection growth rate (starting from
an i.i.d.\,initial condition) is governed by the following question for any fixed $c >  \Exp_{\etaa_1}[f]$:
\begin{question}
Conditioned on $\langle L_n, f \rangle \geq c$, how does $L_n$ behave for large $n$?    
\label{question:gibbs}
\end{question}
The form of the rate function in \eqref{eqn:J-CM-iid} allows us to provide a precise answer to this question.

\begin{proposition}[Gibbs conditioning principle]
  \label{prop-gcp}
Suppose $\Exp_{\etaa_1}[f] < c < \k \cdot \max_{x \in \X} h(x)$. Then for any open set $A \ni \mu^{*,c}$, we have
\begin{align*}
\lim_{\delta \downarrow 0} \limsup_{n \to \infty} \frac{1}{n} \log \Prob\bigg(L_n \notin A \, \, \bigg| \, \,  \langle L_n, f \rangle  \geq  c - \delta  \bigg)  < 0,
\end{align*}
where
\begin{align}
\mu^{*,c}(\tau, X) := \gamma(\text{deg}_\tau(o), X_o) \prod_{v \in N_o(\tau)} \psi_\gamma(X_v), \quad (\tau, X) \in \TX, \label{eqn:mustarc}
\end{align}
with 
\begin{align*}
\psi_\gamma(x) := \frac{1}{\k} \sum_{n =0}^M n \gamma(n, x), \quad x \in \X,
\end{align*}
where $\gamma \in \P(\{0,1,\ldots, M\} \times \X)$ and $\lambda > 0$ is the unique solution to
\begin{align*}
\gamma(n, x) = \alpha(n) \frac{\nu(x)\exp\{\lambda n h(x)\}}{ \sum_{y \in \X} \nu(y)\exp\{\lambda n h(y)\}}, \quad n \in \{0,1,\ldots, M\}, \, x \in \X, 
\end{align*}
and $\sum_{m,y} m h(y) \gamma(m, y) = c$. 
In particular, this implies
\begin{align*}
\lim_{\delta \downarrow 0} \lim_{n \to \infty} \Prob\bigg(L_n \in \cdot  \, \, \bigg| \, \, \langle L_n, f \rangle \geq c-\delta \bigg)  = \delta_{\mu^{*,c}}(\cdot).
\end{align*}
\label{prop:gibbs}
\end{proposition}
\noindent 
In other words, conditioned on the rare event $\langle L_n, f\rangle \geq c$, the distribution of the neighborhood of a uniformly chosen vertex in the marked graph has the following form: the root degree and the root mark is sampled from the tilted distribution $\gamma$ whose density (with respect to the true law $\alpha \otimes \nu$) is proportional to $\exp\{\lambda n h(x)\}$ for a uniquely specified $\lambda > 0$, and the leaf marks are i.i.d.\,with law given by a scaled average of $\gamma$. 

\begin{remark}
In the special case of $\k$-regular graphs, $\langle L_n, f \rangle = (\kappa+1) \sum_{v \in G_n} X_v^n$, and hence, the Gibbs conditioning principle of Proposition \ref{prop:gibbs} reduces to that of (the usual) Sanov's theorem \cite{Csi84} in which  $\mu^{*,c}$ corresponds to i.i.d.\,vertex marks on the $\k$-star.
\end{remark}

\section{The LDP for the component empirical measure}
\label{sec:results-com}
We now state our main result in full generality. Given one of the sequences of random graphs  $\{(G_n, (X^n, Y^n))\}$ introduced in Section \ref{sec:model-iid}, we define the corresponding  component empirical measure sequence:
\begin{align}
U_n := \frac{1}{n} \sum_{v \in G_n} \delta_{(\Conn_v(G_n, (X^n, Y^n)))}, \quad  n \in \N, \label{eqn:component-measure}
\end{align}
where we recall that $\Conn_v(G_n, (X^n, Y^n))$ denotes the connected  component of $(G_n, (X^n, Y^n))$ rooted at $v$, viewed as an element of $\Ginf$. Then  $U_n$ is a $\P(\Ginf)$ random element; the randomness in $U_n$ comes from the randomness of both the underlying graph and of the vertex and edge marks. Note that the neighborhood empirical measure $L_n$ defined in \eqref{eqn:neighborhood-measure} is the depth 1 marginal of $U_n$. We denote the component empirical measure of the $\CM(\alpha_n, \k)$, $\FE(\k)$ and $\ER(\k)$ random graph sequences by $U^\CM_n$, $U^\FE_n$ and $U^\ER_n$, respectively.
\subsection{Preliminaries}
We first introduce some concepts required to state the full LDP.

\subsubsection{Local weak convergence}
We first recall the definition of local weak convergence \cite{BenSch01} and describe the typical behavior of the component empirical measures $\{U^\CM_n\}$, $\{U^\FE_n\}$ and $\{U^\ER_n\}$.
\begin{definition}[Local weak convergence] A sequence of marked random graphs $\{(G_n, (X^n, Y^n))\}$, where $G_n$ has $n$ vertices, {\em converges locally in law} to $\rho \in \P(\Ginf)$ if $\{U_n\}$,  the corresponding sequence of component empirical measures, converges to $\rho$ in $\P(\Ginf)$ weakly as $n \to \infty$.
\end{definition}

\begin{remark}
\label{remark:local-convergence}
The following convergence results generalize Remark \ref{remark:local-convergence-nbd}.
    \begin{enumerate}
        \item Let $\eta^{\alpha} \in \P(\Tinf)$ denote the law of the size-biased  Galton-Watson tree with offspring distribution $\alpha$ and i.i.d.\,vertex marks with law $\nu$ and i.i.d.\,edge marks (independent of the vertex marks) with law $\xi$, both independent of the underlying tree. It is well known that $\{U_n^\CM\}$ converges locally in law to $\eta^{\alpha}$ as $n \to \infty$ (see, e.g.,  \cite[Theorem 5.8]{OliReiSto20}, \cite[Theorem 3.2]{LacRamWu23}).
         \item For $\b \geq 0$, let $\etab{\b}$ denote the law of the Galton-Watson tree with $\text{Poi}(\b)$ offspring distribution and i.i.d.\,vertex marks and i.i.d.\,edge marks (independent of the vertex marks), both independent of the underlying tree.  It is well known that both $\{U_n^\FE\}$ and $\{U_n^\ER\}$ converge locally in law to $\etab{\k}$ as $n \to \infty$ (see, e.g., \cite[Section 2.1]{DemMon10}, \cite[Theorem 3.2]{LacRamWu23}).
    \end{enumerate}
\end{remark}

\subsubsection{Unimodularity}
Unimodularity is an important symmetry property associated with $\P(\Ginf)$ elements. Roughly speaking, this property suggests that ``everything shows up at the root.'' The precise definition is via the following mass transport principle \cite{BenSch01,AldLyo07}. Let $\Gdstar[\X ; \Y]$ denote the space of unlabeled doubly rooted marked graphs (i.e., unlabeled marked graphs with two distinguished vertices, which could coincide), equipped with the corresponding topology of local convergence (see \cite[Definition 2.1]{AldLyo07}).
\begin{definition}[Unimodular measure]
\label{def:unimod}
A probability measure $\rho \in \P(\Ginf)$ is said to be unimodular if, for any measurable function $f : \Gdstar[\X; \Y] \to \R_+$, we have
\begin{align}
\label{eqn:unimod}
\Exp_{\rho}\left[\sum_{v \in \GB}f(\GB, o, v)\right] =  \Exp_{\rho}\left[\sum_{v \in \GB}f(\GB, v, o)\right].
\end{align}
The definition of unimodularity for 
probability measure $\rho \in \P(\Gstar)$ is exactly analog.
Let  $\PUB$ denote the set of unimodular probability measures on $\Tinf$ such that $\Exp_{\rho}\sqpr{\dego} = \b$.
\end{definition}
\noindent The above definition also applies to  $\P(\Gstar)$ elements, that is, probability measures on (unmarked) rooted graphs.
\begin{example}
The following are examples of unimodular measures.
    \begin{enumerate}
        \item Let $(G_n, (x^n, y^n))$ be a deterministic marked graph on $n$ vertices. Its component empirical measure $U_n \in \P(\Ginf)$ is a unimodular probability measure.
        \item  Let $Q \in \P(\NZ)$ be a probability measure with nonzero finite mean. The size-biased Galton-Watson tree is defined as follows: the root has offspring distribution $Q$; all subsequent vertices have independent number of offsprings with distribution $\bar{Q}$, where 
        \begin{align*}
            \bar{Q}(k) := \frac{(k+1) Q(k+1)}{\sum_{j \in \N} j Q(j)}, \quad k \in \NZ.
        \end{align*}
        The law of the size-biased Galton-Watson tree is unimodular.
        \item The laws $\etab{\b}$, $\b \geq 0$, and $\etaa$ in Remark \ref{remark:local-convergence} are all  unimodular.
    \end{enumerate}
    \label{example:unimod}
\end{example}

\subsubsection{Admissibility}
\begin{figure}
\centering
\begin{center}
\begin{tikzpicture}[scale=0.9]
\filldraw[black] (2,0) circle(3pt);
\filldraw[black] (2,2) circle(3pt);
\filldraw[black] (0,2) circle(3pt);
\filldraw[black] (1,-1.5) circle(3pt);
\filldraw[black] (3,-1.5) circle(3pt);
\filldraw[black] (3,3) circle(3pt);
\filldraw[black] (4,2.5) circle(3pt);
\filldraw[black] (0,0) circle(3pt);
\draw [black, thick] (2,0)--(2,2);
\draw [black, thick] (0,2)--(2,2);
\draw [black, thick] (0,0)--(0,2);
\draw [black, thick] (2,0)--(1, -1.5);
\draw [black, thick] (2,0)--(3, -1.5);
\draw [black, thick] (2,2)--(3,3);
\draw [black, thick] (3,3)--(4, 2.5);
\node at (2.3, 1.8) {$u$};
\node at (3, 3.3) {$v$};
\draw [black, dashed] (2.1, 3) -- (3.2, 1.8);
\draw [dotted,thick] (1.1,-0.35) ellipse (2.5cm and 2.7cm);
\node  at (-2.8, 0) {$\Conn_u(\tau \setminus  \{u, v\})$};
\node at (2,2.5) {$y_{(u,v)}$};
\end{tikzpicture}
\end{center}
\caption{ Illustration of the quantities $\Conn_u(\tau \setminus  \{u, v\})$ and $y_{(u,v)}$ from  Definition \ref{def:vcutu} when $\{u, v\}$ is an edge in $\tau$.}
\label{fig:vcutu}
\end{figure}
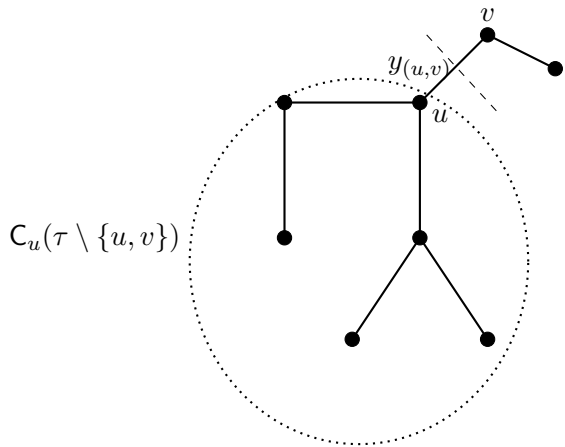
We now define certain elements that generalize the size-biased edge marginal $\SBMuM$ defined in \eqref{eqn:sbmubar}. Rather than edge marginals, the objects associated with the vertices 1 and $o$ are now subtrees (see Figure \ref{fig:vcutu}), as specified in Definition \ref{def:vcutu}.

\begin{definition}
If $\tau$ is a (labeled) marked tree and $\{u, v\}$ is an (undirected) edge in $\tau$, let $\tau(v \setminus u) \in   \Tinf \times \Y$ denote the pair $(\Conn_v(\tau  \setminus \{u, v\}), y_{(v, u)})$, where  $\Conn_v(\tau \setminus \{u, v\})$ is the subgraph of $\tau$ rooted at $v$,  after removing the edge $\{u,v\}$,  viewed as an element of $\Tinf$,  and $y_{(v, u)}$ is the mark on the edge $\{u, v\}$ associated with $v$. For $h \in \NZ$,  let $\tau(v \setminus u)_h$ denote the pair $(\Conn_v(\tau \setminus \{u, v\})_h, y_{(v, u)})$, where $(\Conn_v(\tau \setminus \{u, v\})_h \in \T{h}$ is the truncation of  $\Conn_v(\tau \setminus \{u, v\})$ up to depth $h$. Additionally, if either $u$ or $v$ is not a vertex in $\tau$ or $\{u,v\}$ is not an edge in $\tau$, then define $\tau(v \setminus u) := 0 $.
\label{def:vcutu}
\end{definition}
\begin{remark}
Note that for a $\Tinf$-valued random variable $\TB$, by Remark \ref{remark:label-unlabel} and Definition \ref{def:vcutu}, both $\TBOneO$ and $\TBOOne$ are random variables taking values in $\cpr{\Tinf \times \Y} \cup \{0\}$. Furthermore, on the event that $\dego \geq 1$, both $\TBOneO$ and $\TBOOne$ take values in $\Tinf \times \Y$.
\end{remark}

Next, for $\rho \in \PB$ with $\b > 0$,  similar to the definition in \eqref{eqn:SBRinf-mu}, we define the size-biased version $\bar{\rho}$ of $\rho$ by
\begin{align}
\frac{d\SBRinf}{d\rho}(\tau) :=\frac{\text{deg}_\tau(o)}{\b} =  \frac{\text{deg}_\tau(o)}{\Exp_{\rho}\sqpr{\dego}}, \quad \tau \in \Tinf.
\label{eqn:SBRinf}
\end{align}
Note that $\SBRinf$ is indeed a probability measure because
$
\Exp_{\rho}\sqpr{\frac{d\SBRinf}{d\rho}(\TB)} = 1.
$
Also, note that $\bar{\rho} \ll \rho$ and $\bar{\rho}(\dego \geq 1) = 1$. Likewise, for $h \in \N$ and $\r{h} \in \P^\b(\T{h})$, we define the size-biased version $\SBR{h}$ using \eqref{eqn:SBRinf}, but with $\rho$ replaced by $\r{h}$.

\begin{definition}[Size-biased marginals] 
Let $h \in \N, \r{h} \in \P(\T{h})$ and $\Law(\TB) = \r{h}$, with $0 < \Exp_{\r{h}}\sqpr{\dego} < \infty$.
\begin{enumerate}
    \item Let $\SBRM{h}$ denote the marginal law of $(\TBOOne_{h-1}, \TBOneO_{h-1})$. Since $\SBR{h}(\dego \geq 1)$, it follows that $\SBRM{h} \in \P((\T{h-1} \times \Y)^2)$;
    \item Let $\SBRRMh{h} \in \P(\T{h-1} \times \Y)$ denote the first marginal of $\SBRM{h}$, that is, the law of $\TBOOne_{h-1}$;
    \item For $\tau^\prime \in \T{h-1} \times \Y$, let $\SBRCRMh{h}(\cdot | \tau^\prime)$ denote the conditional law of $\TBOOne_{h-1}$ given $\TBOneO_{h-1} = \tau^\prime$.
\end{enumerate}
\label{def:admissible-h}
\end{definition}
The following notion of {\em admissibility}  generalizes the symmetry property of the size-biased edge marginals $\SBMuM$ in \eqref{eqn:sbmubar}.
\begin{definition}
The measure $\r{h}$ is said to be {\em admissible} if $\SBRM{h}$  is symmetric.
\label{def:admissibility}
\end{definition}

\subsubsection{Unimodular extension of admissible laws}
\label{sec:unimod-extension-one-step}
For any $h \in \N$, we now define a certain admissible probability measure on $\T{h+1}$ as an extension of a given admissible probability measure on $\T{h}$ (see \cite[Section 1.2]{BorCap15}, \cite[Section  2]{DelAna19}). These probability measures will serve as the ``reference laws'' associated with the restriction of the component empirical measures up to depth $h$, and certain auxiliary probability measures associated with these extensions will appear as the ``base measures'' for the relative entropies in the rate function,  playing the same role as $\MP$ and $\MCP$ in the definition \eqref{eqn:frakI2-generic} of $\frakI_{\b, \eta_1}(\mu)$.

Let $\b > 0$, and let $\r{h} \in \P^\b(\T{h})$ be admissible. Let $\Law(\TB) = \r{h}$. For $\tau,\tau^\prime \in \T{h-1} \times \Y$ such that $\SBRM{h}(\tau, \tau^\prime) > 0$, define
\begin{align}
    \PU{h}(\tau, \tau^\prime)(\TBdef) & := \r{h}\cpr{ \TBOneO_h   = \TBdef \mid    \TBOneO_{h-1}  = \tau, \TBOOne_{h-1}  = \tau' }, \quad \TBdef \in \T{h},
    \label{eqn:def-PU}
\end{align}
that is,
$\PU{h}(\tau, \tau^\prime)(\cdot)$ denotes the conditional law of $\TBOneO_h$ given $\TBOneO_{h-1} = \tau$ and $\TBOOne_{h-1} = \tau'$. We now use these conditional laws to define a one-step unimodular extension of $\r{h}$, denoted by $\Rstar{h+1} \in \P(\T{h+1})$.

\begin{definition}[One-step unimodular extension]
First, sample $\TB \in \T{h}$ with law $\r{h}$ (i.e., $\Law(\TB) = \r{h}$). If $\dego = 0$, stop and set the unimodular extension to be the root. If $\dego \geq 1$, independently for each  $v \in N_o(\TB)$, sample  $\TB^\prime \in \T{h} \times \Y$ with law $\PU{h}(\TBOV_{h-1},\TBVO_{h-1})(\cdot)$ and replace $\TBOV_{h-1}$ with $\TB^\prime$ to obtain a $\T{h+1}$ random element. Set $\Rstar{h+1}$ as the law of the resultant $\T{h+1}$ random element. In other words, for $\tau \in \T{h}$, $\Rstar{h}(\tau)$ is given by
\begin{align}
\Rstar{h}(\tau) = \r{h-1}(\tau_{h-1}) \prod_{v \in N_o(\tau)} \PU{h-1}(\TOV_{h-2}, \TVO_{h-2})(\TOV_{h-1}).
\label{eqn:Rstarh}
\end{align}
\label{def:unimod-extension}
\end{definition}

The expression in \eqref{eqn:Rstarh} clearly shows that $\Rstar{h}$ is an extension of $\r{h-1}$ because
\begin{align}
    (\Rstar{h})_{h-1} = \r{h-1}.
    \label{eqn:UGW-marginal}
\end{align}

\begin{remark}
\label{rmk:unimod-full-extension}   
\begin{enumerate}
    \item We can express $\PU{h}$ in terms of $\SBRM{h}$ and some combinatorial quantities; see \eqref{eqn:PU}. This alternative expression will be used later in some of the proofs.
    \item     When $h = 1$, a similar extension procedure (referred to as the ``tiling procedure'')  has been carried out in a different context in  \cite[Section 4.2]{LacRamWu20} and \cite[Proposition 6.2]{Gan22} for certain models of interacting diffusions and jump processes, respectively,  to extend the ``local equation'' that describes the root-and-neighborhood marginal law of a  particle system to the law of the full particle system on the infinite random tree.
\end{enumerate}
\end{remark}

\subsection{Statement of the component empirical measure LDP}
\label{subs-compLDP}

We are now in a position to introduce the rate functions associated with the general component empirical measure LDP. 
Let $h \geq 2$. Let $\r{h} \in \P(\T{h})$ be admissible, suppose that $0 < \Exp_{\r{h}}\sqpr{\dego} < \infty$, and recall the definition of $\Rstar{h}$  from \eqref{eqn:Rstarh}.  It can be shown that $\r{h} \ll \Rstar{h}$ and  $\SBRM{h} \ll \SBRstarM{h}$ (see Lemma \ref{lemma:preliminary-prob-prop}).  Define the probability measures $\RP_h, \RCP_h \in \P(\T{h})$  by
\begin{align}
\RP_h & := \Rstar{h}, \label{eqn:RPh-def} \\ 
\frac{d\RCP_h}{d\Rstar{h}} (\tau) & := \prod_{v \in N_o(\tau)} \frac{d\SBRM{h}}{d\SBRstarM{h}}(\TOV_{h-1}, \TVO_{h-1}), \quad \tau \in \T{h}. \label{eqn:RCPh}
\end{align} 

\begin{remark}
It can be shown that  $\RCP_h$ is indeed a probability measure (see Section \ref{sec:alternate-form-h}). It is also true that
$\SBRRMh{h} = \SBRstarRM{h}$ (see Lemma \ref{lemma:preliminary-prob-prop}), and hence,
\begin{align}
\frac{d\RP_h}{d\Rstar{h}}(\tau) &= \prod_{v \in N_o(\tau)} \frac{d\SBRRMh{h}}{d\SBRstarRM{h}}(\TOV_{h-1}), \quad \tau \in \T{h}.
\label{eqn:RPh}
    \end{align}
For $h=1$, set $\RP_1 = \MP$, which is defined analogously in \eqref{eqn:RP}. Also, if $\rho \in \PUB$, it can be shown that $\r{h}$ is admissible for all $h \in \N$, see Remark \ref{remark:admissibility}.
\label{rmk:rh-rstarh}
\end{remark}

We now introduce a family of functions $\calI_{\b, \eta}: \P(\Ginf) \to [0, \infty]$, $\b \geq 0$ and $\eta \in \PUB$, as follows:
\begin{align}
\calI_{\b, \eta}(\rho) : = \begin{dcases}
H(\r{o} \| \nu) & \text{if } \beta = 0 \text{ and } \Exp_{\rho}[\dego] = 0, \\
\frac{1}{2} \sum_{h \geq 1} \sqpr{ H(\r{h} \| \RP_h) + H(\r{h} \| \RCP_h)} & \text{if } \beta > 0, \,  \rho \in \PUB \,  \text{ and } \r{1} \ll \eta_1,\\
\infty & \text{otherwise}.
\end{dcases}
\label{eqn:calI-generic}
\end{align}
Here, $\eta$ plays the role of the ``true law'' that satisfies the following assumption:
\begin{assumption}
$\eta$ is the law of a random tree that has i.i.d.\,vertex marks with law $\nu$ and i.i.d.\,edge marks (independent of the vertex marks) with law $\xi$, both independent of the tree structure, with $\Exp_{\eta}\sqpr{\dego} = \b$.
\label{assm:generic-eta-def}
\end{assumption}
We can now state our full LDP.
\begin{theorem}[LDP for the component empirical measure]
\label{thm:iid}
The sequence $\{U_n^{\CM}\}$ (respectively, $\{U_n^{\FE}\}$ and $\{U_n^{\ER}\}$) satisfies the LDP  on $\P(\Ginf)$ with rate function $\calI^{\CM}$ (respectively, $\calI^{\FE}$ and $\calI^{\ER}$), where
\begin{align}
\calI^{\CM}(\rho) & := 
\begin{dcases}
\calI_{\k, \etaa}(\rho) & \text{ if } \rho_{o, \text{deg}} = \alpha,\\
\infty & \text{ otherwise},
\end{dcases}
\label{eqn:I-CM-iid-new} \\
\calI^{\FE}(\rho) &:= \calI_{\k, \etab{\k}}(\rho),
\label{eqn:I-FE-iid-new} \\
\calI^{\ER}(\rho) & := 
\begin{dcases}
\ell_\k(\b') + \calI_{\b', \etab{\b'}}(\rho) & \text{ if }  \b' : = \Exp_{\rho}[\dego] < \infty,\\
\infty & \text{ otherwise},
\end{dcases}
\label{eqn:I-ER-iid-new}
\end{align}
where $\ell_k$ is defined in  \eqref{eqn:def-ell}  and  $\etab{\k}$, $\etaa$ and $\etab{\b}$ are defined in Remark \ref{remark:local-convergence}.
\end{theorem}

Note that the rate functions in Theorem \ref{thm:iid} consist of a countable sum of term, with the $h$th summand being an average of two relative entropies that capture the entropic cost of deviation of $\r{h}$ from $\RP_h$ and $\RCP_h$. They have a similar interpretation as in the case of the neighborhood empirical measure. 

\begin{remark}
\label{rmk:fe-er-compare}
Although the families $\{U_n^\FE\}$ and $\{U_n^\ER\}$ both converge weakly to $\etab{\k}$ in $\P(\Ginf)$, note that the rate functions that govern their large deviations behavior are slightly different (although of a similar form). Since the number of edges in the $\FE$ case is fixed, the probability that the average degree of the root deviates from the true mean is superexponentially small, that is, $\calI^\FE(\rho) = \infty$ unless $\rho \in \PK$. However, for the $\ER$ sequence, deviations of the average degree are permissible at the large deviation scale, and this is captured by the additional term $\ell_{\kappa}$ in $\calI^\ER$. 
\end{remark}

\subsection{Discussion and outline of  proofs}
\label{sec:discussion}
The proof of Theorem \ref{thm:iid} is carried out in two steps.

\noindent {\bf Step 1:} 
As indicated in the introduction, the LDP for $\{U^\CM_n\}$ was proved in \cite{CheRamYas23} and the LDP for $\{U^\FE_n\}$ and $\{U^\ER_n\}$ were proved in \cite{BalOlietal22}, by leveraging the colored configuration model of \cite{BorCap15} and its extension in \cite{DelAna19} to marked graphs. The rate functions for these LDPs (denoted by $\calboldI^\CM$, $\calboldI^\FE$ and $\calboldI^\ER$) are expressed in terms of various combinatorial quantities, including the so-called microstate entropy, and thus their definitions are deferred to Section \ref{sec:iid-proof-complete}; ; see \eqref{eqn:I-CM-entropy-form}-\eqref{eqn:I-entropy-form-er}.   For each of the random graph sequences, we leverage admissibility properties of  $\rho_h$ to first rewrite these combinatorial quantities in terms of suitable relative entropies with respect to the true law $\eta$ (see Assumption \ref{assm:generic-eta-def}),  to yield the intermediate representation for the rate functions given in Theorem \ref{thm:iid-old} below. To define the rate functions, we first define  $\calItilde_{\b, \eta} : \P(\Ginf) \to [0, \infty]$, $\b \geq 0$, by  
\begin{align}
\calItilde_{\b, \eta}(\rho) : = \begin{dcases}
H(\r{o} \| \nu) & \text{if } \beta = 0 \text{ and } \Exp_{\rho}[\dego] = 0,\\
H(\r{1} \| \eta_1) - \frac{\beta}{2} H(\SBRM{1} \| \SBetabarM) &  \\
\quad + \sum_{h \geq 2} \sqpr{H(\r{h} \| \Rstar{h} ) - \frac{\beta}{2} H(\SBRM{h} \| \SBRstarM{h})} & \text{if } \beta > 0,  \rho \in \PUB, \r{1} \ll \eta_1, \\
& \text{ and } \EDLDRinf < \infty,\\ 
\infty & \text{ otherwise},
\end{dcases}
\label{eqn:calItilde-generic}
\end{align}
where $\eta$ plays the role of the ``true law'' as stated in Assumption \ref{assm:generic-eta-def}.

\begin{theorem}
The sequence $\{U_n^{\CM}\}$ (respectively, $\{U_n^{\FE}\}$ and $\{U_n^{\ER}\}$) satisfies the LDP  on $\P(\Ginf)$ with rate function $\calItilde^{\CM}$ (respectively, $\calItilde^{\FE}$ and $\calItilde^{\ER}$),
where
\begin{align}
\calItilde^{\CM}(\rho) &:= 
\begin{dcases}
\calItilde_{\k, \etaa}(\rho) & \text{ if } \rho_{o, \text{deg}} = \alpha,\\
\infty & \text{ otherwise},
\end{dcases} \label{eqn:I-CM-iid-old} \\
\calItilde^{\FE}(\rho) &:= \calItilde_{\k, \etab{\k}}(\rho), \label{eqn:I-FE-iid-old} \\
\calItilde^{\ER}(\rho) &:= 
\begin{dcases}
\ell_\k(\b') + \calItilde_{\b', \etab{\b'}}(\rho) & \text{ if }  \b' : = \Exp_{\rho}[\dego] < \infty,\\
\infty & \text{ otherwise}.
\end{dcases}
\label{eqn:I-ER-iid-old}
\end{align}

\label{thm:iid-old}
\end{theorem}

The proof of Theorem \ref{thm:iid-old} is carried out in Section \ref{sec:iid-proof-complete}. Although $\calItilde_{\b, \eta}$ is defined in terms of differences between relative entropies, both of which could possibly be $\infty$, we show in Lemma \ref{lemma:finiteness} that  $\calItilde_{\b, \eta}$ is well defined.\\

\noindent {\bf Step 2:}  
This step is more involved and entails showing the following: 
\begin{theorem}
The following identity is satisfied: $\calI_{\b, \eta} = \calItilde_{\b, \eta}$.
\label{thm:alternate-form-theorem}
\end{theorem}
\noindent The proof of Theorem exploits various properties of relative entropy, and uses unimodularity, the mass transport principle and random tree labelings to show that the relative entropy terms appearing in \eqref{eqn:calItilde-generic} can be expressed as 
expectations under the marginal of a size-biased distribution and to express the unimodular extensions in terms of suitable conditional laws.
\begin{proof}[Proof of Theorem \ref{thm:iid}]
    The result follows from  Theorems \ref{thm:alternate-form-theorem} and \ref{thm:iid-old}, and the definitions of the rate functions in \eqref{eqn:I-CM-iid-old}-\eqref{eqn:I-ER-iid-old} and \eqref{eqn:I-CM-iid-new}-\eqref{eqn:I-ER-iid-new}.
\end{proof}

\section{Gibbs conditioning principle: Proof of Proposition \ref{prop:gibbs}}
\label{sec:gibbs-proof}
Consider the optimization problem 
\begin{align}
v^*: = \inf \pr{\frakI^\CM(\mu) : \mu \in \P(\TX), \, \langle \mu, f \rangle \geq  c}. \label{eqn:gibbs-opt}
\end{align}
Note that our setting satisfies the assumptions required for the standard Gibbs conditioning principle for large deviations \cite[Theorem 7.1]{Leo10}. 
Indeed, 
\begin{enumerate}
    \item With $\TXM : = \{(\tau, X) \in \TX : \text{deg}_\tau(o) \leq M\} \subset \TX$,  by Theorem \ref{thm:iid-nbd} and \cite[Lemma 4.1.5(b)]{DemZei98}, the sequence $\{L_n\} = \{L^\CM_n\}$ satisfies the LDP on $\P(\TXM)$ with rate function $\frakI^\CM$ restricted to $\P(\TXM)$;
    \item The map $\mu \mapsto \langle \mu, f \rangle $ is linear and continuous on  $\P(\TXM)$;
    \item The constraint set $C:=\{\mu \in \P(\TXM): \langle \mu, f \rangle \geq c\}$ satisfies
    \begin{enumerate}
        \item $\inf_{\mu \in C} \frakI^\CM(\mu) < \infty$;
        \item Let $\delta > 0$. With $C_\delta :=\{\mu \in \P(\TXM): \langle \mu, f \rangle \geq c-\delta \}$, we have $C = \cap_{\delta > 0} C_\delta$, and $\Prob(L^\CM_n \in C_\delta) > 0$ for all $n$;
        \item $C$ is contained in the interior of $C_\delta$ for all $\delta > 0$.
    \end{enumerate}
\end{enumerate}
Moreover, since $\{\mu \in \P(\TX) : \mu_{o, \text{deg}} = \alpha, \,  \langle \mu, f \rangle \geq c\}$ is closed in $\P(\TX)$, and $\frakI^\CM$ has compact level sets (which is a consequence of the LDP in Theorem \ref{thm:iid-nbd}), it follows that the infimum in \eqref{eqn:gibbs-opt} is attained. By \cite[Theorem 7.1]{Leo10}, the proof would be complete if we show that $\mu^{*,c}$ defined in \eqref{eqn:mustarc} is the unique solution to the optimization problem  \eqref{eqn:gibbs-opt}. We do this in three steps.

{\bf Step 1:} We first show that, for any admissible $\mu \in \P(\TX)$, $\langle \mu, f\rangle$ depends on $\mu$ only via the joint law of the root degree and root mark. By \eqref{eqn:SBRinf} and the admissibility of $\mu$, we have 
\begin{align}
\frac{1}{\k}  \Exp_{\mu}\sqpr{\dego \indf\{\bfX_1 = a\}} = \SBMuRM(a) = \SBMuMZero(a) = \frac{1}{\k}\Exp_{\mu}\sqpr{\dego \indf\{\bfX_o 
= a\}}.
\label{eqn:gibbs-step0}
\end{align}
Also, since the random variables $\{\bfX_1, \ldots \bfX_i\}$ given $\dego = i \geq 1$ are exchangeable, by the admissibility of $\mu$, we also have
\begin{align}
\langle \mu, f \rangle = \Exp_{\mu} \sqpr{\sum_{v \in N_o(\TB)} h(\bfX_v)} =  \Exp_{\mu} \sqpr{\dego h(\bfX_1)}  = \k \langle \SBMuRM, h \rangle = \k \langle \SBMuMZero, h \rangle  = \Exp_{\mu} \sqpr{\dego h(\bfX_o)}.
\label{eqn:gibbs-constraint-step1}
\end{align}
This proves the claim.

{\bf Step 2:} Show that $v^*$ in \eqref{eqn:gibbs-opt} solves the following simpler constrained optimization problem:
\begin{align}
v^* = \inf \biggr\{ H(\gamma \| \alpha \otimes \nu) : \gamma \in \P([M_0] \times \X), \,  & \sum_{x \in \X} \gamma(n,x) = \alpha(n) \, \forall n \in [M_0], \, \sum_{n=0}^M \sum_{x \in \X} n h(x) \gamma(n, x) \geq c\biggr\}, \label{eqn:gibbs-new-inf}
\end{align}
where $[M_0] = \{0,1, \ldots, M\}$. 

For any $\mu \in \P(\TX)$, we introduce the following notation: let $\mu_{\mathbf{d}_o, \bfX_o}$ denote the joint law of $(\dego, \bfX_o) = (\mathbf{d}_o, \bfX_o)$ when $(\TB, \bfX)$ has law $\mu$, let $\mu_{\bfX_o | \, \mathbf{d}_o}$ denote the conditional law of $\bfX_o$ given $\mathbf{d}_o$, and let $\mu_{\bfX_1, \ldots, \bfX_{\mathbf{d}_o} |\,  \mathbf{d}_o, \bfX_o}$ denote the conditional joint law of $(\bfX_1, \ldots, \bfX_{\mathbf{d}_o})$ given $(\mathbf{d}_o, \bfX_o)$. 
Given an admissible $\mu$, we now define an admissible $\psi$ such that $\psi_{\mathbf{d}_o, \bfX_o} = \mu_{\mathbf{d}_o, \bfX_o}$ and $\frakI^\CM(\mu) \geq \frakI^\CM(\psi)$. Define
\begin{align}
\psi(\tau, X) := \alpha(\text{deg}_\tau(o)) \, \mu_{\bfX_o | \,  \mathbf{d}_o}(X_o | \, \text{deg}_\tau(o)) \prod_{v \in N_o(\tau)} \SBMuRM(X_v), \quad (\tau, X) \in \TX.\label{eqn:gibbs-mu-underline}
\end{align}
We first verify that $\psi$ is admissible. Note that $\Exp_{\psi}[\mathbf{d}_o] = \k$ since the degree distribution of $\psi$ is $\alpha$. For $a, b \in \X$, using the fact that $\Exp_{\psi}\sqpr{\indf\{\bfX_1 = a\} | \mathbf{d}_o} = \SBMuRM(a)$ on the set $\mathbf{d}_o \geq 1$,  \eqref{eqn:gibbs-step0}, and the identity  $\psi_{\mathbf{d}_o, \bfX_o} = \mu_{\mathbf{d}_o, \bfX_o}$,  we have

\begin{align*}
\bar{\psi}(a,b) & = \frac{1}{\k} \Exp_{\psi}\sqpr{\mathbf{d}_o \indf \{\bfX_1 = a, \bfX_o = b\}} \\ 
& = \frac{1}{\k} \Exp_{\psi} \sqpr{ \Exp_{\psi} \sqpr{\indf\{\bfX_o = b\} | \, \mathbf{d}_o } \,  \mathbf{d}_o \,   \Exp_{\psi} \sqpr{\indf\{\bfX_1 = a\} | \, \mathbf{d}_o} } \\
& =  \frac{\SBMuRM(a)}{\k} \Exp_{\psi} \sqpr{\indf\{\bfX_o = b\} \mathbf{d}_o } \\
& = \SBMuRM(a) \SBMuRM(b). \numberthis \label{gibbs:step3}    
\end{align*}
It follows that $\psi$ is admissible. Next, we show that $\frakI^\CM(\mu) \geq \frakI^\CM(\psi)$. 
Using the chain rule and non-negativity of relative entropy, we obtain
\begin{align*}
    H(\mu \| \MP) & = H(\mu_{\mathbf{d}_o, \bfX_o} \| \MP_{\mathbf{d}_o, \bfX_o}) \\
    &\qquad \qquad + \int H(\mu_{\bfX_1, \ldots, \bfX_{\mathbf{d}_o} | \mathbf{d}_o, \bfX_o} \| \MP_{\bfX_1, \ldots, \bfX_{\mathbf{d}_o} | \mathbf{d}_o, \bfX_o}) \, d \mu_{\mathbf{d}_o, \bfX_o} \\
    & \geq H(\mu_{\mathbf{d}_o, \bfX_o} \| \MP_{\mathbf{d}_o, \bfX_o}). \numberthis \label{eqn:gibbs-step4}
\end{align*}
An exact analogous argument yields
\begin{align*}
H(\mu \| \MCP)
& \geq H(\mu_{\mathbf{d}_o, \bfX_o} \| \MCP_{\mathbf{d}_o, \bfX_o}). \numberthis \label{eqn:gibbs-step5}
\end{align*}
By \eqref{eqn:gibbs-mu-underline}, \eqref{eqn:RP} and \eqref{eqn:RCP},  the conditional joint law of the leaf marks given the degree and the root mark is the same under $\psi$, $\widecheck{\psi}$ and $\widehat{\psi}$.
Also, note that $\psi_{\mathbf{d}_o, \bfX_o} = \mu_{\mathbf{d}_o, \bfX_o}$, and $\MP_{\mathbf{d}_o, \bfX_o} = \MCP_{\mathbf{d}_o, \bfX_o} = \alpha \otimes \nu$. Together with \eqref{eqn:gibbs-step4},  \eqref{eqn:gibbs-step5}, and the definition of $\frakI^\CM$ in \eqref{eqn:J-CM-iid},  it follows that  
\begin{align*}
\frakI^\CM(\mu) \geq \frakI^\CM(\psi) = H(\psi_{\mathbf{d}_o, \bfX_o} \| \alpha \otimes \nu).
\end{align*}
Using \eqref{eqn:gibbs-constraint-step1}, \eqref{eqn:gibbs-mu-underline} and the fact that $\psi_{\mathbf{d}_o, \bfX_o} = \mu_{\mathbf{d}_o, \bfX_o}$, \eqref{eqn:gibbs-new-inf} follows.

{\bf Step 3:} We finally solve the finite-dimensional constrained optimization problem in \eqref{eqn:gibbs-new-inf}. First, note that the constraint set is closed and convex. Since $H(\cdot \| \alpha \otimes \nu)$ is strictly convex and has compact lower level sets, it follows that there exists a unique minimizer, which we denote by $\gamma$. The necessary conditions for optimality imply that there exist $\lambda  \geq 0 $ and $\lambda' \in \R$ such that
\begin{align}
1 + \log\cpr{\frac{\gamma(n, x)}{\alpha(n) \nu(x)}} - \lambda' - \lambda n h(x) & = 0 \quad \text{ for all } n \in \{0,1,\ldots, M\}, \, x \in \X,  \label{eqn:gibbs-lagrange-step1}\\
\lambda \cpr{c - \sum_{n =0}^M \sum_{x \in \X} n h(x) \gamma(n,x) } & = 0.  \label{eqn:gibbs-lagrange-step2}
\end{align}
From \eqref{eqn:gibbs-lagrange-step1}, the constraint $\sum_{x \in \X} \gamma(n,x) = \alpha(n)$ for all $n \in \{0,1, \ldots, M\}$ implies that
\begin{align*}
\exp\{1 - \lambda' \} = \sum_{x \in \X} \nu(x) \exp\{\lambda n h(x)\}.
\end{align*}
Using this, \eqref{eqn:gibbs-lagrange-step1} reduces to
\begin{align}
\gamma(n,x) = \alpha(n) \frac{\nu(x) \exp\{\lambda n h(x)\}}{\sum_{y \in \X} \nu(y) \exp\{\lambda n h(y)\} }, \quad n \in \{0,1,\ldots, M\}, \, x \in \X. \label{eqn:gibbs-opt-form}
\end{align}
From \eqref{eqn:gibbs-lagrange-step2}, it follows that either $\lambda = 0$ or $\sum_{n,x} n h(x) \gamma(n,x) = c$. However, $\lambda = 0$ is not possible since it would imply that $\gamma = \alpha \otimes \nu$ is the joint law of the degree and the root mark under $\etaa_1$. But $\langle \etaa_1, f \rangle < c$ by assumption, hence the condition $\sum_{n=0}^M \sum_{x \in \X} n h(x) \gamma(n, x) \geq c$ is violated. Therefore,  we must have $\sum_{n,x} n h(x) \gamma(n,x) = c$, that is,
\begin{align*}
c =  \frac{\sum_{n = 0}^M \sum_{x \in \X} n \alpha(n)\nu(x)  h(x) \exp\{\lambda n h(x)\}}{\sum_{y \in \X} \nu(y) \exp\{\lambda n h(y)\} } =:g(\lambda).
\end{align*}
To show the existence of $\lambda$ satisfying the above, note that $g$ is continuous in $\lambda$. We have $g(0) = \Exp_{\etaa_1}[f] < c$ by assumption. Also, $g(\lambda) \to \k \max_{x \in \X}h(x) > c$ as $\lambda \to \infty$, whence there exists $\lambda > 0$ such that the first equality in the previous display holds. It follows that the optimizer of \eqref{eqn:gibbs-new-inf} is given by \eqref{eqn:gibbs-opt-form} where $\lambda > 0$ is chosen so that $\sum_{n, x} n h(x) \gamma(n, x) = c$. The uniqueness of the optimizer $\gamma$ follows from the strict convexity of $H(\cdot \| \alpha \otimes \nu)$. Finally, given this optimizer $\gamma$ for the problem in \eqref{eqn:gibbs-new-inf}, the unique optimizer for the optimization problem in \eqref{eqn:gibbs-opt} can be constructed via \eqref{eqn:gibbs-mu-underline}. This completes the proof.
\qed

\section{Proof of the alternative representation}
\label{sec:alternate-form}
In this section, we prove Theorem \ref{thm:alternate-form-theorem}. 
We show equality of the first summands of $\calI_{\b, \eta}$ and $\calItilde_{\b, \eta}$ in Section \ref{sec:alternate-form-1} and equality of each of the subsequent summands in Section \ref{sec:alternate-form-h}. The proofs rely on some preliminary results first established in Section \ref{sec:alternate-form:prelims}. Throughout this section, we fix $\b \geq 0$ and the ``true law'' $\eta \in \PUB$ as defined in Assumption \ref{assm:generic-eta-def}.
\subsection{Labeled trees}
\label{sec:alternate-form:prelims}
We start with two simple observations about $\P(\Tinf)$ elements and the true law $\eta$.

\begin{remark}[Exchangeability]
Let $h \in \N$ and let  $\r{h} \in \P(\T{h})$. Let $\TB$ be a $\T{h}$ random element with law $\r{h}$, that is, $\Law(\TB) = \r{h}$. Using the Ulam-Harris-Neveu labeling described in Section \ref{sec:labeling}, one has the following exchangeability properties: 
\begin{enumerate}
    \item Conditioned on $\NO = i \geq 1$,  the random variables $\{(\TBOV_{h-1}, \TBVO_{h-1}),  v \in [i])\}$ are exchangeable.
    \item For any $v \in N_o(\TB)$, the conditional marginal law of $(\TBOV_{h-1}, \TBVO_{h-1})$ given $\NO = i \geq 1$ is the same as the conditional marginal law of $(\TBOOne_{h-1}, \TBOneO_{h-1})$ given $\NO = i$. That is,
    \begin{align*}
        \Law\cpr{(\TBOV_{h-1}, \TBVO_{h-1}) \mid \NO = i } = \Law((\TBOOne_{h-1}, \TBOneO_{h-1}) \mid \NO = i).
    \end{align*}
Therefore, if $f$ is a measurable function on $\Y \times \T{h-1}$ with suitable integrability properties, then 
\begin{align}
\Exp_{\r{h}} &  \sqpr{\sum_{v \in N_o(\TB)} f\cpr{\TBOV_{h-1}, \TBVO_{h-1}} \biggr| \dego} \nonumber \\
& = \Exp_{\r{h}} \sqpr{\dego f\cpr{ \TBOOne_{h-1}, \TBOneO_{h-1}}  \biggr| \dego}. \label{eqn:exchangeability}
\end{align}
\end{enumerate} 
\label{remark:exchageability}
\end{remark}

\begin{remark}
Suppose $\Law(\TB) = \eta$, where $\eta$ satisfies Assumption \ref{assm:generic-eta-def}.
\begin{enumerate}
    \item The random variables $\{\TBOV_0, v \in [i]\}$ are conditionally independent given $\NO = i \geq 1$.
    \item The random variables $\{\TBOV_0, v \in N_o(\TB)\}$ are conditionally independent given $\{\TBVO_0, v \in N_o(\TB)\}$. Also, for each $v \in N_o(\TB)$, the random variable $\TBOV_0$ is conditionally independent of $\{\TBUO_0, u \in N_o(\TB), u \neq v\}$ given $\TBVO_0$. However, note that the random variables $\TBOV_0$ and $\TBVO_0$ need not be independent since the edge mark on the directed edges $(o , v)$ and $(v , o)$ could be correlated.
\end{enumerate}
\label{remark:eta-indep}
\end{remark}

\subsection{Alternative form of the first term of $\calItilde_{\b, \eta}$}
\label{sec:alternate-form-1}
Throughout this section, we fix $\b > 0$ and  $\rho \in \PUB$. We start with a simple observation about the measure $\RP_1$ defined in \eqref{eqn:RP}.
\begin{lemma}
Let $\b > 0$ and $\rho \in \PUB$. If $\r{1} \ll \eta_1$, then $\SBRM{1} \ll \SBetabarM$ and  $H(\SBRM{1} \| \SBetabarM) < \infty$.
\label{lemma:pirllpieta}
\end{lemma}
\begin{proof}
For the first assertion, it follows immediately from 
the  definition in \eqref{eqn:SBRinf} that  $\r{1} (\cdot \mid \NO \geq 1)$ and $\eta_1(\cdot \mid \NO \geq 1)$ are equivalent to $ \SBR{1}$ and  $\SBeta{1}$, respectively. Since $\r{1} \ll \eta_1$ by assumption, we obtain
\begin{align*}
\SBRM{1} \sim \r{1}^{1,o}(\cdot \mid \NO \geq 1) \ll \eta_1^{1,o}(\cdot \mid \NO \geq 1) \sim \SBetabarM,
\end{align*}
where `$\sim$' denotes the equivalence of measures.

The second assertion is an immediate consequence of the first assertion and the fact that $(\X \times \Y) \times (\X \times \Y)$ is a finite set.
\end{proof}

We now make the important observation that the measures  $\RP_1$ and $\RCP_1$ defined in Remark \ref{rmk:rh-rstarh} by \eqref{eqn:RP} and \eqref{eqn:RCP}, respectively,  are indeed probability measures and satisfy certain absolute continuity properties. The proofs are deferred to Appendix \ref{appendix:prob-measure-h=1}.
\begin{lemma} 
\label{lemma:prob-measures-h=1}
The measures $\RP_1$ and $\RCP_1$ are probability measures on $\T{1}$. Furthermore, $\r{1} \ll \RP_1$ and $\r{1} \ll \RCP_1$.
\end{lemma}

We now  express $\b H(\SBRM{1} \| \SBetabarM)$ as an expectation under $\r{1}$.
\begin{lemma}  Suppose that 
$\r{1} \ll \eta_1$. Then
\begin{align}
\b H(\SBRM{1} \| \SBetabarM) = \Exp_{\r{1}}\sqpr{ \sum_{v \in N_o(\TB)}  \log \cpr{\frac{d\SBRM{1}}{d\SBetabarM} \cpr{\TBOV_0, \TBVO_0}}}.
\label{eqn:secondterm-h=1}
\end{align}
\label{lemma:secondterm-h=1}
\end{lemma}
\begin{proof}
Since $\r{1} \ll \eta_1$, we have $\SBRM{1} \ll \SBetabarM$ by Lemma \ref{lemma:pirllpieta}. Therefore,  using  \eqref{eqn:SBRinf},  $\SBRinf(\NO \geq 1) = 1$ and the identity $\b = \Exp_{\rho}\sqpr{\NO}$,
\begin{align*}
\b H(\SBRM{1} \| \SBetabarM) & = \b \Exp_{\SBRM{1}}\sqpr{\log \cpr{\frac{d\SBRM{1}}{d\SBetabarM}}} \\
& = \b \Exp_{\SBR{1}}\sqpr{\log \cpr{\frac{d\SBRM{1}}{d\SBetabarM}\cpr{\TBOOne_0, \TBOneO_0}} \indf\{\NO \geq 1\} } \\
& = \b \Exp_{\r{1}} \sqpr{\frac{d\SBR{1}}{d\r{1}}\cpr{\TXY}\log \cpr{\frac{d\SBRM{1}}{d\SBetabarM}\cpr{\TBOOne_0, \TBOneO_0} } \indf\{\NO \geq 1\}}\\
&  =\Exp_{\r{1}} \sqpr{  \NO \log \cpr{\frac{d\SBRM{1}}{d\SBetaM}\cpr{\TBOOne_0, \TBOneO_0}} }. \numberthis \label{eqn:second-term-h=1-step1}
\end{align*} 
By Lemma \ref{lemma:pirllpieta},  we also have $H(\SBRM{1} \| \SBetabarM) < \infty$.  Therefore, conditioning on $\NO$ and using \eqref{eqn:exchangeability} with $h=1$ and $f(\cdot, \cdot) =  \log \cpr{\frac{d\SBRM{1}}{d\SBetaMI} (\cdot, \cdot)}$, we obtain
\begin{align*}
\b H(\SBRM{1} \| \SBetabarM) & = \Exp_{\r{1}} \sqpr{  \sum_{v \in N_o(\TB)} \Exp_{\r{1}}\sqpr{  \log \cpr{\frac{d\SBRM{1}}{d\SBetaMI}\cpr{\TBOV_0, \TBVO_0}} \biggr|\NO }   },
\end{align*}
which is equal to the right-hand side of  \eqref{eqn:secondterm-h=1}.
\end{proof}

We now establish a key result, which expresses
$H(\SBRM{1} \| \SBetabarM)$ in terms of  $\RP_1$ and $\RCP_1$.

\begin{proposition}  Suppose that  $\r{1} \ll \eta_1$. Then we have
\begin{align}
\b H(\SBRM{1} \| \SBetabarM) = \Exp_{\r{1}} \sqpr{\log\cpr{\frac{d\RP_1}{d\eta_1}(\TXY)  }} + \Exp_{\r{1}}\sqpr{\log\cpr{\frac{d\RCP_1}{d\eta_1}(\TXY)  }}.
\label{eqn:extra-terms-h=1}
\end{align}
\label{lemma:extra-terms-h=1}
\end{proposition}
\begin{proof}
We first show that
\begin{align}
\Exp_{\r{1}}\sqpr{ \sum_{v \in N_o(\TB)} \log \cpr{\frac{d\SBRMOne{1}}{d\SBetabarMOne}\cpr{\TBVO_0}}} < \infty.
\label{eqn:extra-terms-h=1-step0-finiteness}
\end{align}
By Lemma \ref{lemma:pirllpieta}, we have  $H(\SBRM{1} \| \SBetabarM) < \infty$. Using the chain rule and the non-negativity of relative entropy, it follows that  $H(\SBRMZero{1} \| \SBetabarMZero) < \infty.$  Using the admissibility of $\r{1}$ and $\eta_1$ (which follow from the unimodularity of $\rho$ (resp. $\eta$); see Remark \ref{remark:admissibility}), and  \eqref{eqn:SBRinf}, we have
\begin{align*}
\b H(\SBRMZero{1} \| \SBetabarMZero) & = \b \Exp_{\SBRMZero{1}}\sqpr{ \log \cpr{\frac{d\SBRMZero{1}}{d\SBetabarMZero}  }} \\
& = \b \Exp_{\SBRM{1}}\sqpr{ \log \cpr{\frac{d\SBRMOne{1}}{d\SBetabarMOne}\cpr{\TBOneO_0} }} \\
& = \Exp_{\r{1}}\sqpr{ \NO  \log \cpr{\frac{d\SBRMOne{1}}{d\SBetabarMOne}\cpr{\TBOneO_0} }}.
\end{align*}
Since $H(\SBRMZero{1} \| \SBetabarMZero) < \infty$ and $\NO \geq 0$, the above shows that for $a \in \{+, -\}$,
\begin{align}
\Exp_{\r{1}}\sqpr{ \NO  \cpr{\log \cpr{\frac{d\SBRMOne{1}}{d\SBetabarMOne}\cpr{\TBOneO_0} }}^a} & < \infty. \quad \label{eqn:extra-terms-h=1-finiteness-0}
\end{align}
Using \eqref{eqn:exchangeability} with $h=1$ and $f(\tau, \cdot) = \cpr{ \log \cpr{\frac{d\SBRMOne{1}}{d\SBetabarMOne}(\cdot) }}^a$ for all $\tau \in \Y \times \T{0}$, one obtains \eqref{eqn:extra-terms-h=1-step0-finiteness}.

Next, we show that 
\begin{align}
\Exp_{\r{1}}\sqpr{ \sum_{v \in N_o(\TB)} \log \cpr{\frac{d\SBRMOne{1}}{d\SBetabarMOne}\cpr{\TBVO_0}}}  = \Exp_{\r{1}}\sqpr{ \sum_{v \in N_o(\TB)} \log \cpr{\frac{d\SBRMOne{1}}{d\SBetabarMOne}\cpr{\TBOV_0}}}.
\label{eqn:h=1-proof-exchange-unimod}
\end{align}
Recall that $\rho$ is unimodular. For $a \in \{+, -\}$, define the function $f^a: \Tdstar[\X; \Y] \to \R_+$ by
\begin{align*}
f^a(\tau, o_1, o_2) := \indf\{o_1 \in N_{o_2}(\tau)\} \cpr{\log \cpr{\frac{d\SBRMOne{1}}{d\SBetabarMOne} \cpr{\tau(o_1 \setminus o_2)_0}}}^a,
\end{align*}
for $(\tau, o_1, o_2) \in \Tdstar[\X ; \Y]$. Here, $\Tdstar[\X; \Y] \subset \Gdstar[\X; \Y]$ denotes the set of marked doubly rooted trees equipped with the natural local topology. Applying \eqref{eqn:unimod} with $f = f^a$, we obtain
\begin{align*}
\Exp_{\r{1}}\sqpr{ \sum_{v \in N_o(\TB)} \cpr{\log \cpr{\frac{d\SBRMOne{1}}{d\SBetabarMOne}\cpr{\TBVO_0}}}^a}   = \Exp_{\r{1}}\sqpr{ \sum_{v \in N_o(\TB)} \cpr{\log \cpr{\frac{d\SBRMOne{1}}{d\SBetabarMOne}\cpr{\TBOV_0}}}^a}.
\end{align*}
Combining the last display for $a \in \{+, -\}$ with \eqref{eqn:extra-terms-h=1-step0-finiteness}, we conclude that  \eqref{eqn:h=1-proof-exchange-unimod} holds.

Next, by the admissibility of $\r{1}$ and $\eta_1$, for any $\tau, \tau' \in \X \times \Y$, we have
\begin{align*}
\frac{d\SBRM{1}}{d\SBetabarM}(\tau, \tau^\prime) & = \frac{d\SBRMZero{1}}{d\SBetabarMZero}(\tau') \times  \frac{d\SBRMOneGivenZero{1}}{d\SBetabarMOneGivenZero}(\tau | \tau') \\
& = \frac{d\SBRMOne{1}}{d\SBetabarMOne}(\tau') \times  \frac{d\SBRMOneGivenZero{1}}{d\SBetabarMOneGivenZero}(\tau | \tau').
\end{align*}
Together with Lemma \ref{lemma:secondterm-h=1} (which is applicable due to our assumption $\r{1} \ll \eta_1$),  \eqref{eqn:extra-terms-h=1-step0-finiteness}, the fact that $H(\SBRM{1} \| \SBetabarM) < \infty$ and \eqref{eqn:h=1-proof-exchange-unimod}, this implies that 
\begin{align*}
\b H(\SBRM{1} \| \SBetabarM) & = \Exp_{\r{1}}\sqpr{ \sum_{v \in N_o(\TB)} \log \cpr{\frac{d\SBRMOne{1}}{d\SBetabarMOne}\cpr{\TBVO_0}}} + \Exp_{\r{1}} \sqpr{ \sum_{v \in N_o(\TB)} \log \cpr{ \frac{d\SBRMOneGivenZero{1}}{d\SBetabarMOneGivenZero}(\TBOV_0 | \TBVO_0)}} \\
& =  \Exp_{\r{1}}\sqpr{ \sum_{v \in N_o(\TB)} \log \cpr{\frac{d\SBRMOne{1}}{d\SBetabarMOne}\cpr{\TBOV_0}}} + \Exp_{\r{1}} \sqpr{ \sum_{v \in N_o(\TB)} \log \cpr{ \frac{d\SBRMOneGivenZero{1}}{d\SBetabarMOneGivenZero}(\TBOV_0 | \TBVO_0)}}. \numberthis 
\label{eqn:extra-terms-h=1-step5}
\end{align*}
Also, the definition of $\RP_1$ in \eqref{eqn:RP} implies 
\begin{align}
\Exp_{\r{1}} \sqpr{\log\cpr{\frac{d\RP_1}{d\eta_1}(\TXY)  }} = \Exp_{\r{1}} \sqpr{\sum_{v \in N_o(\TB)}\log\cpr{\frac{d\SBRMOne{1}}{d\SBetabarMOne}(\TBOV_0)  }}.\label{eqn:extra-terms-h=1-step0}
\end{align}
Similarly, the definition of $\RCP_1$ in \eqref{eqn:RCP} implies 
\begin{align}
\Exp_{\r{1}}\sqpr{\log\cpr{\frac{d\RCP_1}{d\eta_1}(\TXY)  }} = \Exp_{\r{1}} \sqpr{ \sum_{v \in N_o(\TB)} \log \cpr{\frac{d\SBRMOneGivenZero{1}}{d\SBetabarMOneGivenZero}(\TBOV_0 | \TBVO_0)}}. \label{eqn:extra-terms-h=1-step2}
\end{align}
Thus, \eqref{eqn:extra-terms-h=1} follows from \eqref{eqn:extra-terms-h=1-step5}, \eqref{eqn:extra-terms-h=1-step0} and \eqref{eqn:extra-terms-h=1-step2}.
\end{proof}

Combining the previous  results, we now obtain the main result of this subsection.

\begin{proposition}[Alternative form of the first term of $\calItilde_{\b, \eta}$]
Suppose that $\r{1} \ll \eta_1$.   Then
\begin{align}
H&(\r{1} \| \eta_1) - \frac{\b}{2} H(\SBRM{1} \| \SBetabarM) =  \frac{1}{2}\sqpr{H(\r{1} \| \RP_1) + H(\r{1} \| \RCP_1)}.
\label{eqn:alternate-form-1}
\end{align}
\label{prop:alternate-form-1}
\end{proposition}
\begin{proof}
Since $\r{1} \ll \eta_1$, we have $H(\SBRM{1} \| \SBetabarM) < \infty$ by Lemma \ref{lemma:pirllpieta}. 
Thus, by \eqref{eqn:RP} and the second assertion of Lemma \ref{lemma:prob-measures-h=1}, we have  $\RP_1 \ll \eta_1$ and   $\r{1} \ll \RP_1$, respectively. Hence,
$
\frac{d\r{1}}{d\eta_1} = \frac{d\r{1}}{d\RP_1} \times \frac{d\RP_1}{d\eta_1}.
$
This implies
\begin{align*}
H(\r{1} \| \eta_1)  & = \Exp_{\r{1}} \sqpr{\log \cpr{\frac{d\r{1}}{d\RP_1}}}  + \Exp_{\r{1}} \sqpr{\log \cpr{\frac{d\RP_1}{d\eta_1}}} =  H(\r{1} \| \RP_1) + \Exp_{\r{1}} \sqpr{\log \cpr{\frac{d\RP_1}{d\eta_1}}}. \numberthis \label{eqn:h=1-thm-term1}
\end{align*}
Similarly, by \eqref{eqn:RCP} and the second assertion of Lemma \ref{lemma:prob-measures-h=1}, we have $\RCP_1 \ll \eta_1$ and  $\r{1} \ll \RCP_1$.  Hence,
$
\frac{d\r{1}}{d\eta_1} = \frac{d\r{1}}{d\RCP_1} \times \frac{d\RCP_1}{d\eta_1},
$
which implies
\begin{align*}
H(\r{1} \| \eta_1)  & = 
H(\r{1} \| \RCP_1) + \Exp_{\r{1}} \sqpr{\log \cpr{\frac{d\RCP_1}{d\eta_1}}}. \numberthis \label{eqn:h=1-thm-term2}
\end{align*}
Summing \eqref{eqn:h=1-thm-term1} and  \eqref{eqn:h=1-thm-term2}, and using \eqref{eqn:extra-terms-h=1}, we obtain \eqref{eqn:alternate-form-1}.
\end{proof}

\subsection{Alternative form of the remaining terms of $\calItilde_{\b, \eta}$}
\label{sec:alternate-form-h}
We start with two preliminary lemmas. Lemma \ref{lemma:finiteness} shows that the rate function $\calItilde_{\b, \eta}$ in \eqref{eqn:calItilde-generic} is well defined, and Lemma \ref{lemma:preliminary-prob-prop} shows that $\RCP_h$ is indeed a probability measure and establishes various relations between $\r{h}$, $\SBR{h}$, $\Rstar{h}$ and $\RCP_h$. Their proofs are involved and technical, and hence
deferred to Appendices  \ref{appendix:lemma-finiteness} and \ref{appendix:lemma-Rstarabscont}, respectively.

\begin{lemma}[Finiteness]
Let $\b > 0$ and let $\rho \in \PUB$. If  $\EDLDRinf < \infty$, then the following properties hold:
\begin{enumerate}[label={(P\arabic*)}]
\item $H(\r{h}) < \infty$ for all $h \in \N$, \label{item:HRH}
\item $H(\r{h} \| \Rstar{h}) < \infty$ for all $h \geq 2$, \label{item:HRHD}
\item $H(\SBRM{h}) < \infty$ for all $h \in \N$, \label{item:HPIRH}
\item $H(\SBRM{h} \| \SBRstarM{h}) < \infty$ for all $h \geq 2$. \label{item:HPIRHD}
\end{enumerate}
\label{lemma:finiteness}
\end{lemma}

\begin{lemma}
\label{lemma:preliminary-prob-prop}
Let $\b > 0$ and let $\rho \in \PUB$. The following properties hold for $h \geq 2$:
\begin{enumerate}[label={(P\arabic*)}]
\item $\r{h} \ll \Rstar{h}$ and $\SBRM{h} \ll \SBRstarM{h}$. \label{lemma:Rstarabscont}
\item  $\SBRMOne{h} = \SBRstarMOne{h}$. \label{lemma:PIRRstarMarginal}
\item $\RCP_h$ is a probability measure.  \label{lemma:RCPh-prob-measure}
\item $\r{h} \ll \RCP_h$. \label{lemma:RCPh-abs-cont}
\end{enumerate}
\end{lemma}

For the rest of this section we fix $\b > 0$,  $\rho \in \PUB$ and $h \geq 2$.
We now express $H(\SBRM{h} \| \SBRstarM{h})$ in the definition of $\calItilde_{\b, \eta}$ as an expectation with respect to  $\r{h}$. 

\begin{lemma}
If  $\EDLDRinf < \infty$, then it follows that
\begin{align}
\b H(\SBRM{h} \| \SBRstarM{h}) = \Exp_{\r{h}} & \sqpr{\sum_{v \in N_o(\TB)} \log \cpr{\frac{d\SBRM{h}}{d\SBRstarM{h}}(\TBOV_{h-1}, \TBVO_{h-1}}}. \label{eqn:PIRPIRstar-simplify}
\end{align}
\label{lemma:PIRPIRstar-simplify}
\end{lemma}
\begin{proof}
By \ref{lemma:Rstarabscont} of Lemma \ref{lemma:preliminary-prob-prop}, we have  $\r{h} \ll \Rstar{h}$ and  $\SBRM{h} \ll \SBRstarM{h}$. Together with the fact that $\SBR{h}(\NO \geq 1)$,
\begin{align*}
\b H(\SBRM{h} \| \SBRstarM{h}) & = \b \Exp_{\SBRM{h}} \sqpr{\log\cpr{\frac{d\SBRM{h}}{d\SBRstarM{h}}}} \\
& = \b \Exp_{\SBR{h}} \sqpr{\log\cpr{\frac{d\SBRM{h}}{d\SBRstarM{h}}(\TBOOne_{h-1},\TBOneO_{h-1})} \indf\{\NO \geq 1\}}.
\end{align*}
Substituting the expression for  $\SBR{h}$ from \eqref{eqn:SBRinf} into the above display yields
\begin{align*}
\b H(\SBRM{h} \| \SBRstarM{h}) 
& = \Exp_{\r{h}} \sqpr{ \NO  \log\cpr{\frac{d\SBRM{h}}{d\SBRstarM{h}}(\TBOOne_{h-1},\TBOneO_{h-1})}}. \numberthis \label{eqn:PIRPIRstar-simplify-step1}
\end{align*}
Also, since $\EDLDRinf < \infty$, it follows from Lemma \ref{lemma:finiteness} that  $H(\SBRM{h} \| \SBRstarM{h}) < \infty$. Therefore, from \eqref{eqn:PIRPIRstar-simplify-step1}, conditioning on $\dego$ and using \eqref{eqn:exchangeability} with $f(\cdot, \cdot) = \log \cpr{\frac{d\SBRM{h}}{d\SBRstarM{h}}(\cdot, \cdot)}$, we obtain \eqref{eqn:PIRPIRstar-simplify}.
\end{proof}

We are now in a position to establish the main result of this subsection.
\begin{proposition}[Alternative form of the remaining terms of $\calItilde_{\b, \eta}$]
\label{prop:alternate-form-h}
Suppose that
\[ \EDLDRinf < \infty. \]
Then for $h \geq 2$, we have
\begin{align}
H(\r{h} \| \Rstar{h}) - \frac{\b}{2} H(\SBRM{h} \| \SBRstarM{h}) = \frac{1}{2}\sqpr{H(\r{h} \| \RP_h) + H(\r{h} \| \RCP_h)}.
\label{eqn:alternate-form-h}
\end{align}
\end{proposition}
\begin{proof}
Since $\EDLDRinf < \infty$, it follows from \ref{item:HRHD} and \ref{item:HPIRHD} of Lemma \ref{lemma:finiteness} that  $H(\r{h} \| \Rstar{h}) < \infty$ and $H(\SBRM{h} \| \SBRstarM{h}) < \infty$. We first write $H(\r{h} \| \Rstar{h})$ in terms of $H(\r{h} \| \RCP_h)$. Since $\rho \in \PUB$, by \ref{lemma:Rstarabscont} and  \ref{lemma:RCPh-abs-cont} of Lemma \ref{lemma:preliminary-prob-prop}, it follows that $\r{h} \ll \Rstar{h}$ and  $\r{h} \ll \RCP_h$, respectively. Moreover, $\RCP_h \ll \Rstar{h}$  by \eqref{eqn:RCPh}. Hence, writing 
$
\frac{d\r{h}}{d\Rstar{h}} = \frac{d\r{h}}{d\RCP_h} \times \frac{d\RCP_h}{d\Rstar{h}}, 
$
we have 
\begin{align*}
H(\r{h} \| \Rstar{h}) &=  \Exp_{\r{h}} \sqpr{\log\cpr{\frac{d\r{h}}{d\RCP_h} \times \frac{d\RCP_h}{d\Rstar{h}}}}  = H(\r{h} \| \RCP_h) + \Exp_{\r{h}} \sqpr{\log \cpr{\frac{d\RCP_h}{d\Rstar{h}}}}. \numberthis \label{thm:atlernate-form-h-step1}
\end{align*}
Using \eqref{eqn:RCPh} to rewrite $\log \cpr{\frac{d\RCP_h}{d\Rstar{h}} \cpr{\tau}}$ for $\tau \in \T{h}$, and
invoking Lemma  \ref{lemma:PIRPIRstar-simplify}  (which is justified by the assumption $\EDLDRinf < \infty$), we obtain
\begin{align*}
H(\r{h} \| \Rstar{h}) &= H(\r{h} \| \RCP_h) + \Exp_{\r{h}} \sqpr{\sum_{v \in N_o(\TB)} \log \cpr{\frac{d\SBRM{h}}{d\SBRstarM{h}}(\TBOV_{h-1}, \TBVO_{h-1})}}.\\
& =  H(\r{h} \| \RCP_h) + \b H(\SBRM{h} \| \SBRstarM{h}). \numberthis \label{thm:atlernate-form-h-step2}
\end{align*}
On the other hand, since $\RP_h = \Rstar{h}$ from \eqref{eqn:RPh-def}, we have  
$
H(\r{h} \| \Rstar{h}) = H(\r{h} \| \RP_h).
$
Thus, \eqref{eqn:alternate-form-h} follows by adding $H(\r{h} \| \Rstar{h}) - \b H(\SBRM{h} \| \SBRstarM{h})$ to both sides of \eqref{thm:atlernate-form-h-step2} and dividing by $2$.
\end{proof}

\subsection{Proof of Theorem \ref{thm:alternate-form-theorem}}
We fix $\b \geq 0$ and assume that  $\eta$ is either $\etaa$ or $\etab{\b}$. We start with a preliminary lemma.

\begin{lemma} 
\label{lemma:EDLDinf-cond}
Let $\rho \in \PUB$. If $\EDLDRinf = \infty$,  then 
\begin{align*}
H(\rho_{o, \text{deg}} \| \etab{\b}_{o, \text{deg}}) = \infty.
\end{align*}
\end{lemma}	
\begin{proof}
Let $\rho \in \PUB$ be such that $\EDLDRinf = \infty$.  This, in particular, implies that $0 < \b < \infty$. Let $\Geom$ denote the geometric distribution with parameter $1/2$. It follows that
\begin{align*}
0 \leq H(\rho_{o, \text{deg}} \| \Geom)   =  \sum_{k \in \NZ} \rho_{o, \text{deg}}(k) \log \cpr{\frac{\rho_{o, \text{deg}}( k)}{\Geom(k)}}
& = -H(\rho_{o, \text{deg}}) + \log 2 \cdot  \EDRinf,
\end{align*}
where $H$ is the Shannon entropy, and hence
\begin{align}
H(\rho_{o, \text{deg}}) \leq \log 2 \cdot  \EDRinf = \b \log 2. \label{eqn:thm-main-alternate-form-step1}
\end{align}
On the other hand, since there exists constants $c_1, c_2 > 0$ such that $\log(k!) \geq c_1 + c_2 k \log k$ by Stirling's approximation, 
\begin{align*}
H(\rho_{o, \text{deg}} \| \etab{\b}_{o, \text{deg}}) \geq  - H(\rho_{o, \text{deg}}) + \b - \log \b \EDRinf  + c_1 +  c_2 \EDLDRinf,
\end{align*} 
and the result follows from \eqref{eqn:thm-main-alternate-form-step1} and the assumption $\EDLDRinf = \infty$.
\end{proof}

We can now prove Theorem \ref{thm:alternate-form-theorem}.
\begin{proof}[Proof of Theorem \ref{thm:alternate-form-theorem}]
If $\b = 0$, from the definitions of $\calItilde_{\b, \eta}$ and $\calI_{\b, \eta}$ in \eqref{eqn:calItilde-generic} and \eqref{eqn:calI-generic}, respectively, we have $\calItilde_{\b, \eta}(\rho) = \calI_{\b, \eta}(\rho) = H(\r{o} \| \nu)$. 

Next, assume $\b > 0$. In this situation, we have three cases.

{\em Case I:} Suppose $\rho \notin \PUB$.  Then by \eqref{eqn:calItilde-generic} and \eqref{eqn:calI-generic}  we have $\calItilde_{\b, \eta}(\rho) = \calI_{\b, \eta}(\rho) = \infty$. 

{\em Case II:} Suppose  $\rho \in \PUB$ and  $\EDLDRinf = \infty$.  Then from the definition of $\calItilde_{\b, \eta}$ in \eqref{eqn:calItilde-generic},  we have  $\calItilde_{\b, \eta}(\rho) =\infty$. We now argue that $\calI_{\b, \eta}(\rho) = \infty$. Due to \eqref{eqn:calI-generic}, we may assume that $\r{1} \ll \eta_1$. Since $\EDLDRinf = \infty$ and $\alpha$ has finite support, it follows that $\eta = \etab{\b}$ for some $\b > 0$. From the definitions of $\calI_{\b, \eta}$ and $\RP_1$ in \eqref{eqn:calI-generic} and \eqref{eqn:RP}, respectively, and using the chain rule and non-negativity of relative entropy, we obtain
\begin{align*}
2\calI_{\b, \eta}(\rho) & \geq H(\r{1} \| \RP_1) \geq H(\rho_{o, \text{deg}} \| \etab{\b}_{o, \text{deg}}) = \infty,
\end{align*}
where the latter inequality follows from Lemma \ref{lemma:EDLDinf-cond}. This shows  $\calItilde_{\b, \eta}(\rho) = \calI_{\b, \eta}(\rho) = \infty$.

{\em Case III:} Suppose let $\rho \in \PUB$ and $\EDLDRinf < \infty$. Then Proposition \ref{prop:alternate-form-1}, Proposition \ref{prop:alternate-form-h}, \eqref{eqn:calItilde-generic}, and  \eqref{eqn:calI-generic} together imply  $\calItilde_{\b, \eta}(\rho) = \calI_{\b, \eta}(\rho)$. This completes the proof.
\end{proof}

\subsection{Proof of Corollary \ref{cor-kapparf}}
\label{sec:proof-vertex-only}
Let $\mu \in \P^\b(\T{1})$ be such that $\SBMuM$ is symmetric, $\Exp_{\mu} \sqpr{\dego} = \b > 0$, and $\mu \ll \eta_1$. From \eqref{eqn:RP}-\eqref{eqn:RCP} it is automatic that $\MCP \ll \MP$. Since $\eta^{1|o}(\cdot | x_o) = \nu(\cdot)$ for all $x_o \in \X$  in the case when there are no edge marks, it follows that 
\begin{align*}
    \frac{d\MCP}{d\MP} (\tau) = \prod_{v \in N_o(\tau)} \frac{d\SBMuM}{d(\SBMuRM \otimes \SBMuMZero)}\cpr{X_v, X_o}, \quad \tau = (\tau, X) \in \TX.
\end{align*}
Therefore, by the definition in  \eqref{eqn:sbmubar} and the exchangeability property in Remark \ref{remark:exchageability}, we get
\begin{align*}
    H(\mu \| \MP)  &= H(\mu \| \MCP) + \Exp_{\mu} \sqpr{\log \cpr{\frac{d\MCP}{d\MP}}}\\ 
    & = H(\mu \| \MCP) + \Exp_{\mu} \sqpr{\dego \log \cpr{\frac{d\SBMuM }{d(\SBMuRM \otimes \SBMuMZero)}(\bfX_1,\bfX_o)}}\\
    &  = H(\mu \| \MCP) + \b \Exp_{\bar{\mu}} \sqpr{ \log \cpr{\frac{d\SBMuM }{d(\SBMuRM \otimes \SBMuMZero)}(\bfX_1, \bfX_o)}} \\ 
    & = H(\mu \| \MCP) + \b H(\SBMuM \| \SBMuRM \otimes \SBMuMZero),
\end{align*}
and the result follows.
\qed 

\section{Proof of Theorems \ref{thm:iid-old} and \ref{thm:iid-nbd}}
In this section, we complete the proof of our main result. We define the microstate entropy and related quantities in Section \ref{sec:prelim-pi-eh} and state the existing LDPs with rate functions involving the microstate entropy in Section \ref{sec:iid-entropy-form-proof}. We then prove Theorem \ref{thm:iid-old} by establishing the equivalence between these rate functions and  \eqref{eqn:I-CM-iid-old}-\eqref{eqn:I-ER-iid-old} in Sections \ref{sec:I-cm-equivalence-proof}-\ref{sec:I-fe-er-equivalence-proof}. Together with Theorem \ref{thm:alternate-form-theorem}, this completes the proof of the component empirical measure LDP (Theorem \ref{thm:iid}). Finally, the LDP for the neighborhood empirical measure (Theorem \ref{thm:iid-nbd}) is deduced as a corollary to Theorem \ref{thm:iid} in Section \ref{sec:iid-nbd-pf}.
\label{sec:iid-proof-complete}

\subsection{Preliminaries}
\label{sec:prelim-pi-eh}
We begin with some preliminaries. Throughout this subsection, we fix $\b > 0$. Also, recall $\PB$ from Section \ref{sec:prelim-probability-graphs}.

\subsubsection{Covariance measures associated with $\PB$ elements}
We start by defining some combinatorial quantities associated with $\Tinf$ elements and identify the size-biased marginals in Definition \ref{def:admissible-h} using them. 

Let $h \geq 1$. Recall Definition \ref{def:vcutu}. We make the following definition.
\begin{definition}
\label{def:Eh}
For $h \in \N$, given $\tau, \tau' \in \T{h-1} \times \Y$, define $E_h(\tau, \tau^\prime)  : \T{h} \to \NZ$ by
\begin{align}
E_h(\tau, \tau^\prime) (\TBdef) := \left| \pr{v \in N_o(\TBdef) : \TBdefOV_{h-1} = \tau, \TBdefVO_{h-1} = \tau^\prime} \right|, \quad \TBdef \in \Tinf.
\label{eqn:def-Eh}    
\end{align}
In words, $E_h(\tau, \tau')(\TBdef)$ is equal to the number of neighbors $v$ of the root $o$ in the marked tree $\TBdef$ such that the pair of trees comprised of (i) the $(h-1)$-neighborhood of $v$ with root $v$ in the tree $\TBdef$ with the edge $\{v, o\}$ removed (but with the edge mark $y_{(v, o)}$ retained), and (ii) the $(h-1)$-neighborhood of $o$ with root $o$ in the tree $\TBdef$ with the edge $\{v,o\}$ removed (but with the edge mark $y_{(o, v)}$ retained) is equal to $(\tau, \tau')$.
\end{definition}
We now describe certain covariance measures associated with $\PB$ elements that were first defined for the case of unmarked graphs in  \cite[Section 1.3]{BorCap15}, and then extended to marked graphs in \cite[Section 2]{DelAna19}.
\begin{definition}[Covariance measures]
\label{def-PIRh}
Let $\rho \in \PB$. For $h \in \N$, define $\PIR{h} \in \P((\T{h-1} \times \Y)^2)$ by
\begin{align}
\PIR{h}(\tau, \tau^\prime) := \frac{1}{\b} \Exp_{\r{h}}\sqpr{E_h(\tau, \tau^\prime) (\TB)}, \, \text{ for } \tau, \tau^\prime \in \T{h-1} \times \Y.
\label{eqn:PIRh}
\end{align}
Also, let $\PIR{h}^1$ (resp. $\PIR{h}^o$) denote the first (resp. second) marginal of $\PIR{h}$.
\end{definition}

\begin{remark}
The arguments $\tau$ and $\tau'$ in the definition of $E_h$ in \eqref{eqn:def-Eh} are elements of $\T{h-1} \times \Y$, that is, they are rooted marked trees of depth at most $h-1$ (with an additional edge mark); in particular, they could also be trees of depth $h-2$ (with an additional edge mark). The same applies to the arguments of $\PIR{h}$ as well.
\label{remark:Eh}
\end{remark}

\begin{remark}
In Definition \ref{def:Eh}, when $h = 1$, both $\tau$ and $\tau'$ are $\X \times \Y$ elements. Therefore, in Definition \ref{def-PIRh},  $\PIR{1}$ is a probability measure on $(\X \times \Y)^2$. 
\label{remark:h=1-PIRh}
\end{remark}

\begin{remark}
Let $h \in \N$, $\tau, \tau' \in \T{h-1} \times \Y$, and $\r{h} \in \P^\b(\T{h})$. Let  $\Law(\TB) =\r{h}$. For $\TBdef \in \T{h}$ with $\text{deg}_{\TBdef}(o) \geq 1$, the conditional probability that $(\TBOOne_{h-1} = \tau$, $\TBOneO_{h-1} = \tau')$ given $\TB  =\TBdef$ is 
\begin{align*}
\r{h}\cpr{ \cpr{\TBOOne_{h-1} = \tau, \TBOneO_{h-1} = \tau'} \mid \TB = \TBdef} =  \frac{E_h(\tau, \tau^\prime)(\TBdef)}{\text{deg}_{\TBdef}(o)}.
\end{align*}
This follows from the Ulam-Harris-Neveu labeling scheme in Section \ref{sec:labeling}.
\label{remark:Eh-conditional}
\end{remark}

We now show that $\PIR{h}$ coincides with size-based marginals $\SBRM{h}$ in Definition \ref{def:admissible-h}.

\begin{lemma}Let $\rho \in \PB$. For every $h \in \N$, we have $\PIR{h} = \SBRM{h}$.
\label{lemma:PI-SBR}
\end{lemma}
\begin{proof}
Fix $h \in \N$. Applying first Definition \ref{def:admissible-h} and Remark \ref{remark:Eh-conditional}, then \eqref{eqn:SBRinf} and the identity $\b = \Exp_{\rho} \sqpr{\NO}$, and finally \eqref{eqn:PIRh}, for $\tau, \tau^\prime \in \T{h-1} \times \Y$, we have
\begin{align*}
\SBRM{h}(\tau, \tau^\prime) & = \sum_{\substack{\TBdef \in \T{h}: \\ \text{deg}_{\TBdef}(o) \geq 1}} \SBR{h}(\TBdef) \frac{E_h(\tau, \tau^\prime)(\TBdef)}{\text{deg}_{\TBdef}(o)} \\
& = \sum_{\substack{\TBdef \in \T{h}:\\ \text{deg}_{\TBdef}(o) \geq 1}}   \r{h}(\TBdef) \frac{E_h(\tau, \tau^\prime)(\TBdef)}{\b} \\
& = \PIR{h}(\tau, \tau^\prime).
\end{align*}
\end{proof}

In particular, we have:
\begin{remark}[Equivalent characterization of admissibility]
    In view of Lemma \ref{lemma:PI-SBR} and the definition of admissibility in Definition \ref{def:admissibility}, the admissibility of $\r{h}$ is equivalent to the condition that $\PIR{h}$ is symmetric. 
    \label{remark:equivalent-admissibility}
\end{remark}

\begin{remark}
\label{remark:admissibility}
    If $\rho$ is an element of  $\PUB$ from Definition \ref{def:unimod}, we claim that $\r{h}$ is admissible for each $h \in \N$. Indeed, for any $\tau, \tau' \in \T{h-1} \times \Y$, applying the mass transport principle \eqref{eqn:unimod} with the function $f: \Tdstar[\X ; \Y] \to \R_+ $ defined by
\begin{align*}
f(\TBdef, o_1, o_2) = \indf\{o_1 \in N_{o_2}(\TBdef), \, \TBdef(o_2 \setminus o_1)_{h-1} = \tau, \,  \TBdef(o_1 \setminus o_2)_{h-1} = \tau^\prime\}, \quad (\TBdef, o_1, o_2) \in \Tdstar[\X; \Y],
\end{align*}
it follows that $\PIR{h}(\tau, \tau^\prime) = \PIR{h}(\tau^\prime, \tau)$. By Remark \ref{remark:equivalent-admissibility}, this shows that $\r{h}$ is admissible.
\end{remark}

\subsubsection{Microstate entropy}
\label{sec:sigma}
We start with some definitions. 
\begin{definition}
    Given $\rho \in \PUB$, define the function $\bvr: \Y \times \Y \to \R_+$ by
\begin{align}
\bvr(y,y^\prime) := \b \sum_{x,x^\prime \in \X} \PIR{1}((x,y), (x',y^\prime)), \quad y,y^\prime \in \Y.
\label{eqn:def-kvr}
\end{align}
Note that $\frac{\bvr}{\b} \in \P(\Y \times \Y)$, since 
\begin{align}
\sum_{y,y^\prime \in \Y}\bvr(y,y^\prime) = \b,
\label{eqn:kvr-sum}
\end{align}
and $\frac{\bvr}{\b}$ is the $(\Y \times \Y)$-marginal of $\PIR{1} \in \P((\X \times \Y) \times (\X \times \Y))$.
Also, define
\begin{align}
s(\b) &:= \frac{\b}{2} - \frac{\b}{2} \log \b,\\
\vec{s}(\bvr) &:= \sum_{y, y^\prime \in \Y} \cpr{\frac{\bvr(y,y^\prime)}{2} - \frac{\bvr(y,y^\prime)}{2} \log \cpr{ \bvr(y,y^\prime)}}.
\label{eqn:def-s-kvec}
\end{align}
\label{def:s-s-bvr}
\end{definition}

Fix $h \in \N$. If $\EDLDRinf < \infty$, by Lemma \ref{lemma:finiteness} and Lemma \ref{lemma:PI-SBR}, we see that  $H(\r{h}) < \infty$ and $H(\PIR{h}) < \infty$. Define $J_h : \P(\G{h}) \to [-\infty, \infty)$  by
\begin{align}
J_h(\r{h}) :=-s(\b) + H(\r{h}) - \frac{\b}{2} H(\pi_{\r{h}}) - \sum_{\tau, \tau^\prime \in \T{h-1} \times \Y } \Exp_{\r{h}} [ \log \cpr{\sqpr{ E_h(\tau, \tau^\prime)(\TB)}!}], 
\label{eqn:jh}
\end{align}
if $\r{h} \in \P^\b(\T{h})$, $\r{h}$ is admissible, and $\EDLDR{h} < \infty$; and  $J_h(\r{h}) := -\infty$ otherwise. Although $J_h$ depends on $\b$, we do not denote this dependence for the ease of readability.

We now define the microstate entropy. Given $\theta \in \P(\X)$ and a symmetric function $\lambdavec : \Y \times \Y \to \R_+$, define the map  $\Sigma_{\lambdavec, \theta} : \P(\Ginf) \to [-\infty, \infty]$ as follows:
\begin{align}
\Sigma_{\lambdavec, \theta}(\rho):=
\begin{dcases}
\lim_{h \to \infty} J_h(\r{h}) & \text{ if } 0 < \b:= \Exp_{\rho}\sqpr{\dego} < \infty, \, \rho \in \PUB,   \bvr = \lambdavec, \text{ and }  \r{o} = \theta,  \\
-\infty & \text{ otherwise} 
\end{dcases}
\label{eqn:sigma}
\end{align}
The quantity  $\Sigma_{\lambdavec, \theta}$ is well defined due to the fact that $h \mapsto J_h(\r{h})$ is non-decreasing (see \cite[Theorem 1.3]{BorCap15}, \cite[Theorem 3]{DelAna19}).
This quantity was first introduced in  \cite{BorCap15} in the setting of unmarked graphs and extended to the marked setting in \cite{DelAna19}. In both settings, it is referred to as the ``microstate entropy'' (see also \cite{CheRamYas23}) since it is defined in terms of a graph counting problem. 

\begin{remark}
\label{remark:sigma}
    Note that $\Sigma_{\lambdavec, \theta}(\rho) = -\infty$ whenever either (i) $\bvr \neq \lambdavec$, (ii) $\r{o} \neq \theta$, or (iii) $\rho \notin \PUB$.
\end{remark}

\subsection{The LDP for the component empirical measure}
\label{sec:iid-entropy-form-proof}
We now state the existing results on the component empirical measure LDPs. 
Define the rate function for the $\CM$ graph sequence by
\begin{align}
\label{eqn:I-CM-entropy-form}
\calboldI^\CM(\rho) :=   H(\r{o} \| \nu) + \frac{\k}{2} H\cpr{\frac{\kvr}{\k} \biggr\| \bar{\xi}}  - \Exp_{\alpha}[\log \cpr{\NO !}] + H(\alpha)  +  H(\r{o}) +  \vec{s}(\kvr) - 2s(\k) - \Sigma_{\kvr, \r{o}}(\rho),
\end{align}
if $\rho_{o, \text{deg}} = \alpha$; and $\calboldI^\CM(\rho) := \infty$ otherwise. Here, $\kvr$ is as defined as in \eqref{eqn:def-kvr} with $\b$ replaced by $\k$. For $\b > 0$, define $\calboldI_\b: \P(\Ginf) \to [0, \infty]$ by
 \begin{align}
\calboldI_\b(\rho) := \begin{dcases}
H(\r{o}) + H(\r{o} \| \nu) +\frac{\b}{2} H\cpr{\frac{\bvr}{\b} \biggr\| \bar{\xi}} +  \vec{s}(\bvr) - \Sigma_{\bvr, \r{o}}(\rho), & \text{ if } \Exp_{\rho}[\dego] = \b,  \\
\infty, & \text{ otherwise},
\end{dcases}
\label{eqn:I-entropy-form}
\end{align}
and define the rate functions for the $\FE$ and $\ER$ sequences by

\begin{align}
    \calboldI^\FE(\rho) & : = \calboldI_\k(\rho) \\
    \calboldI^\ER(\rho) & := \begin{dcases}
\ell_\k(0)  + H(\r{o} \| \nu) & \text{if }  \Exp_{\rho}[\dego] = 0,\\
\ell_\k(\b') + \calboldI_{\b'}(\rho) , & \text{if } 0 < \b': = \Exp_{\rho}[\dego]  < \infty,  \\
\infty, & \text{otherwise}.
\end{dcases}
\label{eqn:I-entropy-form-er}
\end{align}

\begin{theorem}
\label{thm:old-iid-microstate-entropy}
    The sequence $\{U^\CM_n\}$ (respectively, $\{U^\FE_n\}$ and $\{U^\ER_m\}$) satisfies the LDP on $\P(\Ginf)$ with rate function $\calboldI^\CM$ (respectively, $\calboldI^\FE$ and $\calboldI^\ER$).
\end{theorem}
\begin{proof}
    The LDP for $\{U^\CM_n\}$ on $\P(\Ginf)$ with rate function $\calboldI^\CM$ is proved in \cite{CheRamYas23}. The LDP for $\{U^\FE_n\}$ (resp. $\{U^\ER_n\}$) on $\P(\Ginf)$ with rate function $\calboldI^\FE$ (resp. $\calboldI^\ER$) is proved in \cite[Theorem 4.1]{BalOlietal22} (resp.  \cite[Theorem 4.2]{BalOlietal22}).
\end{proof}

\subsection{Proof of $\calboldI^\CM = \calItilde^\CM$}
\label{sec:I-cm-equivalence-proof}
\subsubsection{Preliminary lemmas}
We need to first show some preliminary results.
Recall the definitions $\etaa$ and $\etab{\b}$ for $\b \geq 0$ from Remark \ref{remark:local-convergence}.

\begin{lemma} 
\label{lemma:pi-1-simplify}
Let $\beta > 0$ and let $\rho \in \PUB$. The quantities $\PIR{1}$, $\bar{\xi}$ and $\bvr$ defined in \eqref{eqn:PIRh}, \eqref{eqn:def-xibar} and \eqref{eqn:def-kvr}, respectively,  satisfy following identities:
\begin{enumerate}
    \item[(i)]  
    \begin{align}
    H(\pi_{\r{1}} \| \pi_{\etab{\b}_1}) & = H(\pi_{\r{1}} \| \pi_{\etaa_1}) \nonumber \\
    & = -H(\pi_{\r{1}}) - 2 \sum_{\substack{ y, y^\prime \in \Y, \\ x, x^\prime \in \X}} \pi_{\r{1}}((x,y), (x^\prime, y^\prime)) \log (\nu(x^\prime)) \nonumber \\
    & \qquad + H\cpr{\frac{\bvr}{\b} \biggr\| \bar{\xi}} -\frac{1}{\b} \sum_{y,y^\prime \in \Y} \bvr(y,y^\prime) \log(\bvr(y,y^\prime)) + \log \b.
    \label{eqn:iid-pf-pi-1-simplify}
    \end{align}
    \item[(ii)]  
    \begin{align}
    \sum_{y,y^\prime \in \Y} \bvr(y,y^\prime) \log\left(\bar{\xi}(y,y^\prime)\right) &  = -\b H\cpr{\frac{\bvr}{\b} \biggr\| \bar{\xi}} + \sum_{y,y^\prime \in \Y} \bvr(y,y^\prime) \log(\bvr(y,y^\prime)) - \b \log \b.
\label{eqn:iid-pf-kxibar}
\end{align}
\end{enumerate}
\end{lemma}

\begin{proof}  
We first show the second identity \eqref{eqn:iid-pf-kxibar}. Using the definition of relative entropy in \eqref{eqn:def-renelt} and the fact that  $\sum_{y,y^\prime  \in \Y} \bvr(y, y^\prime) = \b$ (see \eqref{eqn:kvr-sum}), we see that
\begin{align*}
\sum_{y,y^\prime \in \Y} \bvr(y,y^\prime) \log\left(\bar{\xi}(y,y^\prime)\right) & = -\b \sum_{y,y^\prime \in \Y} \frac{\bvr(y,y^\prime)}{\b} \log\cpr{\frac{\bvr(y,y^\prime)/\b}{{\bar{\xi}(y,y^\prime)}}} + \sum_{y,y^\prime \in \Y} \bvr(y,y^\prime) \log(\bvr(y,y^\prime)/\b),
\end{align*}
and  \eqref{eqn:iid-pf-kxibar} follows.

We now show the first identity \eqref{eqn:iid-pf-pi-1-simplify}. Since both $\etab{\b}$ and $\etaa$ have i.i.d.\,vertex and edge marks that are independent of each other and the tree structure, by the definition of $\PIeta$ in \eqref{eqn:PIRh} and Remark \ref{remark:h=1-PIRh}, it follows that both $\pi_{\etab{\b}_1}((x, y), (x^\prime, y^\prime))$ and $\pi_{\etaa_1}((x, y), (x^\prime, y^\prime))$ are equal to $ \nu(x) \nu(x^\prime) \bar{\xi}(y, y^\prime)$ for all $(x,y), (x', y') \in \X \times \Y$. This yields the identity
\begin{align}
H(\pi_{\r{1}} \| \pi_{\etab{\b}_1})  = H(\pi_{\r{1}} \| \pi_{\etaa_1}).
\label{eqn:pi-1-step1}
\end{align}
Moreover, together with  \eqref{eqn:def-entropy}, \eqref{eqn:def-renelt}, the admissibility of $\r{1}$ (which is  a consequence of the unimodularity of $\rho$ and Remark \ref{remark:admissibility}), and the definition of $\bvr$ in \eqref{eqn:def-kvr}, this implies
\begin{align*}
H(\pi_{\r{1}} \| \pi_{\etab{\b}_1}) & = - H(\pi_{\r{1}}) - \sum_{\substack{ y, y^\prime \in \Y, \\ x, x^\prime \in \X}} \pi_{\r{1}}((x,y), (x^\prime, y^\prime)) \log\cpr{\nu(x) \nu(x^\prime) \bar{\xi}(y, y^\prime)} \\
& = -H(\pi_{\r{1}}) - 2 \sum_{\substack{ y, y^\prime \in \Y, \\ x, x^\prime \in \X}} \pi_{\r{1}}((x,y), (x^\prime, y^\prime)) \log \cpr{\nu(x^\prime)} \\
& \qquad - \sum_{\substack{ y, y^\prime \in \Y, \\ x, x^\prime \in \X}} \pi_{\r{1}}((x,y), (x^\prime, y^\prime)) \log\cpr{\bar{\xi}(y,y^\prime)} \\
& = -H(\pi_{\r{1}}) - 2 \sum_{\substack{ y, y^\prime \in \Y, \\ x, x^\prime \in \X}} \pi_{\r{1}}((x,y), (x^\prime, y^\prime)) \log (\nu(x^\prime)) \\
& \qquad - \frac{1}{\b}  \sum_{ y, y^\prime \in \Y} \bvr(y,y^\prime) \log\left(\bar{\xi}(y,y^\prime)\right). \numberthis
\end{align*}
Together with \eqref{eqn:pi-1-step1},  \eqref{eqn:iid-pf-kxibar}, and \eqref{eqn:pi-1-step1}, we arrive at \eqref{eqn:iid-pf-pi-1-simplify}.
\end{proof}

\begin{lemma}
\label{lemma:jh-recursion}
Let $\beta > 0$. Let $\rho \in \PUB$  and suppose that $\EDLDRinf < \infty$. If $J_1(\r{1}) > -\infty$, then
\begin{align*}
\Sigma_{\bvr, \r{o}}(\rho) = J_1(\r{1}) -  \sum_{h \geq 2} \sqpr{ H(\rho_{h} \| \rho^*_{h}) - \frac{\b}{2} H(\pi_{\rho_{h}} \| \pi_{\rho^*_{h}})}.
\end{align*}
\end{lemma}
\begin{proof}
By repeating the computations in \cite[Remark 5.13]{BorCap15} verbatim for the marked case, we can deduce the relation
\begin{align*}
J_{h}(\rho_{h}) = J_{h-1}(\r{h-1})  - \sqpr{ H(\rho_{h} \| \rho^*_{h}) - \frac{\b}{2} H(\pi_{\rho_{h}} \| \pi_{\rho^*_{h}})}, \quad h \geq 2.
\end{align*}
(In particular,  by Lemma \ref{lemma:finiteness}, we have $J_h(\r{h}) > -\infty$ for all $h \geq 2$.) Together with the definition of $\Sigma$ in \eqref{eqn:sigma}, the conclusion follows.
\end{proof}

\begin{lemma}
\label{lemma:j1-h-relation-CM}
Let $\rho \in \PUK$ with $\rho_{o, \text{deg}} = \alpha$. Then
\begin{align*}
J_1(\r{1}) = H(\r{o} \| \nu) & + \frac{\k}{2} H\cpr{\frac{\kvr}{\k} \biggr\| \bar{\xi}}  - \Exp_{\alpha}[\log \cpr{\NO !}] + H(\alpha) + \vec{s}(\kvr) - 2s(\k)  \\
& +  H(\r{o}) +  \cpr{H(\r{1} \| \etaa_1) - \frac{\k}{2} H(\PIR{1} \| \PIetaalp) }. \numberthis \label{eqn:j1-h-relation-cm}
\end{align*}
\end{lemma}
\begin{proof}
Since $\rho_{o, \text{deg}} = \alpha$ and $\alpha$ has finite support, we have $H(\alpha) < \infty$, $\Exp_{\alpha}[\log \cpr{\NO !}] < \infty$ and $H(\r{1} \| \etaa_1) < \infty$. Together with the fact that both $\X$ and $\Y$ are finite sets, it follows that all the terms on the right-hand side of \eqref{eqn:j1-h-relation-cm} are finite. Using \eqref{eqn:def-renelt}  and \eqref{eqn:def-entropy},  we have
\begin{align}
H(\r{1} \| \etaa_1) = -H(\r{1}) - \sum_{\tau \in \T{1}} \r{1}(\tau) \log \cpr{\etaa_1(\tau)}.
\label{eqn:h-simplify1-cm}
\end{align}
Since $\etaa$ is the law of the size-biased Galton-Watson tree with offspring distribution $\alpha$ with i.i.d.\,vertex marks with law $\nu$ and i.i.d edge marks with law $\xi$, both independent of each other and the  underlying unmarked tree (see Remark \ref{remark:local-convergence}), for any $\tau \in \T{1}$ with root mark $x$, we have, with $E_1$ given by \eqref{eqn:def-Eh},
\begin{align*}
\etaa_1(\tau) = \alpha(\text{deg}_\tau(o)) \,  
\sqpr{\text{deg}_\tau(o)}! \, \nu(x)  \prod_{\substack{x^\prime \in \X \\ y,y^\prime \in \Y}} \frac{(\nu(x^\prime) \bar{\xi}(y,y^\prime))^{E_1((x,y), (x^\prime, y^\prime))(\tau)}}{\sqpr{E_1((x,y), (x^\prime, y^\prime))(\tau)}!}.
\end{align*}
This implies
\begin{align}
-\sum_{\tau\in \T{1}} \r{1}(\tau) \log\cpr{\etaa_1(\tau)} = I_1 + I_2 + I_3,
\label{eqn:ent-simplify-cm}
\end{align}
where, using the identity $\sum_{\tau \in \T{1} : \tau_o = x} \r{1}(\tau) = \r{o}(x)$ and  recalling that $\rho_{o, \text{deg}} = \alpha$, we have 
\begin{align*}
I_1 &:= -\sum_{x \in \X} \r{o}(x) \log \cpr{\nu(x)} - \Exp_{\alpha}\sqpr{\log \cpr{\NO !}}  - \Exp_{\alpha}\sqpr{ \alpha(\NO)} \\
& = H(\r{o}) + H(\r{o} \| \nu) -  \Exp_{\alpha}[\log \cpr{\NO ! }] + H(\alpha), \\
I_2 & := -\sum_{\substack{y, y^\prime \in \Y, \\ x, x^\prime \in \X}} \cpr{\sum_{\substack{\tau \in \T{1}: \\ \tau_o = x}}\r{1}(\tau) E_1((x, y), (x^\prime, y^\prime))(\tau)} \log(\bar{\xi}(y,y^\prime) \nu(x^\prime)) , \numberthis \label{eqn:proof-I_2}\\
I_3 & := \sum_{\substack{y, y^\prime \in \Y, \\ x, x^\prime \in \X}} \sum_{\substack{\tau \in \T{1} :\\ \tau_o = x}} \r{1}(\tau) \log \cpr{ \sqpr{E_1((x, y), (x^\prime, y^\prime))(\tau)}!}.\numberthis \label{eqn:proof-I_3}
\end{align*}

We first consider $I_2$. By  \eqref{eqn:PIRh}, \eqref{eqn:def-Eh} and the fact that $\Exp_{\alpha}\sqpr{\NO} = \k$, we have  \\  $\sum_{\substack{\tau \in \T{1}: \\ \tau_o = x}}\r{1}(\tau) E_1((x, y), (x^\prime, y^\prime))(\tau) = \k \PIR{1}((x,y), (x^\prime, y^\prime))$. Together with the definition of $\kvr$ in \eqref{eqn:def-kvr}, this implies
\begin{align*}
I_2 & =  - \k \sum_{\substack{ y, y^\prime \in \Y, \\ x, x^\prime \in \X}}  \pi_{\r{1}}((x,y), (x^\prime, y^\prime)) \log (\bar{\xi}(y,y^\prime) \nu(x^\prime))   \\
& = -  \sum_{y,y^\prime \in \Y} \kvr(y,y^\prime) \log(\bar{\xi}(y,y^\prime)) - \k \sum_{\substack{y, y^\prime \in \Y \\ x, x^\prime \in \X}} \pi_{\r{1}}((x,y), (x^\prime, y^\prime)) \log (\nu(x^\prime)). \numberthis \label{eqref:proof_I_2-simplify}
\end{align*}
Combining the last two displays with \eqref{eqn:ent-simplify-cm} and \eqref{eqn:h-simplify1-cm}, we see that 
\begin{align}
H(\r{1} \| \etaa_1)&  = -H(\r{1}) +  H(\rho_{o} \| \nu) +  H(\rho_{o}) - \Exp_{\alpha}\sqpr{\log\cpr{ \NO !}} + H(\alpha) \nonumber \\
& \qquad - \sum_{y,y^\prime \in \Y} \kvr(y,y^\prime) \log(\bar{\xi}(y,y^\prime)) - \k \sum_{\substack{ y, y^\prime \in \Y, \\ x, x^\prime \in \X}} \pi_{\r{1}}((x,y), (x^\prime, y^\prime))  \log (\nu(x^\prime))  \nonumber \\
& \qquad + \sum_{\substack{ y, y^\prime \in \Y, \\ x, x^\prime \in \X}} \Exp_{\r{1}}\sqpr{\log \cpr{\sqpr{E_1((x,y), (x^\prime, y^\prime))(\TB)}!}}.
\label{eqn:h-simplify2-cm}
\end{align}
Using  the expression in \eqref{eqn:iid-pf-kxibar} for the sixth term on the right-hand side above, the identity \eqref{eqn:iid-pf-pi-1-simplify} and  the definition of $J_1$ in \eqref{eqn:jh}, we obtain
\begin{align*}
H(\r{1} \| \etaa_1) - \frac{\k}{2} H(\pi_{\r{1}} \| \pi_{\etaa_1}) &  = -H(\r{1}) + \frac{\k}{2}H(\pi_{\r{1}}) + H(\r{o} \| \nu) + H(\r{o}) - \Exp_{\alpha}\sqpr{\log \cpr{\NO !}} + H(\alpha)  \\
& \qquad  +\frac{\k}{2} H\cpr{\frac{\kvr}{\k} \biggr\| \bar{\xi}}  + \frac{\k}{2} \log \k -\frac{1}{2} \sum_{y,y^\prime \in \Y} \kvr(y, y^\prime) \log(\kvr(y, y^\prime)) \\
& \qquad +  \sum_{\substack{ y, y^\prime \in \Y, \\ x, x^\prime \in \X}} \Exp_{\r{1}}\sqpr{\log \cpr{\sqpr{E_1((x,y), (x^\prime, y^\prime))(\TB)}!}}\\
& = -J_1(\r{1}) +  H(\r{o} \| \nu) + H(\r{o}) - \Exp_{\alpha}\sqpr{\log \cpr{\NO !}} + H(\alpha) \\
& \qquad +\frac{\k}{2} H\cpr{\frac{\kvr}{\k} \biggr\| \bar{\xi}}+ \vec{s}(\kvr) - 2s(\k).
\end{align*}
This proves \eqref{eqn:j1-h-relation-cm}.
\end{proof}

\subsubsection{Proof of $\calboldI^\CM = \calItilde^\CM$}
\label{sec:icm-itildecm-proof}
By the definitions in \eqref{eqn:I-CM-entropy-form} and \eqref{eqn:I-CM-iid-old} we have $\calItilde^\CM(\rho) =\calboldI^\CM(\rho)= \infty$ whenever either $\rho_{o, \text{deg}} \neq \alpha$ or $\rho \notin \PUK$. On the other hand, when $\rho \in \PUK$ with $\rho_{o, \text{deg}} = \alpha$,  Lemma \ref{lemma:j1-h-relation-CM} together with \eqref{eqn:I-CM-entropy-form}   implies that 
\begin{align*}
\calboldI^\CM(\rho) =   H(\r{1} \| \etaa_1) - \frac{\k}{2} H(\pi_{\r{1}} \| \pi_{\etaa_1}) + J_1(\rho) - \Sigma_{\kvr, \r{o}}(\rho).
\end{align*}
Since $\r{1}$ has finite support (because $\rho_{o, \text{deg}} = \alpha$), from the definition of $J_1(\r{1})$ in \eqref{eqn:jh}, it follows that  $J_1(\r{1}) >-\infty$. Hence, the last display together with  Lemma \ref{lemma:jh-recursion} and Lemma \ref{lemma:PI-SBR} imply $\calboldI^\CM(\rho) = \calItilde^\CM(\rho)$. 
\qed

\subsection{Proof of $\calboldI^\FE = \calItilde^\FE$ and $\calboldI^\ER = \calItilde^\ER$}
\label{sec:I-fe-er-equivalence-proof}
\subsubsection{Preliminary lemmas}
We first prove a preliminary lemma, which shows a similar statement as in Lemma \ref{lemma:j1-h-relation-CM} in the case of the $\CM$ random graph sequence.
\begin{lemma}
\label{lemma:j1-h-relation}
Let $\beta > 0$. Let $\rho \in \PUB$ and suppose that $\EDLDRinf  < \infty$. Then
\begin{align}
J_1(\r{1}) = \vec{s}(\bvr) + H(\r{o}) + H(\r{o} \| \nu) + \frac{\b}{2} H\cpr{\frac{\bvr}{\b} \biggr\| \bar{\xi} } - \left(H(\r{1} \| \etab{\b}_1) - \frac{\b}{2} H(\pi_{\r{1}} \| \PIetab{\b})\right).
\label{eqn:j1-h-relation}
\end{align}
\end{lemma}
\begin{proof}
The proof proceeds similar to that of Lemma \ref{lemma:j1-h-relation-CM} in the $\CM$ case, and we shall only indicate the main differences. 

Since $\EDLDRinf < \infty$, by Lemma \ref{lemma:finiteness}, we have $H(\r{1}) < \infty$. Using \eqref{eqn:def-renelt} and \eqref{eqn:def-entropy}, we get
\begin{align}
H(\r{1} \| \etab{\b}_1) = -H(\r{1}) - \sum_{\tau \in \T{1}} \r{1}(\tau) \log \cpr{\etab{\b}_1(\tau)},
\label{eqn:h-simplify1}
\end{align}
For any $\tau \in \T{1}$ with root mark $x$, using the definition  of $\etab{\b}_1$, we have
\begin{align*}
\etab{\b}_1(\tau) =  \nu(x) \prod_{\substack{y,y^\prime \in \Y, \\ x^\prime \in \X} } \frac{\exp\{-\b \bar{\xi}(y, y^\prime) \nu(x^\prime)\}   (\b \bar{\xi}(y, y^\prime) \nu(x^\prime))^{E_1((x, y), (x^\prime, y^\prime))(\tau)}}{\sqpr{E_1((x, y), (x^\prime, y^\prime))(\tau)}!}.
\end{align*}
Using this expression, similar to \eqref{eqn:ent-simplify-cm}, we arrive at
\begin{align}
-\sum_{\tau \in \T{1}} \r{1}(\tau) \log\cpr{\etab{\b}_1(\tau)} = I_1 + I_2 + I_3,
\label{eqn:ent-simplify}
\end{align}
where $I_2$ and $I_3$ are as given in \eqref{eqn:proof-I_2} and \eqref{eqn:proof-I_3}, respectively, with $\k$ replaced by $\b$, and 
\begin{align*}
I_1 &:= -\sum_{x \in \X} \r{o}(x) \log \cpr{\nu(x)} + \sum_{x \in \X} \r{o}(x) \cpr{ \sum_{\substack{y, y^\prime \in \Y, \\ x^\prime \in \X} } \b \bar{\xi}(y,y^\prime)\nu(x^\prime) } - \b \log \b.
\end{align*}
Note that, since $\X$ and $\Y$ are finite sets, $I_1 < \infty$. Also, since $E_1((x, y), (x^\prime, y^\prime))(\tau) \leq \text{deg}_\tau(o)$, and since $\Exp_{\rho}[\dego] = \b < \infty$, we have $I_2 < \infty$. If $\rho$ is such that $J_1(\r{1}) = - \infty$, then from the definition of $J_1$ in \eqref{eqn:jh} and the finiteness properties established in Lemma \ref{lemma:finiteness},  it follows  that $I_3 = \infty$. Therefore, from \eqref{eqn:ent-simplify} and \eqref{eqn:h-simplify1}, we get $H(\r{1} \| \etab{\b}_1) = \infty$. Hence, both sides of \eqref{eqn:j1-h-relation} evaluates to $-\infty$ when $J_1(\r{1}) = -\infty$. 
Therefore, in the rest of the proof, we may assume that $J_1(\r{1}) > - \infty$, which, by Lemma \ref{lemma:finiteness}, yields $I_3 < \infty$. 

Consider $I_1$ in the previous display. Since $\sum_{x \in \X} \r{o}(x) \sum_{\substack{y, y^\prime \in \Y, \\ x^\prime \in \X} } \b \bar{\xi}(y,y^\prime)\nu(x^\prime) = \b$, using  \eqref{eqn:def-entropy} and \eqref{eqn:def-renelt}, we get
\begin{align*}
I_1 = H(\rho_{o} \| \nu) +  H(\rho_{o}) + \b - \b \log \b.
\end{align*}
Together with the expression of $I_2$ from \eqref{eqref:proof_I_2-simplify} with $\k$ replaced by $\b$, the definition of $I_3$ from \eqref{eqn:proof-I_3}, \eqref{eqn:ent-simplify} and \eqref{eqn:h-simplify1},  we obtain
\begin{align}
H(\r{1} \| \etab{\b}_1)&  = -H(\r{1}) +  H(\rho_{o} \| \nu) +  H(\rho_{o}) + \b  - \b \log \b \nonumber \\
& \qquad - \sum_{y,y^\prime \in \Y} \bvr(y,y^\prime) \log(\bar{\xi}(y,y^\prime)) - \b \sum_{\substack{ y, y^\prime \in \Y, \\ x, x^\prime \in \X}}  \log (\nu(x^\prime)) \pi_{\r{1}}((x,y), (x^\prime, y^\prime)) \nonumber \\
& \qquad + \sum_{\substack{ y, y^\prime \in \Y, \\ x, x^\prime \in \X}} \Exp_{\r{1}}\sqpr{\log \cpr{\sqpr{E_1((x,y), (x^\prime, y^\prime))(\TB)}!}}.
\label{eqn:h-simplify2}
\end{align}
Substituting for the sixth term on the right-hand side above from \eqref{eqn:iid-pf-kxibar} of Lemma \ref{lemma:pi-1-simplify}, using the identity \eqref{eqn:iid-pf-pi-1-simplify} for $H(\PIR{1} \| \PIetab{\b})$ from Lemma \ref{lemma:pi-1-simplify}, using the definitions of $J_1$ in \eqref{eqn:jh} and $s(\b)$ and $\vec{s}(\bvr)$ from Definition \ref{def:s-s-bvr}, it follows that
\begin{align*}
H(\r{1} \| \etab{\b}_1) - \frac{\b}{2} H(\pi_{\r{1}} \| \PIetab{\b}) = -J_1(\r{1}) +  H(\r{o} \| \nu) + H(\r{o}) +\frac{\b}{2} H\cpr{\frac{\bvr}{\b} \biggr\| \bar{\xi}} + \vec{s}(\bvr).
\end{align*}
Rearranging the above yields \eqref{eqn:j1-h-relation}, and thus completes the proof of the lemma. 
\end{proof}

Next, we show that $\calboldI_\b =  \calItilde_{\b, \etab{\b}}$, where $\calboldI_\b$ and $\calItilde_{\b, \etab{\b}}$ are defined in \eqref{eqn:I-entropy-form} and \eqref{eqn:calItilde-generic}, respectively.
\begin{lemma}
\label{lemma:new-old-fe-er}
    Let $\beta > 0$. We have $\calboldI_\b(\rho) = \calItilde_{\b, \etab{\b}}(\rho)$ for all $\rho \in \P(\Ginf)$.
\end{lemma}
\begin{proof}
We consider three cases.

{\em Case I:} Let $\rho \in \P(\Ginf) \setminus \PUB$. By the definition of $\calItilde_{\b, \etab{\b}}$, we have $\calItilde_{\b, \etab{\b}}(\rho) = \infty$. Also, by Remark \ref{remark:sigma}, we have $\Sigma_{\bvr, \r{o}}(\rho) = - \infty$. This shows $\calboldI_\b(\rho) =  \calItilde_{\b, \etab{\b}}(\rho)$ whenever $\rho \in \P(\Ginf) \setminus \PUB$. \\

{\em Case II:} Let $\rho \in \PUB$ be such that $\EDLDRinf= \infty$. In this case, by the definition of $\Sigma$ in \eqref{eqn:sigma} and the definition of $J_1$ in \eqref{eqn:jh}, we see that $\Sigma_{\bvr, \r{o}}(\rho) = -\infty$. Therefore, $\calboldI_\b(\rho) = \infty$. On the other hand, by Lemma \ref{lemma:EDLDinf-cond}, we have $H(\rho_{o, \text{deg}} \| \etab{\b}_{o, \text{deg}}) = \infty$. Therefore, by the chain rule and non-negativity of relative entropy, we have
\begin{align*}
H(\r{1} \| \etab{\b}_1) & \geq H(\rho_{o, \text{deg}} \| \etab{\b}_{o, \text{deg}}) = \infty.
\end{align*}
Together with Lemma \ref{lemma:finiteness} and \eqref{eqn:calItilde-generic},  this shows that $\calItilde_{\b, \etab{\b}}(\rho) = \infty$. Hence, $\calboldI_\b(\rho) = \calItilde_{\b, \etab{\b}}(\rho)$ whenever $\rho \in \PUB$ is such that $\EDLDRinf = \infty$. 
 
{\em Case III:} Let $\rho \in \PUB$ be such that $\EDLDRinf < \infty$. Recall from \eqref{eqn:sigma} that $\Sigma_{\bvr, \r{o}}(\rho) = \lim_{h \to \infty} J_h(\r{h})$. First, suppose that $J_1(\r{1}) = -\infty$. Since $J_h(\r{h})$ is non-increasing in $h$ \cite[Theorem 3]{DelAna19}, it follows that $J_h(\r{h}) = -\infty$ for all $h \geq 2$, and hence, from \eqref{eqn:sigma}, we get $\Sigma_{\bvr, \r{o}}(\rho) = -\infty$. Therefore, $\calboldI_\b(\rho) = \infty$. On the other hand, by Lemma \ref{lemma:j1-h-relation}, $J_1(\r{1}) = -\infty$ implies that $H(\r{1} \| \etab{\b}_1) = \infty$. Therefore, together with Lemma \ref{lemma:finiteness}, we have $\calItilde_{\b, \etab{\b}}(\rho) = \infty$. Next, suppose that $J_1(\r{1}) > -\infty$. In this case, by Lemma \ref{lemma:jh-recursion} and  Lemma \ref{lemma:j1-h-relation}, it follows  that 
\begin{align*}
\Sigma_{\bvr, \r{o}}(\sigma) & = \vec{s}(\bvr) + H(\r{o}) + H(\r{o} \| \nu) + \frac{\b}{2} H\cpr{\frac{\bvr}{\b} \biggr\| \bar{\xi} } \\
& \quad - \left(H(\r{1} \| \etab{\b}_1) - \frac{\b}{2} H(\pi_{\r{1}} \| \pi_{\etab{\b}_1})\right) - \sum_{h \geq 2} \sqpr{ H(\rho_{h} \| \rho^*_{h}) - \frac{\b}{2} H(\pi_{\rho_{h}} \| \pi_{\rho^*_{h}})}.
\end{align*}
Using the above display, the definitions of $\calboldI_\b(\rho)$ and $\calItilde_{\b, \etab{\b}}(\rho)$ in \eqref{eqn:I-entropy-form} and \eqref{eqn:calItilde-generic} respectively, and Lemma \ref{lemma:PI-SBR}, it follows that $\calItilde_{\b, \etab{\b}}(\rho) = \calboldI_\b(\rho)$.  This completes the proof of the lemma
\end{proof}

\subsubsection{Proof of $\calboldI^\FE = \calItilde^\FE$ and $\calboldI^\ER = \calItilde^\ER$}
\label{sec:ifeer-itildefeer-proof}
Since $\k > 0$, by Lemma \ref{lemma:new-old-fe-er} and \eqref{eqn:I-FE-iid-old}, we have $\calboldI_\k = \calItilde_{\k, \etab{\k} } = \calItilde^\FE$.

For the $\ER$ random graph sequence, substituting into \eqref{eqn:I-entropy-form-er}  the identity $\calboldI_{\b'} = \calItilde_{\b', \etab{\b'}}$ when $\b' > 0$ from Lemma \ref{lemma:new-old-fe-er}, using the definition of $\calItilde_{\b', \etab{\b'}}$ when $\b' = 0$ from \eqref{eqn:calItilde-generic}, and invoking the definition of $\calItilde^\ER$ from \eqref{eqn:I-ER-iid-old}, it follows that $\calboldI^\ER = \calItilde^\ER$. 
\qed

We can now complete the proof of Theorem \ref{thm:iid-old}.

\begin{proof}[Proof of Theorem \ref{thm:iid-old}]
The theorem follows from Theorem \ref{thm:old-iid-microstate-entropy} together with the facts $\calboldI^\CM = \calItilde^\CM$, $\calboldI^\FE = \calItilde^\FE$ and $\calboldI^\ER = \calItilde^\ER$ established in Sections \ref{sec:icm-itildecm-proof} and \ref{sec:ifeer-itildefeer-proof}.
\end{proof}

As a consequence of our results, we also arrive at an alternative expression for the microstate entropy in terms of the rate function $\calI_{\b, \etab{\b}}$ in \eqref{eqn:calI-generic}.
\begin{corollary}
Let $\rho \in \P(\Ginf)$ be such that $0 < \b:=\Exp_{\rho}[\dego] < \infty$. Then
\begin{align*}
\Sigma_{\bvr, \r{o}}(\rho) = H(\r{o}) + H(\r{o} \| \nu) +\frac{\b}{2} H\cpr{\frac{\bvr}{\b} \biggr\| \bar{\xi}} +  \vec{s}(\bvr) - \calI_{\b, \etab{\b}}(\rho).
\end{align*}
    \label{cor:sigma-alternative}
\end{corollary}
\begin{proof}
This is an immediate consequence of the LDP for the $\FE$ random graph sequence established in Theorems \ref{thm:iid}, \ref{thm:iid-old} and \ref{thm:old-iid-microstate-entropy}.
\end{proof}

\subsection{Proof of Theorem \ref{thm:iid-nbd}}
\label{sec:iid-nbd-pf}
We now prove Theorem \ref{thm:iid-nbd} using Theorem \ref{thm:iid}. Consider the projection map $\varpi : \P(\Ginf) \to \P(\T{1})$ that maps $\rho$ to its depth 1 marginal $\rho_1$. We first show that this map is continuous.
\begin{lemma} The map $\varpi: \P(\Ginf) \to \P(\T{1})$ is continuous.
\label{lemma:varpi-continuity}
\end{lemma}
\begin{proof}
Let $\rho^n \to \rho$ in $\P(\Ginf)$ as $n \to \infty$, i.e., for any $f \in C_b(\Ginf)$, we have
\begin{align*}
\int_{\Ginf} f \, d\rho^n \to \int_{\Ginf} f \, d\rho \text{ as } n \to \infty.
\end{align*}
Define $\mu^n = \varpi(\rho^n)$, $n \in \N$, and $\mu = \varpi(\rho)$. Let $g \in C_b(\T{1})$. Using $f(G) = g(G_1)$, $G \in \Ginf$, note that
\begin{align*}
\int_{\Ginf} f  \, d\rho^n = \int_{\T{1}} g \, d\mu^n, \, n \in \N; \text{ and } \, \int_{\Ginf} f \, d\rho = \int_{\T{1}} g \, d \mu,
\end{align*}
and hence $\mu^n \to \mu$ in $\P(\T{1})$ as $n \to \infty$. This shows that $\varpi$ is continuous.
\end{proof}
Recall the one-step extension from Definition \ref{def:unimod}. To prove Theorem \ref{thm:iid-nbd}, we will also need to define the full unimodular extension (see Remark \ref{rmk:unimod-full-extension}). Let $h \in \N$, and let $\r{h}$ be admissible.
\begin{definition}[Unimodular extension]
\label{def:unimod-full-extension}
First, we sample $\TB \in \T{h}$ with law $\r{h}$ (i.e., $\Law(\TB) = \r{h}$). If $\dego = 0$, we stop. If $\dego  \geq 1$, independently for each  $v \in N_o(\TB)$, sample  $\TB^\prime \in \T{h} \times \Y$ with law $\PU{h}(\TBOV_{h-1},\TBVO_{h-1})(\cdot)$ and replace $\TBOV_{h-1}$ with $\TB^\prime$ to obtain a $\T{h+1}$ random element.  We repeat this construction independently for the vertices in $\bbV_2$  -- that is, for each vertex  $v \in \bbV_2$, we (independently) sample $\TB^{''}$  with law $\PU{h}(\TB(v \setminus p(v))_{h-1}, \TB(p(v) \setminus v)_{h-1})$ and replace $\TB(v \setminus p(v))_{h-1}$ with $\TB^{''}$. This results in a $\T{h+2}$ random element. We repeat this construction independently for all vertices in $\bbV_3$, $\bbV_4$, etc.  We define $\UGWR{h}$ as the law of the resultant $\Tinf$ random element.
\end{definition}
Note that $\Rstar{h} = (\UGWR{h-1})_{h}$, where $\Rstar{h}$ is defined in \eqref{eqn:Rstarh}.

We now prove Theorem \ref{thm:iid-nbd}.
\begin{proof}[Proof of Theorem \ref{thm:iid-nbd}]
Consider $L_n^\CM$. Note that $L_n^\CM = \varpi(U^\CM_n)$.  By Lemma \ref{lemma:varpi-continuity},  $\varpi$ is continuous. Since $\{U_n^\CM\}$ satisfies the LDP on $\P(\Ginf)$  with rate function $\calI^\CM$ (by Theorem \ref{thm:iid}), it follows from the contraction principle (e.g., \cite[Theorem 4.2.1]{DemZei98}) that the sequence $\{L^\CM_n\}$ satisfies the LDP with rate function 
\begin{align*}
\P(\T{1}) \ni \mu \mapsto \inf\{\calI^\CM(\rho) : \mu = \varpi(\rho)\}.
\end{align*}
On the one hand, if $\mu$ is either not admissible or $\mu_{o, \text{deg}} \neq \alpha$, then  for any $\rho$ such that $\varpi(\rho) = \mu$, by Remark \ref{remark:admissibility}, we have either $\rho \notin \PUK$ or $\rho_{o, \text{deg}} \neq \alpha$. In this case, the infimum in the previous display is $\infty$. Recalling the definition of $\frakI^\CM$ in \eqref{eqn:J-CM-iid}, we have $\frakI^\CM(\mu) = \infty$. On the other hand, if $\mu$ is admissible and $\rho_{o, \text{deg}} = \alpha$, then the above infimum is attained by $\mathsf{UGWT}(\mu)$. Indeed, by the definitions in \eqref{eqn:calI-generic}, \eqref{eqn:I-CM-iid-new}, \eqref{eqn:frakI2-generic}, and \eqref{eqn:J-CM-iid}, the above infimum is larger than $\frakI^\CM(\mu)$. If $\rho = \mathsf{UGWT}(\mu)$, then $\r{h} = \Rstar{h}$ for all $h \geq 2$. In particular, $\SBRM{h} = \SBRstarM{h}$, and by \eqref{eqn:RPh} and \eqref{eqn:RCPh}, it follows that  $\RP_h = \RCP_h = \Rstar{h} = \r{h}$ for all $h \geq 2$, whence  $\calI^\CM(\rho) = \frakI^\CM(\mu)$. Therefore, we conclude that the above infimum is attained by $\mathsf{UGWT}(\mu)$. Hence, the map in the above display coincides with the function  $\frakI^\CM$, and it follows that $\{L_n^\CM\}$ satisfies the LDP on $\P(\T{1})$ with rate function $\frakI^\CM$. Using analogous arguments, we also conclude the LDP for the families  $\{L_n^\FE\}$ and $\{L_n^\ER\}$ on $\P(\T{1})$ with rate functions $\frakI^\FE$ and $\frakI^\ER$, respectively. This completes the proof of the theorem.
\end{proof}

\appendix

\section{Finiteness of various entropies: Proof of Lemma \ref{lemma:finiteness}}
\label{appendix:lemma-finiteness}

We now prove the finiteness of entropies and relative entropies involving $\rho_h$, $\rho_h^*$, $\bar{\rho}_h$ and $\bar{\rho}_h^*$, as stated in Lemma \ref{lemma:finiteness}. Throughout this appendix, we use the definitions introduced in Section \ref{sec:prelim-pi-eh}. We will also appeal to the identity $\SBRM{h} = \PIR{h}$ proved in  Lemma \ref{lemma:PI-SBR}. Note that Lemma \ref{lemma:PI-SBR} does not depend on any result in Section \ref{sec:alternate-form}, and hence there is no circular reasoning in replacing $\SBRM{h}$ by $\PIR{h}$ throughout this section. 

To prove Lemma \ref{lemma:finiteness}, we first establish an alternative expression for the conditional laws $\PU{h}$ introduced in \eqref{eqn:def-PU}  in terms of the quantity $E_h$ defined in \eqref{eqn:def-Eh}. To this end, we
introduce the following definition.

\begin{definition}
For $h \in \N$, $\tau \in \T{h} \times \Y$ and $\tau^\prime \in \T{h-1} \times \Y$, let $\tau \oplus \tau^\prime$ denote the $\T{h}$ element formed by attaching $\tau^\prime$ to the root of $\tau$.  
\label{def:oplus}
\end{definition}
\noindent See Figure \ref{fig:t-oplus} for illustration. Also, note that the degree of the root in $\tau \oplus \tau'$ is one more than the degree of the root of $\tau$, that is,
\begin{align}
\text{deg}_{\tau \oplus \tau'}(o) = 1+\text{deg}_\tau(o), \quad \tau \in \T{h}\times \Y, \tau' \in \T{h-1} \times \Y.
    \label{eqn:deg-oplus}
\end{align}

\begin{figure}
\centering
\begin{subfigure}{.25\textwidth}
\begin{tikzpicture}[scale=0.7]
\filldraw[black] (0,0) circle(3pt);
\filldraw[black] (-1,-1) circle(3pt);
\filldraw[black] (1,-1) circle(3pt);
\filldraw[black] (-1.5,-2) circle(3pt);
\filldraw[black] (-0.5,-2) circle(3pt);
\filldraw[black] (1.5,-2) circle(3pt);
\draw[black, thick] (0,0) -- (-1,-1);
\draw[black, thick] (0,0) -- (1,-1);
\draw[black, thick] (-1.5,-2) -- (-1,-1);
\draw[black, thick] (-0.5,-2) -- (-1,-1);
\draw[black, thick] (1.5,-2) -- (1,-1);
\draw[black, thick, dashed] (0,0) -- (1,0);
\draw[black] (-0.2,-0.2) rectangle (0.2,0.2);
\end{tikzpicture}
\caption{$\tau$}
\end{subfigure}\quad
\begin{subfigure}{0.25\textwidth}
\begin{tikzpicture}[scale=0.7]
\filldraw[black] (3,0) circle(3pt);
\filldraw[black] (2,-1) circle(3pt);
\filldraw[black] (3,-1) circle(3pt);
\filldraw[black] (4,-1) circle(3pt);
\draw[black, thick] (3,0) -- (2,-1);
\draw[black, thick] (3,0) -- (3,-1);
\draw[black, thick] (3,0) -- (4,-1);
\draw[black, thick, dashed] (3,0) -- (2.5,1);
\draw[black] (3,0) circle (6pt);
\end{tikzpicture}
\caption{$\tau'$}
\end{subfigure}\quad
\begin{subfigure}{0.25\textwidth}
\begin{tikzpicture}[scale=0.7]
\filldraw[black] (0,0) circle(3pt);
\filldraw[black] (-1,-1) circle(3pt);
\filldraw[black] (1,-1) circle(3pt);
\filldraw[black] (-1.5,-2) circle(3pt);
\filldraw[black] (-0.5,-2) circle(3pt);
\filldraw[black] (1.5,-2) circle(3pt);
\node at (0,0.4) {$o$};
\draw[black, thick] (0,0) -- (-1,-1);
\draw[black, thick] (0,0) -- (1,-1);
\draw[black, thick] (-1.5,-2) -- (-1,-1);
\draw[black, thick] (-0.5,-2) -- (-1,-1);
\draw[black, thick] (1.5,-2) -- (1,-1);

\filldraw[black] (4,-1) circle(3pt);
\filldraw[black] (3,-2) circle(3pt);
\filldraw[black] (4,-2) circle(3pt);
\filldraw[black] (5,-2) circle(3pt);
\draw[black, thick] (4,-1) -- (3,-2);
\draw[black, thick] (4,-1) -- (4,-2);
\draw[black, thick] (4,-1) -- (5,-2);
\draw[black, thick] (4,-1) -- (0,0);

\draw[black] (4,-1) circle (6pt);
\draw[black] (-0.2,-0.2) rectangle (0.2,0.2);

\end{tikzpicture}
\caption{$\tau \oplus \tau'$}
\end{subfigure}

\caption{Illustration of $\tau \oplus \tau'$, where $\tau \in \T{2} \times \Y$ and $\tau' \in \T{1}\times \Y$. The dashed lines in both $\tau$ and $\tau'$ represent a $\Y$ element. The root of $\tau$ (resp. $\tau'$) is denoted by a square (resp. circle). The root of $\tau \oplus \tau'$ is denoted by $o$.}
\label{fig:t-oplus}
\end{figure}
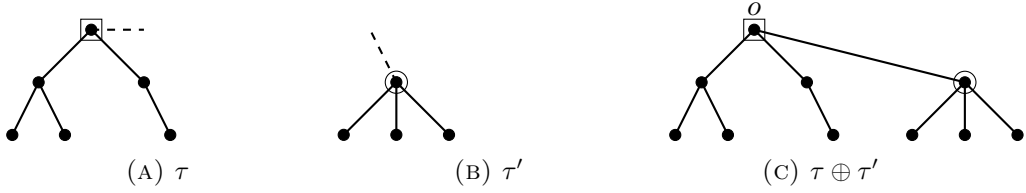

\begin{lemma}
Let $h \in \N$. If $\r{h} \in \P(\T{h})$ is admissible, then for any $\tau, \tau' \in \T{h-1} \times \Y$ with $\PIR{h}(\tau, \tau') > 0$, we have
\begin{align}
\PU{h}(\tau, \tau^\prime)(\TBdef) = \frac{\indf\{\TBdef_{h-1}  = \tau\}}{\b \PIR{h}(\tau, \tau^\prime)} \r{h}(\TBdef  \oplus \tau^\prime) E_h(\tau^\prime, \tau)(\TBdef \oplus \tau^\prime), \, \quad  \TBdef \in \T{h} \times \Y.
\label{eqn:PU}
\end{align}
\end{lemma}

\begin{proof}
For convenience, we rewrite the definition of $\PU{h}$ in \eqref{eqn:def-PU} below:
\begin{align*}
\PU{h}(\tau, \tau')(\TBdef) = \r{h} \cpr{ \TBOneO_h = \TBdef | \TBOneO_{h-1} = \tau, \TBOOne_{h-1} = \tau'},  \quad  \TBdef \in \T{h}.
\end{align*}
If $\TBdef_{h-1} \neq \tau$, it is clear that the right-hand side above as well as the right-hand side of \eqref{eqn:PU} is equal to $0$. Thus, we can assume that $\TBdef_{h-1} = \tau$. In this case, by the last display,
\begin{align}
\PU{h}(\tau, \tau')(\TBdef) = \frac{\r{h} \cpr{ \TBOneO_h = \TBdef, \TBOOne_{h-1} = \tau'}}{\r{h} \cpr{\TBOneO_{h-1} = \tau, \TBOOne_{h-1} = \tau'}}. \label{eqn:proof-pu-conditinal-main}
\end{align}
Consider the denominator first. Using the condition $\PIR{h}(\tau, \tau') > 0$, the admissibility of $\r{h}$ (which implies $\PIR{h}(\tau', \tau) > 0$), the definition of $\SBR{h}$ in  \eqref{eqn:SBRinf}, the identity $\b = \Exp_{\rho}\sqpr{\NO}$, and \eqref{eqn:deg-oplus}, we obtain
\begin{align*}
 \frac{\r{h}\cpr{\TBOneO_{h-1} = \tau, \TBOOne_{h-1} = \tau'}}{\SBR{h}\cpr{\TBOneO_{h-1} = \tau, \TBOOne_{h-1} = \tau'} } = \frac{\b}{\text{deg}_\tau(o) + 1}.
\end{align*}
Together with Definition \ref{def:admissible-h}(i) this implies
\begin{align*}
\r{h} \cpr{\TBOneO_{h-1} = \tau, \TBOOne_{h-1} = \tau'} & = \SBRM{h}(\tau', \tau) \cpr{ \frac{\b}{\text{deg}_\tau(o)+ 1}}.
\end{align*}
Next, to simplify the numerator of the right-hand side of \eqref{eqn:proof-pu-conditinal-main}, note that the $\T{h}$ random element $\TB$, labeled according to the Ulam-Harris-Neveu
scheme (see Remark \ref{remark:label-unlabel}), satisfies $(\TBOneO_h = \TBdef, \TBOOne_{h-1} = \tau')$  if and only if $(\TBOneO_{h-1} = \tau, \TBOOne_{h-1} = \tau')$ and $\TB$ is isomorphic to the (unlabeled) tree $\TBdef \oplus \tau'$. Hence, using Remark \ref{remark:Eh-conditional}, we see that
\begin{align*}
\r{h} \cpr{ \TBOneO_h = \TBdef, \TBOOne_{h-1} = \tau'} = \cpr{\frac{E_h(\tau', \tau)(\TBdef \oplus \tau')}{\text{deg}_\tau(o) + 1} } \r{h} \cpr{\TBdef \oplus \tau'}.
\end{align*}
Substituting the previous two displays into \eqref{eqn:proof-pu-conditinal-main}, invoking the admissibility of $\r{h}$ and using Lemma \ref{lemma:PI-SBR}, we arrive at \eqref{eqn:PU}.
\end{proof}

\begin{proof}[Proof of Lemma \ref{lemma:finiteness}]
Let $\b > 0$. Let $\rho \in \PUB$ and suppose that $\EDLDRinf < \infty$. By \cite[Lemma 4]{DelAna19}, we have $H(\r{h}) < \infty$ for all $h \in \N$. This shows \ref{item:HRH}. 

Also, for $h \geq 2$,  by \cite[(114), page 76]{DelAna19}, we have
\begin{align*}
- \sum_{\tau \in \T{h}} \r{h}(\tau) \log \cpr{\Rstar{h}(\tau)}  < \infty.
\end{align*}
Then \ref{item:HRHD} follows on observing that  \ref{item:HRH} and the above display imply
\begin{align*}
H(\r{h} \| \Rstar{h}) = -H(\r{h})  - \sum_{\tau \in \T{h}} \r{h}(\tau) \log \cpr{\Rstar{h}(\tau)}  <\infty.
\end{align*}

We now turn to the proof of  \ref{item:HPIRH}. Since $\SBRM{1}$ is a probability measure on the finite set $(\X \times \Y) \times (\X \times \Y)$, it follows that $H(\SBRM{1}) < \infty$. Let $h \geq 2$. By Lemma \ref{lemma:PI-SBR}, we know that $\SBRM{h} = \PIR{h}$.  We first  show that
\begin{align}
- \sum_{\tau, \tau^\prime \in \T{h-1} \times \Y} \PIR{h}(\tau, \tau^\prime) \log \cpr{\PIRstar{h}(\tau, \tau^\prime)}  <\infty. \label{eqn:finiteness-step3}
\end{align}
Let $\tau, \tau^\prime \in \T{h-1} \times \Y$. From the definition of $\PIR{h}$ in \eqref{eqn:PIRh},  Lemma \ref{lemma:PI-SBR}, and the definition of $\Rstar{h}$ in \eqref{eqn:Rstarh}, we have
\begin{align*}
\PIRstar{h}(\tau, \tau^\prime) & = \frac{1}{\b}\sum_{\tilde{\tau}: E_h(\tau, \tau^\prime)(\tilde{\tau}) \geq 1} \Rstar{h}(\tilde{\tau})  E_h(\tau, \tau^\prime)(\tilde{\tau}) \\
& \geq \frac{1}{\b} \r{h-1}(\tau^\prime \oplus \tau_{h-2})  \PU{h-1}(\tau_{h-2}, \tau^\prime_{h-2})(\tau),
\end{align*}
where $\PU{h-1}$ is given by \eqref{eqn:PU}. Using the expression for $\PU{h-1}(\tau_{h-2}, \tau^\prime_{h-2})$ in \eqref{eqn:PU}, and  noting that  $E_{h-1}(\tau^\prime_{h-2}, \tau_{h-2})(\tau \oplus \tau_{h-2}^\prime) \geq 1$, the above display becomes
\begin{align*}
\PIRstar{h}(\tau, \tau^\prime)  & \geq  \frac{1}{\b^2 \PIR{h-1}(\tau_{h-2}, \tau^\prime_{h-2})}\r{h-1}(\tau^\prime \oplus \tau_{h-2}) \r{h-1}(\tau \oplus \tau^\prime_{h-2}). \numberthis \label{eqn:finiteness-step0}
\end{align*}
Also since $\r{h}$ is admissible $\PIR{h}(\tau, \tau') = \PIR{h}(\tau', \tau)$ for all $\tau, \tau' \in \T{h-1}$. Together with \eqref{eqn:finiteness-step0} this  implies that
\begin{align*}
- \sum_{\tau, \tau^\prime \in \T{h-1} \times \Y}  &  \PIR{h}(\tau, \tau^\prime) \log \cpr{\PIRstar{h}(\tau, \tau^\prime)} \\
&  \leq  2 \log \b + \sum_{\tau, \tau^\prime \in \T{h-1} \times \Y} \PIR{h}(\tau, \tau^\prime) \log \cpr{\PIR{h-1}(\tau_{h-2}, \tau^\prime_{h-2})} \\
& \qquad - \sum_{\tau, \tau^\prime \in \T{h-1} \times \Y} \PIR{h}(\tau, \tau^\prime) \log \cpr{\r{h-1}(\tau^\prime \oplus \tau_{h-2})}\\
& \qquad - \sum_{\tau, \tau^\prime \in \T{h-1} \times \Y} \PIR{h}(\tau, \tau^\prime) \log \cpr{ \r{h-1}(\tau \oplus \tau^\prime_{h-2})} \\
&  \leq 2 \log \b - 2 \sum_{\tau, \tau^\prime \in \T{h-1} \times \Y} \PIR{h}(\tau, \tau^\prime) \log \cpr{\r{h-1}(\tau^\prime \oplus \tau_{h-2})}, \numberthis \label{eqn:finiteness-step1}
\end{align*}
where the last inequality follows since $\sum_{\tau, \tau^\prime \in \T{h-1} \times \Y} \PIR{h}(\tau, \tau^\prime) \log \cpr{\PIR{h-1}(\tau_{h-2}, \tau^\prime_{h-2})} = -H(\PIR{h-1}) \leq 0$. 
Observe that for any $\tau_{h-2} \in \T{h-2} \times \Y$ and $\tau^\prime \in \T{h-1} \times \Y$, using \eqref{eqn:PIRh} and the monotone convergence theorem, we have
\begin{align*}
\sum_{\substack{\TBdef \in \T{h-1} \times \Y: \\ \TBdef_{h-2} = \tau_{h-2}}} \PIR{h}(\TBdef, \tau^\prime)  & =  \frac{1}{\b} \Exp_{\r{h}} \sqpr{\sum_{\substack{\TBdef \in \T{h-1} \times \Y: \\ \TBdef_{h-2} = \tau_{h-2}}} E_h(\TBdef, \tau')(\TB)}. \numberthis \label{eqn:finiteness-step2-sum}
\end{align*}
Note that for any $\tau'' \in \T{h}$, by using the definition of $E_h$ from \eqref{eqn:def-Eh} and Remark \ref{remark:Eh} for the first equality below, and Definition \ref{def:oplus} (also see Figure \ref{fig:t-oplus}) for the second equality, we obtain
\begin{align*}
\sum_{\substack{\TBdef \in \T{h-1} \times \Y: \\ \TBdef_{h-2} = \tau_{h-2}}} E_h(\TBdef, \tau')(\tau'') = E_{h}(\tau_{h-2}, \tau')(\tau'') & = \begin{dcases}
E_h(\tau_{h-2}, \tau')(\tau' \oplus \tau_{h-2}) & \text{ if } \tau''_{h-1} = \tau' \oplus \tau_{h-2},\\
0 & \text{ otherwise},
\end{dcases} \\
&  \leq \begin{dcases}
\text{deg}_{\tau' \oplus \tau_{h-2}}(o)  & \text{ if } \tau''_{h-1} = \tau' \oplus \tau_{h-2},\\
0 & \text{ otherwise}.
\end{dcases}
\end{align*}
Together with \eqref{eqn:finiteness-step2-sum} and the property \eqref{eqn:deg-oplus}, this implies that
\begin{align*}
- \sum_{\tau, \tau^\prime \in \T{h-1} \times \Y} & \PIR{h}(\tau, \tau^\prime) \log \cpr{\r{h-1}(\tau^\prime \oplus \tau_{h-2})} \\
& \leq -  \sum_{\substack{\tau_{h-2} \in \T{h-2} \times \Y \\ \tau^\prime \in \T{h-1} \times \Y}} \cpr{\frac{1+\text{deg}_{\tau'}(o)}{\b}}\r{h-1}( \tau^\prime \oplus \tau_{h-2}) \log \cpr{\r{h-1}( \tau^\prime \oplus \tau_{h-2})} \\
& =  -\frac{1}{\b} \sum_{\tilde{\tau} \in \T{h-1}} \text{deg}_{\tilde{\tau}}(o) \r{h-1}(\tilde{\tau}) \log \cpr{\r{h-1}(\tilde{\tau})},  \numberthis \label{eqn:finiteness-step2}
\end{align*}
where the last equality uses Definition \ref{def:oplus}, the property \eqref{eqn:deg-oplus}, and the fact that the set $\{\tau' \oplus \tau'': \tau' \in \T{h-1} \times \Y, \tau'' \in \T{h-2} \times \Y\}$ coincides with the set $\T{h-1}$.
Since $H(\r{h-1}) < \infty$ and $\EDLDRinf < \infty$, it follows from \cite[Lemma 5.4]{BorCap15} that 
\begin{align}
-\sum_{\tilde{\tau} \in \T{h-1}} \text{deg}_{\tilde{\tau}}(o) \r{h-1}(\tilde{\tau}) \log \cpr{\r{h-1}(\tilde{\tau})}  < \infty.
\label{eqn:finiteness-step4}
\end{align}
Combining  \eqref{eqn:finiteness-step1},  \eqref{eqn:finiteness-step2}, and  \eqref{eqn:finiteness-step4}, we arrive at \eqref{eqn:finiteness-step3}. Note that 
\begin{align*}
H(\PIR{h}) & = - \sum_{\tau, \tau^\prime \in \T{h-1} \times \Y} \PIR{h}(\tau, \tau^\prime)\log \cpr{\PIR{h}(\tau, \tau^\prime)}\\
& = - \sum_{\tau, \tau^\prime \in \T{h-1} \times \Y} \PIR{h}(\tau, \tau^\prime) \log \cpr{\PIRstar{h}(\tau, \tau^\prime)} - H(\PIR{h} \| \PIRstar{h}),
\end{align*}
which is finite by the non-negativity of relative entropy and \eqref{eqn:finiteness-step3}. Since $\SBRM{h} = \PIR{h}$ by Lemma \ref{lemma:PI-SBR}, this proves \ref{item:HPIRH}.

Finally, we show \ref{item:HPIRHD}. Let $h \geq 2$. We have
\begin{align*}
H(\PIR{h} \| \PIRstar{h}) & = \sum_{\tau, \tau^\prime \in \T{h-1} \times \Y} \PIR{h}(\tau, \tau^\prime) \log \cpr{\frac{\PIR{h}(\tau, \tau^\prime)}{\PIRstar{h}(\tau, \tau^\prime)}} \\
& = -H(\PIR{h}) - \sum_{\tau, \tau^\prime \in \T{h-1} \times \Y} \PIR{h}(\tau, \tau^\prime) \log \cpr{\PIRstar{h}(\tau, \tau^\prime)},
\end{align*}
which is finite by the non-negativity of Shannon entropy and \eqref{eqn:finiteness-step3}. Since $\SBRM{h} = \PIR{h}$ from Lemma \ref{lemma:PI-SBR}, this proves \ref{item:HPIRHD}.
\end{proof}

\section{Proof of Lemma \ref{lemma:prob-measures-h=1}}
\label{appendix:prob-measure-h=1}
As in Appendix \ref{appendix:lemma-finiteness}, we will refer to quantities defined in 
Section \ref{sec:prelim-pi-eh}, and use  the identity $\pi_\mu = \SBMuM$ from Lemma \ref{lemma:PI-SBR}. Note that Lemma \ref{lemma:PI-SBR} does not depend on any result in Section \ref{sec:alternate-form}, and hence there is no circular reasoning in replacing $\SBMuM$ by $\pi_{\mu}$.

\subsection{Proof of the first assertion of Lemma \ref{lemma:prob-measures-h=1}}
\label{proof:lemma:prob-measures-h=1}
By \eqref{eqn:RP}, showing $\RP_1 \in \P(\T{1})$ is equivalent to showing 
\begin{align}
\Exp_{\eta_1} \sqpr{\prod_{v \in N_o(\TB)} \frac{d\SBRMOne{1}}{d\SBetabarMOne} (\TBOV_0)} & = 1.     \label{eqn:h=1-prod-1}
\end{align}
We first show \eqref{eqn:h=1-prod-1} with $\r{1} =\mu$, $\SBRM{1} = \PIR{1} = \pi_\mu$ and $\SBetabarM = \PIeta$, the latter as justified by Lemma \ref{lemma:PI-SBR}. By Remark \ref{remark:eta-indep}(1), for $i \geq 1$, we have
\begin{align}
\Exp_{\eta_1} \sqpr{\prod_{v \in N_o(\TB)} \frac{d\pi_{\mu}^1}{d\pi_{\eta_1}^1}(\TBOV_{0}) \biggr| \NO = i} & = \prod_{v \in [i]}\Exp_{\eta_1} \sqpr{\frac{d\pi_{\mu}^1}{d\pi_{\eta_1}^1}(\TBOV_{0}) \biggr| \NO = i }. \label{eqn:h=1-prod1-step1}
\end{align}
If $\Law(\TXY) = \eta_1$, by the independence properties in Remark \ref{remark:eta-indep}(1), for any $v \in [i]$, the conditional law of $\TBOV_0$ given $\NO = i \geq 1$ is equal to $\pi_{\eta_1}^1$. Hence, we have
\begin{align*}
\Exp_{\eta_1} \sqpr{\frac{d\pi_{\mu}^1}{d\pi_{\eta_1}^1}(\TBOV_{0}) \biggr| \NO = i } = \Exp_{\pi_{\eta_1}^1} \sqpr{\frac{d\pi_{\mu}^1}{d\pi_{\eta_1}^1} } = 1.
\end{align*}
Together with \eqref{eqn:h=1-prod1-step1} and the convention that a product over the empty set is identically $1$, this implies that 
\begin{align*}
\Exp_{\eta_1} \sqpr{\prod_{v \in N_o(\TB)} \frac{d\pi_{\mu}^1}{d\pi_{\eta_1}^1}(\TBOV_{0}) \biggr| \NO} = 1 \quad \eta_1\text{-a.s.}
\end{align*}
Then \eqref{eqn:h=1-prod-1} follows on taking expectations of both sides, and using $\PIeta = \SBetabarM$ and $\pi_\mu = \SBMuM$ from Lemma \ref{lemma:PI-SBR}.

We now turn to the proof of $\RCP_1 = \MCP \in \P(\T{1})$, which by \eqref{eqn:RCP} is equivalent to showing
\begin{align}
    \Exp_{\eta_1} \sqpr{\prod_{v \in N_o(\TB)} \frac{d\SBRMOneGivenZero{1}}{d\SBetabarMOneGivenZero}(\TBOV_0 | \TBVO_0)} & = 1. \label{eqn:h=1-prod-2}
\end{align}
We show \eqref{eqn:h=1-prod-2} with $\r{1} =\mu$, $\SBRM{1} = \PIR{1} = \pi_\mu$ and $\SBetabarM = \PIeta$. Define $\PIMUT \in \P((\X \times \Y)\times (\X \times \Y))$ by
\begin{align}
\PIMUT(\tau, \tau^\prime) := \pi_{\eta_1}^o(\tau') \pi_{\mu}^{1|o}(\tau | \tau'), \, \tau, \tau' \in \X \times \Y,
\label{eqn:SBRRMT}
\end{align}
where $\PIeta^o$ is the second marginal of $\PIeta$ and $\pi_{\mu}^{1|o}$ is the conditional law of $\TB(1 \setminus  o)$ given $\TB(o \setminus 1)$ when $\Law(\TB) = \mu$. Note that, for any $\tau \in \T{1}$ such that $\text{deg}_\tau(o) \geq 1$, by disintegration, we have
\begin{align*}
\frac{d\pi_{\mu}^{1|o}}{d\pi_{\eta_1}^{1|o}}(\TOV_0 | \TVO_0) = \frac{d\PIMUT}{d\PIeta}(\TOV_0), \TVO_0)\quad \text{ for all }  v \in N_o(\tau).
\end{align*}
Therefore, using the properties of conditional expectations, we have
\begin{align*}
\Exp_{\eta_1} & \sqpr{\prod_{v \in N_o(\TB)} \frac{d\pi_{\mu}^{1|o}}{d\pi_{\eta_1}^{1|o}} (\TBOV_0 | \TBVO_0)} \\
& = \Exp_{\eta_1} \sqpr{\prod_{v \in N_o(\TB)}\frac{d\PIMUT}{d\PIeta}(\TBOV_0 , \TBVO_0) }\\
 &= \Exp_{\eta_1}\sqpr{ \Exp_{\eta_1} \sqpr{ \prod_{v \in N_o(\TB)}\frac{d\PIMUT}{d\PIeta}  (\TBOV_0 , \TBVO_0) \biggr| \TBUO_0, u \in N_o(\TB)}}. \numberthis \label{eqn:h=1-prod2-step1} 
\end{align*}
By Remark \ref{remark:eta-indep}(2),  it follows that
\begin{align*}
\Exp_{\eta_1} &  \sqpr{\prod_{v \in N_o(\TB)}\frac{d\PIMUT}{d\PIeta}(\TBOV_0 , \TBVO_0)  \biggr| \TBUO_0, u \in N_o(\TB)} \\
&= \prod_{v \in N_o(\TB)}  \Exp_{\eta_1} \sqpr{\frac{d\PIMUT}{d\PIeta}(\TBOV_0 , \TBVO_0) \biggr| \TBVO_0}. \numberthis \label{eqn:h=1-prod2-step2}
\end{align*}
Note that the conditional joint law of $( \TBOV_0, \TBVO_0)$ given $\TBVO_0 = \tau'$ when $\Law(\TXY) = \eta_1$ is the same as the conditional joint law of $(\SB, \SB')$ given $\SB' = \tau'$ when $\Law(\SB, \SB') = \PIeta$; this follows from the independence properties in Remark \ref{remark:eta-indep} and the definition of $\PIeta$. Therefore, it follows that
\begin{align*}
\Exp_{\eta_1} \sqpr{\frac{d\PIMUT}{d\PIeta}(\TBOV_0 , \TBVO_0) \biggr| \TBVO_0 = \tau'} = \Exp_{\PIeta} \sqpr{\frac{d\PIMUT}{d\PIeta}(\SB, \SB')  \biggr| \SB' = \tau'}. \numberthis \label{eqn:h=1-prod2-step3}
\end{align*}

We now argue that the conditional expectation on the right-hand side of the above display is equal to $1$ $\PIeta$-a.s. To this end, let $A = \{(\sigma, \sigma') \in (\X \times \Y)^2: \sigma' = \tau'\}$. Then from the definition of $\PIMUT$ in \eqref{eqn:SBRRMT}, we see that $\PIeta(A) = \PIMUT(A)$. Therefore, we have
\begin{align*}
\Exp_{\PIeta}\sqpr{\indf\{A\}} = \Exp_{\PIMUT}\sqpr{\indf\{A\}} = \Exp_{\PIeta}\sqpr{ \frac{d\PIMUT}{d\PIeta}\indf\{A\}},
\end{align*}
which shows that
\begin{align*}
\Exp_{\PIeta} \sqpr{\frac{d\PIMUT}{d\PIeta}(\SB, \SB') \biggr| \SB'} = 1 \quad  \PIeta\text{-a.s.}
\end{align*}
When combined with \eqref{eqn:h=1-prod2-step1}, \eqref{eqn:h=1-prod2-step2}, \eqref{eqn:h=1-prod2-step3}, and the relations $\SBMuM 
= \pi_\mu$ and $\SBetabarM = \pi_\eta$ from Lemma \ref{lemma:PI-SBR}, this proves  \eqref{eqn:h=1-prod-2}.
\qed

\subsection{Proof of the second assertion of Lemma \ref{lemma:prob-measures-h=1}}
\label{proof:lemma:abs-cont-h=1}
We first show that $\mu \ll \MP$. Suppose that $\tau \in \T{1}$ is such that $\MP(\tau) = 0$. From the definition of $\MP$ in \eqref{eqn:RP}, there exists some $v \in N_o(\tau)$ such that 
\begin{align*}
\frac{d\pi_{\mu}^1}{d\pi_{\eta_1}^1}(\TOV_0) = 0
\end{align*} 
If we set $\tau^1 := \TOV_0$ and $\tau^o := \TVO_0$, then by \eqref{eqn:def-Eh} it follows that  $E_1(\tau^1, \tau^o)(\tau) \geq 1$. On the other hand, the last display shows that 
$
\pi_{\mu}^1(\tau^1) = 0.
$
Together with the definition of $\pi_\mu$ in \eqref{eqn:PIRh}, note that
\begin{align*}
0= \pi_\mu^1(\tau^1) \geq \pi_\mu(\tau^1, \tau^o) =  \frac{1}{\b} \Exp_{\mu}\sqpr{E_1(\tau^1, \tau^o)(\TB)} \geq \frac{1}{\b}\mu(\tau). \\
\end{align*}
This implies  $\mu(\tau) = 0$ and proves  $\mu \ll \MP$. 

Next, we show $\mu \ll \MCP$ using an analogous argument. Suppose that $\tau \in \T{1}$ is such that $\MCP(\tau) =0$. From the definition  of $\MCP$ in \eqref{eqn:RCP}, there exists some $v \in N_o(\tau)$ such that 
\begin{align*}
\frac{d\pi_{\mu}^{1|o}}{d\pi_{\eta_1}^{1|o}}(\TOV_0 | \TVO_0) = 0.
\end{align*}
Setting $\tau^1 := \TOV_0$ and $\tau^o := \TVO_0$, \eqref{eqn:def-Eh} implies $E_1(\tau^1, \tau^o)(\tau) \geq 1$. Also,  using the last display, we obtain 
\begin{align*}
\frac{d\pi_\mu}{d\PIeta}(\tau^1, \tau^o)  = \frac{d\pi_{\mu}^o}{d\pi_{\eta_1}^o}(\tau^o) \times \frac{d\pi_{\mu}^{1|o}}{d\pi_{\eta_1}^{1|o}}(\tau^1 | \tau^o) = 0,
\end{align*}
and hence, 
$
\pi_\mu(\tau^1, \tau^o) = 0.
$
Together with the definition of $\pi_\mu$ in \eqref{eqn:PIRh}, this implies
\begin{align*}
0 =\pi_\mu(\tau^1, \tau^o)  = \frac{1}{\b} \Exp_{\mu}\sqpr{E_1(\tau^1, \tau^o)(\TB)} \geq \frac{1}{\b} \mu(\tau),
\end{align*}
which shows that $\mu(\tau) = 0$. Therefore, we conclude that $\mu \ll \MCP$. 
\qed

\section{Proof of Lemma \ref{lemma:preliminary-prob-prop}}
\label{appendix:lemma-Rstarabscont}
We now prove the four properties listed in Lemma \ref{lemma:preliminary-prob-prop}. We shall use the covariance measure $\PIR{h}$ introduced in Definition \ref{def-PIRh} instead of the size-biased marginals $\SBRM{h}$ throughout this appendix.

\subsection{Proof of \ref{lemma:Rstarabscont}}
Suppose that $\Rstar{h}(\tau) = 0$ for some $\tau \in \T{h}$. We proceed to show that then $\r{h}(\tau)  = 0$. 

From  \eqref{eqn:UGW-marginal}, since the $(h-1)$-neighborhood marginals of both $\r{h}$ and $\Rstar{h}$ are the same, if $\r{h-1}(\tau_{h-1}) = 0$, then we have $\r{h}(\tau) = 0$. On the other hand, suppose that $\r{h-1}(\tau_{h-1}) > 0$. By the definition of $\Rstar{h}$ in  \eqref{eqn:Rstarh}, the fact that  $\Rstar{h}(\tau) = 0$ implies there exists $u \in N_o(\tau)$ such that
\begin{align*}
\PU{h-1}(\TOU_{h-2}, \TUO_{h-2})(\TOU_{h-1}) = 0.
\end{align*}
Set $\tau^1 :=\TOU_{h-1}$ and $\tau^o := \TUO_{h-2}$. Using the definition of $\PU{h}(\tau^1, \tau^o)$ in \eqref{eqn:PU}, and noting that $E_{h-1}(\tau^o, \tau^1_{h-2})(\tau^1 \oplus \tau^o) \geq 1$, the last display implies that 
\begin{align}
\r{h-1}(\tau^1 \oplus \tau^o) = 0.
\label{eqn:RhRstarh-step1}
\end{align}

Define the function $f = f_{\tau^1, \tau^o} : \Gdstar[\X ; \Y] \to \R_+$ by
\begin{align*}
f(\TBdef, o_1, o_2) = \begin{cases}
1 & \text{ if } o_1 \in N_{\TBdef}(o_2), \, \TBdef(o_2 \setminus o_1)_{h-1} = \tau^1 \text{ and } \TBdef(o_1 \setminus o_2)_{h-2} = \tau^o,\\
0 & \text{ otherwise},
\end{cases}
\end{align*}
for $(\TB, o_1, o_2) \in \Gdstar[\X; \Y]$.
On the one hand, since $f(\tau, o, u) = 1 $ and $u \in N_o(\tau)$, we have 
\begin{align}
\Exp_{\r{h}} \sqpr{\sum_{v \in N_o(\TB)} f(\TB, o, v)}  \geq \r{h}(\tau).
\label{eqn:RhRstarh-step3}
\end{align}
On the other hand, we have 
\begin{align}
\Exp_{\r{h}} \sqpr{\sum_{v \in N_o(\TB)} f(\TB, v,  o)} & = \Exp_{\r{h}} \sqpr{\sum_{v \in N_o(\TB)}\indf\{\TB(o \setminus v)_{h-1} = \tau^1, \TB(v \setminus o)_{h-2} = \tau^o\} }.
\label{eqn:RhRstarh-step2}
\end{align}
Note that
\begin{align*}
\sum_{v \in N_o(\TB)}\indf\{\TB(o \setminus  v)_{h-1} = \tau^1, \TB(v \setminus o)_{h-2} = \tau^o\} \geq 1 \text{ if and only if } \TB_{h-1} = \tau^1 \oplus \tau^o.
\end{align*}
Substituting this back into \eqref{eqn:RhRstarh-step2}, then  noting that $\tau^1 \oplus \tau^o \in \T{h-1}$ and using \eqref{eqn:RhRstarh-step1},  it follows that
\begin{align*}
\Exp_{\r{h}} \sqpr{\sum_{v \in N_o(\TB)} f(\TB, v,  o)} & = \Exp_{\r{h-1}} \sqpr{\sum_{v \in N_o(\TB)} f(\TB, v,  o)} =  \r{h-1}(\tau^1 \oplus \tau^o) E_{h-1}(\tau^o, \tau^1)(\tau^1 \oplus \tau^o) =0 . \numberthis \label{eqn:RhRstarh-step4}
\end{align*}
Since $\rho$ is unimodular, applying \eqref{eqn:unimod} with $f=f_{\tau^1, \tau^o}$, we conclude that the left-hand sides of \eqref{eqn:RhRstarh-step4} and  \eqref{eqn:RhRstarh-step3} are equal. Hence,  $\r{h} (\tau)= 0$, which proves $\r{h} \ll \Rstar{h}$.

Next, we show $\SBRM{h} \ll \SBRstarM{h}$. Since $\r{h} \ll \Rstar{h}$, from the definition in \eqref{eqn:SBRinf}, we conclude that
\begin{align*}
\SBRM{h} \sim \r{h}^{1,o}(\cdot \mid \NO \geq 1) \ll (\Rstar{h})^{1,o}(\cdot \mid \NO \geq 1) \sim \SBRstarM{h}.
\end{align*}
Here, `$\sim$' denotes the equivalence of measures. This proves \ref{lemma:Rstarabscont} of Lemma \ref{lemma:preliminary-prob-prop}.
\qed

\subsection{Proof of  \ref{lemma:PIRRstarMarginal}}
 \ref{lemma:PIRRstarMarginal} follows immediately from Lemma \ref{lemma:PI-SBR} and Lemma \ref{lemma:pih_pihstar_marginal} below.
\begin{lemma} For any  $\tau \in \T{h-2} \times \Y$ and  $\tau^\prime \in \T{h-1} \times \Y$, we have $\PIR{h}(\tau, \tau^\prime) = \PIRstar{h}(\tau, \tau^\prime)$. Furthermore, we have $\PIR{h}^1 = \PIRstar{h}^1$.
\label{lemma:pih_pihstar_marginal}
\end{lemma}
\begin{proof}    
Let $\tau \in \T{h-2} \times \Y$ and  $\tau^\prime \in \T{h-1} \times \Y$. We first show that $\PIR{h}(\tau, \tau^\prime) = \PIRstar{h}(\tau, \tau^\prime)$. For any $\TBdef  \in \T{h}$, note from \eqref{eqn:def-Eh} that $E_h(\tau, \tau^\prime)(\TBdef)$ is a function of $\TBdef_{h-1}$.  Hence, recalling from \eqref{eqn:UGW-marginal} that the $(h-1)$-neighborhood marginal of  $\Rstar{h}$ is equal to $\r{h-1}$ and using the definition $\PIR{h}$ in \eqref{eqn:PIRh}, we see that
\begin{align*}
\PIR{h}(\tau, \tau^\prime) & = \frac{1}{\b} \Exp_{\r{h}}\sqpr{E_h(\tau, \tau^\prime)(\TXY)} =  \frac{1}{\b} \Exp_{\r{h-1}}\sqpr{E_h(\tau, \tau^\prime)(\TXY)}  =  \frac{1}{\b} \Exp_{\Rstar{h}}\sqpr{E_h(\tau, \tau^\prime)(\TXY)}=  \PIRstar{h}(\tau, \tau^\prime).
\end{align*}
This shows the first assertion.

Next, we show that $\PIRM{h} = \PIRstarM{h}$. Let $\tau \in \T{h-1} \times \Y$. First taking the marginal, then using the admissibility of $\r{h}$, which holds by Remark \ref{remark:admissibility} due to the assumed unimodularity of $\rho$, and substituting the definition of $\PIR{h}$ from \eqref{eqn:PIRh}, this implies
\begin{align}
\PIRM{h}(\tau) & = \sum_{ \tau^\prime \in \T{h-1} \times \Y} \PIR{h}(\tau, \tau^\prime) = \sum_{ \tau^\prime \in \T{h-1} \times \Y} \PIR{h}(\tau^\prime, \tau) = \sum_{ \tau^\prime \in \T{h-1} \times \Y} \frac{1}{\b} \Exp_{\r{h}} \sqpr{E_h(\tau^\prime, \tau)(\TXY)}.
\label{eqn:RPhPIRstar-step1}
\end{align}
On the other hand, since $\Rstar{h}$ is also admissible (which follows from the unimodularity of $\UGWR{h}$ in   Definition \ref{def:unimod-full-extension}  \cite[Lemma 3]{DelAna19} together with the fact that   $\Rstar{h} = (\UGWR{h-1})_{h}$ and Remark \ref{remark:admissibility}),  \eqref{eqn:RPhPIRstar-step1} also holds with $\r{h}$ replaced with $\Rstar{h}$.
Since $E_h(\tau, \tau^\prime)(\cdot) \geq 0$,  by monotone convergence,
\begin{align*}
\Exp_{\r{h}} \sqpr{ \sum_{\tau^\prime \in \T{h-1} \times \Y} E_h(\tau^\prime, \tau)(\TXY) } = \sum_{\tau^\prime \in \T{h-1} \times \Y}\Exp_{\r{h}} \sqpr{E_h(\tau^\prime, \tau)(\TXY)}.
\end{align*}
Thus, to complete the proof, it suffices to show that 
\begin{align*}
\Exp_{\r{h}} \sqpr{ \sum_{\tau^\prime \in \T{h-1} \times \Y} E_h(\tau^\prime, \tau)(\TXY) } = \Exp_{\Rstar{h}} \sqpr{ \sum_{\tau^\prime \in \T{h-1} \times \Y} E_h(\tau^\prime, \tau)(\TXY) }.
\end{align*}
But this is an immediate consequence of the fact that $\sum_{\tau^\prime \in \T{h-1} \times \Y} E_h(\tau^\prime, \tau)(\TBdef)$ depends on $\TBdef$ only through $\TBdef_{h-1}$, and the fact that \eqref{eqn:UGW-marginal} implies $(\Rstar{h})_{h-1} = \r{h-1}$.
\end{proof}

\subsection{Proof of  \ref{lemma:RCPh-prob-measure}}
\label{proof:lemma:RCPh-prob-measure}
Since $\RCP_h \ll \Rstar{h}$ by \eqref{eqn:RCPh}, in order  to show that $\RCP_h$ is a probability measure, it suffices to show that 
\begin{align}
\Exp_{\Rstar{h}}\sqpr{\frac{d\RCP_h}{d\Rstar{h}} (\TXY)} = 1.
\label{eqn:RCPh-abs-cont-step0}
\end{align}
First note that
\begin{align}
\Exp_{\Rstar{h}}\sqpr{\frac{d\RCP_h}{d\Rstar{h}}(\TXY)} & = \Exp_{\Rstar{h}} \sqpr{\Exp_{\Rstar{h}}\sqpr{\frac{d\RCP_h}{d\Rstar{h}}(\TXY) \biggr| \TXY_{h-1} } }.
\label{eqn:RCPh-abs-cont-step5}
\end{align}
The definition of $\Rstar{h}$ in \eqref{eqn:Rstarh} shows  that, if  $\Law(\TB)= \Rstar{h}$,  the random variables \[\pr{(\TBOV_{h-1}, \TBVO_{h-1}), \, v \in N_o(\TB)}\] are conditionally independent given $\TB_{h-1}$. Therefore, using the definition of $\RCP_h$ in \eqref{eqn:RCPh} and the fact that $\SBRM{h} = \PIR{h}$, we conclude that
\begin{align*}
\Exp_{\Rstar{h}}\sqpr{\frac{d\RCP_h}{d\Rstar{h}}(\TXY) \biggr| \TXY_{h-1} } 
& = \prod_{v \in N_o(\TB)} \Exp_{\Rstar{h}} \sqpr{\frac{d\PIR{h}}{d\PIRstar{h}} (\TBOV_{h-1}, \TBVO_{h-1}) \biggr| \TXY_{h-1}}. \numberthis \label{eqn:RCPh-abs-cont-step4}
\end{align*}
Let $\tau := \TBOV_{h-2}$ and $\tau' := \TBVO_{h-2}$. On the one hand, from the definition of $\Rstar{h}$ in \eqref{eqn:Rstarh}, under $\Rstar{h}$, the conditional joint law of $(\TBOV_{h-1}, \TBVO_{h-1})$ given $\TXY_{h-1}$ is $\PU{h-1}(\tau, \tau^\prime)(\TBOV_{h-1})$, which by \eqref{eqn:PU} is equal to 
\begin{align}
\frac{1}{\b \PIR{h-1}(\tau, \tau^\prime)} \r{h-1}(\TBOV_{h-1} \oplus \tau') E_{h-1}(\tau^\prime, \tau)(\TBOV_{h-1} \oplus \tau^\prime),
\label{eqn:RCPh-abs-cont-step1}
\end{align}
where $E_h$ is given by \eqref{eqn:def-Eh}. On the other hand, if $(\SB, \SB')$ has law $\PIRstar{h}$, we have
\begin{align*}
\PIRstar{h}(\SB |(\SB_{h-2}, \SB'_{h-2}) = (\tau,\tau' )) & = \frac{\PIRstar{h}(\SB, \tau')}{\PIRstar{h}(\tau, \tau^\prime)}\\
& = \frac{\PIRstar{h}(\SB, \tau')}{\PIR{h-1}(\tau, \tau^\prime)} \\
& = \frac{\PIRstar{h}(\tau', \SB)}{\PIR{h-1}(\tau, \tau^\prime)}, \numberthis \label{eqn:RCPh-abs-cont-step3}
\end{align*}
where the second equality above follows from the fact that $\tau$ and $\tau'$ are $\T{h-2} \times \Y$ elements and the $(h-1)$-neighborhood marginal of $\Rstar{h}$ equals $\r{h-1}$ by \eqref{eqn:UGW-marginal}, and the third equality holds since $\Rstar{h}$ is admissible. Note that, since $\TBOV_{h-2} = \tau$, we have  $E_h(\tau', \TBOV_{h-1})(\TBOV_{h-1} \oplus \tau') = E_h(\tau', \tau)(\TBOV_{h-1} \oplus \tau')$.  Therefore, from  \ref{lemma:PIRRstarMarginal} of Lemma \ref{lemma:preliminary-prob-prop}, the definition of $\PIR{h}$ in \eqref{eqn:PIRh}, and the fact that $\TOV_{h-1} \oplus \tau' \in \T{h-1}$, we obtain
\begin{align*}
\PIRstar{h}(\tau', \TBOV_{h-1}) & = \PIR{h}(\tau', \TBOV_{h-1}) \\
& = \frac{1}{\b} \r{h}(\TBOV_{h-1} \oplus \tau') E_h(\tau', \TBOV_{h-1})(\TBOV_{h-1} \oplus \tau') \\
& = \frac{1}{\b} \r{h-1}(\TOV_{h-1} \oplus \tau') E_{h-1}(\tau^\prime, \tau)(\TBOV_{h-1} \oplus \tau').
\end{align*}
Together with \eqref{eqn:RCPh-abs-cont-step3} this yields the relation
\begin{align*}
\PIRstar{h}(\SB = \TBOV_{h-1}& |(\SB_{h-2}, \SB'_{h-2}) = (\tau, \tau^\prime)) \\
& = \frac{1}{\b \PIR{h-1}(\tau, \tau^\prime)} \r{h-1}(\TBOV_{h-1} \oplus \tau') E_{h-1}(\tau^\prime, \tau)(\TBOV_{h-1} \oplus \tau'). \numberthis \label{eqn:RCPh-abs-cont-step2}
\end{align*}
Let $\tau''  = \TBVO_{h-1}$. By the definition of $\Rstar{h}$ in \eqref{eqn:Rstarh}, if  $\Law(\TB) = \Rstar{h}$, for any $v \in N_o(\TB)$, the random variables $\TBOV_{h-1} \setminus \TBOV_{h-2}$ and $\TBVO_{h-1} \setminus \TBVO_{h-2}$ are conditionally independent given $\TXY_{h-1}$. Therefore, noting that $\PIRstar{h}$ is the law of $(\TBOOne_{h-1}, \TBOneO_{h-1})$ when $\Law(\TXY) = \SBRstar{h}$ (see Lemma \ref{lemma:PI-SBR}), if $\Law(\SB, \SB') = \PIRstar{h}$, 
the random variable  $\SB$ is conditionally independent of $\SB'$ given $\SB_{h-2} = \tau$ and $\SB'_{h-2} = \tau'$.  Hence, from  \eqref{eqn:RCPh-abs-cont-step1} and \eqref{eqn:RCPh-abs-cont-step2}, it follows that the conditional joint law of $(\TBOV_{h-1}, \TBVO_{h-1})$ given $\TXY_{h-1}$ when $\Law(\TXY) = \Rstar{h}$ equals the conditional  joint law of $(\SB, \SB')$ given $(\SB_{h-2}, \SB') = (\tau, \tau'')$ when $\Law(\SB, \SB') = \PIRstar{h}$. This implies
\begin{align*}
\Exp_{\Rstar{h}} \sqpr{\frac{d\PIR{h}}{d\PIRstar{h}} (\TBOV_{h-1}, \TBVO_{h-1}) \biggr| \TXY_{h-1}} & = \Exp_{\PIRstar{h}} \sqpr{\frac{d\PIR{h}}{d\PIRstar{h}} (\SB, \SB') \biggr| (\SB_{h-2}, \SB') = (\tau, \tau'')}. \numberthis \label{eqn:eqn:RCPh-abs-cont-step7}
\end{align*}

We now show that the above conditional expectation equals $1$ $\PIRstar{h}$-a.s. Let $A = \{(\sigma, \sigma') \in (\T{h-1} \times \Y)^2 : \sigma_{h-2} = \tau, \sigma' = \tau''\}.$ From  \ref{lemma:PIRRstarMarginal} of Lemma \ref{lemma:preliminary-prob-prop}, it follows that $\PIR{h}(A) = \PIRstar{h}(A)$. Therefore,
$
\Exp_{\PIRstar{h}}\sqpr{\indf\{A\}}  = \Exp_{\PIR{h}}\sqpr{\indf\{A\}}  =  \Exp_{\PIRstar{h}}\sqpr{ \frac{d\PIR{h}}{d\PIRstar{h}} \indf\{A\}},
$
which shows that
\begin{align*}
\Exp_{\PIRstar{h}} \sqpr{\frac{d\PIR{h}}{d\PIRstar{h}} (\SB, \SB') \biggr| (\SB_{h-2}, \SB')} = 1 \, \, \PIRstar{h}\text{-a.s.} 
\end{align*}
Hence, using \eqref{eqn:eqn:RCPh-abs-cont-step7} and the above display, \eqref{eqn:RCPh-abs-cont-step4} becomes 
\begin{align*}
\Exp_{\Rstar{h}}\sqpr{\frac{d\RCP_h}{d\Rstar{h}}(\TXY) \biggr| \TXY_{h-1} } =1.
\end{align*}
Together with  \eqref{eqn:RCPh-abs-cont-step5} this implies \eqref{eqn:RCPh-abs-cont-step0}, as desired. 
\qed

\subsection{Proof of  \ref{lemma:RCPh-abs-cont}}
\label{proof:lemma:RCPh-abs-cont}
Suppose that $\RCP_h(\tau) = 0$ for some $\tau \in \T{h}$. From the definition of $\RCP_h$ in \eqref{eqn:RCPh}, there exists $u \in N_o(\tau)$ such that 
\begin{align*}
\frac{d\PIR{h}}{d\PIRstar{h}}(\TOU_{h-1}, \TUO_{h-1}) = 0.
\end{align*}
Let $\tau^1 := \TOU_{h-1}$ and $\tau^o := \TUO_{h-1}$. 
Then $E_h(\tau^1, \tau^o) (\tau) \geq 1$ by \eqref{eqn:PIRh}. Together with the last display, this implies that 
$
0 = \PIR{h}(\tau^1, \tau^o) = \frac{1}{\b} \Exp_{\r{h}} \sqpr{E_h(\tau^1, \tau^o)(\TB)} \geq \frac{1}{\b} \r{h}(\tau).
$
Thus, $\r{h}(\tau) = 0$, and it follows that $\r{h} \ll \RCP_h$.
\qed

\bibliographystyle{plain}
\bibliography{Report.bib}

\begin{thebibliography}{10}

\bibitem{AldLyo07}
David Aldous and Russell Lyons.
\newblock Processes on unimodular random networks.
\newblock {\em Electronic Journal of Probability}, 12:1454--1508, 2007.

\bibitem{AndKonLanPat23}
Luisa Andreis, Wolfgang K{\"o}nig, Heide Langhammer, and Robert~IA Patterson.
\newblock A large-deviations principle for all the components in a sparse
  inhomogeneous random graph.
\newblock {\em Probability Theory and Related Fields}, 186:521--620, 2023.

\bibitem{BacBorSze22}
{\'A}gnes Backhausz, Charles Bordenave, and Bal{\'a}zs Szegedy.
\newblock Typicality and entropy of processes on infinite trees.
\newblock {\em Annales de l'Institut Henri Poincare (B) Probabilites et
  statistiques}, 58(4):1959--1980, 2022.

\bibitem{BalOlietal22}
Rangel Baldasso, Roberto Oliveira, Alan Pereira, and Guilherme Reis.
\newblock Large deviations for marked sparse random graphs with applications to
  interacting diffusions.
\newblock {\em arXiv:2204.08789v1}, 2022.

\bibitem{BenSch01}
Itai Benjamini and Oded Schramm.
\newblock Recurrence of distributional limits of finite planar graphs.
\newblock {\em Electronic Journal of Probability}, 6:1--13, 2001.

\bibitem{Bor16}
Charles Bordenave.
\newblock {Lecture notes on random graphs and probabilistic combinatorial
  optimization}.
\newblock 2016.

\bibitem{BorCap15}
Charles Bordenave and Pietro Caputo.
\newblock Large deviations of empirical neighborhood distribution in sparse
  random graphs.
\newblock {\em Probability Theory and Related Fields}, 163(1):149--222, 2015.

\bibitem{ChaHof21}
Suman Chakraborty, Remco van~der Hofstad, and Frank~den Hollander.
\newblock Sparse random graphs with many triangles.
\newblock {\em arXiv preprint arXiv:2112.06526}, 2021.

\bibitem{ChaDem16}
Sourav Chatterjee and Amir Dembo.
\newblock Nonlinear large deviations.
\newblock {\em Advances in Mathematics}, 299:396--450, 2016.

\bibitem{ChaVar11}
Sourav Chatterjee and S.~R.~S. Varadhan.
\newblock {The large deviation principle for the Erd{\H{o}}s-R{\'e}nyi random
  graph}.
\newblock {\em European Journal of Combinatorics}, 32(7):1000--1017, 2011.

\bibitem{CheRamYas23}
I-Hsun Chen, Kavita Ramanan, and Sarath Yasodharan.
\newblock Large deviations of the empirical component measure of uniform random
  graphs.
\newblock preprint, 2023.

\bibitem{annealed-future-paper}
I-Hsun Chen, Kavita Ramanan, and Sarath Yasodharan.
\newblock Variational formulae for the annealed pressure for statistical
  physics models with discrete and continuous spins.
\newblock forthcoming, 2023.

\bibitem{CocRam23}
Juniper Cocomello and Kavita Ramanan.
\newblock {Exact description of limiting SIR and SEIR dynamics on locally
  tree-like graphs}.
\newblock {\em arXiv preprint arXiv:2309.08829}, 2023.

\bibitem{Csi84}
Imre Csisz{\'a}r.
\newblock {Sanov property, generalized I-projection and a conditional limit
  theorem}.
\newblock {\em The Annals of Probability}, 12(3):768--793, 1984.

\bibitem{DelAna19}
Payam Delgosha and Venkat Anantharam.
\newblock A notion of entropy for stochastic processes on marked rooted graphs.
\newblock {\em arXiv preprint arXiv:1908.00964}, 2019.

\bibitem{DemZei98}
A.~Dembo and O.~Zeitouni.
\newblock {\em Large Deviations Techniques and Applications}.
\newblock Springer-Verlag Berlin Heidelberg, 2 edition, 2010.

\bibitem{DemMon10}
Amir Dembo and Andrea Montanari.
\newblock Gibbs measures and phase transitions on sparse random graphs.
\newblock {\em Braz. J. Probab. Stat.}, 24(2):137--211, 2010.

\bibitem{DemMorShe05}
Amir Dembo, Peter M{\"o}rters, and Scott Sheffield.
\newblock Large deviations of {M}arkov chains indexed by random trees.
\newblock {\em Annales de l'IHP Probabilit{\'e}s et statistiques},
  41(6):971--996, 2005.

\bibitem{DokMor10}
Kwabena Doku-Amponsah and Peter M{\"o}rters.
\newblock Large deviation principles for empirical measures of colored random
  graphs.
\newblock {\em The Annals of Applied Probability}, 20(6):1989--2021, 2010.

\bibitem{Gan22}
Ankan Ganguly.
\newblock {\em Non-Markovian Interacting Particle Systems on Large Sparse
  Graphs: Hydrodynamic Limits and Marginal Characterizations}.
\newblock PhD thesis, Brown University, 2022.

\bibitem{GanRam22-Hydro}
Ankan Ganguly and Kavita Ramanan.
\newblock {Hydrodynamic limits of non-Markovian interacting particle systems on
  sparse graphs}.
\newblock {\em arXiv preprint arXiv:2205.01587}, 2022.

\bibitem{GanHieNam22}
Shirshendu Ganguly, Ella Hiesmayr, and Kyeongsik Nam.
\newblock Upper tail behavior of the number of triangles in random graphs with
  constant average degree.
\newblock {\em to appear in Combinatorica, arXiv:2202.06916}, 2022.

\bibitem{Har63}
Theodore~Edward Harris.
\newblock {\em The Theory of Branching Processes}, volume~6.
\newblock Springer Berlin, 1963.

\bibitem{IbrDok21}
U.~Ibrahim, A.~Lotsi, and K.~Doku-Amponsah.
\newblock Joint large deviation principle for some empirical measures of the
  d-regular random graphs.
\newblock {\em Journal of Discrete Mathematical Sciences and Cryptography},
  24(6):1767--1773, 2021.

\bibitem{LacRamWu23}
Daniel Lacker, Kavita Ramanan, and Ruoyu Wu.
\newblock Local weak convergence for sparse networks of interacting processes.
\newblock {\em The Annals of Applied Probability}, 33(2):843--888, 2023.

\bibitem{LacRamWu20}
Daniel Lacker, Kavita Ramanan, and Ruoyu Wu.
\newblock {Marginal dynamics of interacting diffusions on unimodular
  Galton--Watson trees}.
\newblock {\em Probability Theory and Related Fields}, 187:817--884, 2023.

\bibitem{Leo10}
Christian L{\'e}onard.
\newblock Entropic projections and dominating points.
\newblock {\em ESAIM: Probability and Statistics}, 14:343--381, 2010.

\bibitem{LiaRam24}
Yin-Ting Liao and Kavita Ramanan.
\newblock Large deviations for the interacting particle systems on sparse
  {Erd\H{o}s-R\'enyi} graphs.
\newblock {P}reprint, 2023.

\bibitem{Oco98}
Neil O'Connell.
\newblock Some large deviation results for sparse random graphs.
\newblock {\em Probability Theory and Related Fields}, 110:277--285, 1998.

\bibitem{OliReiSto20}
Roberto~I Oliveira, Guilherme~H Reis, and Lucas~M Stolerman.
\newblock Interacting diffusions on sparse graphs: hydrodynamics from local
  weak limits.
\newblock {\em Electronic Journal of Probability}, 25:1--35, 2020.

\bibitem{Puh05}
Anatolii~A Puhalskii.
\newblock Stochastic processes in random graphs.
\newblock {\em The Annals of Probability}, 33(1):337--412, 2005.

\bibitem{Ram22-ICM}
Kavita Ramanan.
\newblock Interacting stochastic processes on large sparse graphs (invited
  paper).
\newblock In {\em Proceedings of the International Congress of Mathematicians},
  volume~6, pages 4394--4425, 2022.

\bibitem{San57}
I.~N. Sanov.
\newblock On the probability of large deviations of random variables.
\newblock {\em Mat. Sbornik}, 42(84):11--44, 1957.

\end{thebibliography}

\end{document}